\pdfoutput=1
\documentclass[11pt]{amsart}
\usepackage{pinlabel}
\usepackage{amsmath} 
\usepackage{cjhebrew}
\usepackage{graphicx} 
\usepackage{subcaption}
\usepackage{tikz-cd}
\usepackage[all,cmtip]{xy}
\setkeys{Gin}{keepaspectratio}
\usepackage{hyperref}
\usepackage{url}
\usepackage{bm}
\usepackage{mathabx}
\usepackage[left=1.25in,top=1in,right=1.25in,bottom=1in,head=.1in]{geometry}
\usepackage{xcolor}
\usepackage{tikz}
\usepackage{adjustbox}
\usepackage[normalem]{ulem} %
\usepackage{enumerate}
\usepackage{hyperref}

\usepackage{verbatim}
\newcommand{\items}{\begin{itemize}[leftmargin=25pt,rightmargin=15pt]
  \setlength\itemsep{2pt}}

\usepackage{listings}

\usepackage{xcolor}

\lstdefinelanguage{Mathematica}{
    keywords={Module, For, Table, Graphics, Print, If, While, Do, Function, Plot},
    sensitive=true,
    morecomment=[l]{(*},
    morecomment=[s]{(*}{*)},
    morestring=[b]"
}

\usepackage{hyperref}

\newcommand{\stopitems}{\end{itemize}}

\usepackage[textsize=scriptsize]{todonotes}

\usepackage{fancyhdr}

\newtheorem{thm}{Theorem} 

\newtheorem{theorem}{Theorem}[section] 
\newtheorem*{theorem*}{Theorem}
\newtheorem{lemma}[theorem]{Lemma}

\newtheorem*{conjecture*}{Conjecture}
\newtheorem*{question*}{Question}
\newtheorem*{lemma*}{Lemma}
\newtheorem{conj}{Conjecture} %

\newtheorem*{acknowledgement}{Acknowledgement}

\newtheorem{ques}[conj]{Question} 
\newtheorem{proposition}[theorem]{Proposition}
\newtheorem{corollary}[theorem]{Corollary}
\newtheorem*{corollary*}{Corollary}
\theoremstyle{definition}
\newtheorem{definition}[theorem]{Definition}

\newtheorem{remark}[theorem]{Remark}

\newtheorem{example}[theorem]{Example}

\newtheorem*{example*}{Example}
\newtheorem*{remark*}{Remark}
\newtheorem*{remarks*}{Remarks}
\newtheorem*{addenda*}{Addenda}
\newtheorem*{construction*}{Construction}

\newcommand{\RP}{\mathbb{RP}}

\newcommand{\CP}{\mathbb{CP}}

\DeclareMathOperator{\SWF}{SWF}

\DeclareMathOperator{\swapr}{SW_{apr}^{+}}
\DeclareMathOperator{\Bafu}{BF}
\DeclareMathOperator{\FBafu}{FBF}
\DeclareMathOperator{\Det}{det}
\DeclareMathOperator{\ind}{Ind}

\newcommand{\R}{\mathbb R}
\newcommand{\Z}{\mathbb Z}
\newcommand{\bc}{\mathbb C}

\newcommand{\id}{\textup{id}}

\newcommand{\image}{\textup{im}}

\renewcommand{\phi}{\varphi}

\DeclareMathOperator{\SO}{SO}

\DeclareMathOperator{\Pin}{Pin}

\newcommand{\unred}[1]{ \ignorespaces}  

\newcommand{\baf}{\textup{BF}}

\newcommand{\del}{\partial}

\usepackage{stmaryrd}

\title[Constraints on Lefschetz fibrations with four-dimensional fibers]{Constraints on Lefschetz fibrations with four-dimensional fibers from Seiberg--Witten theory}

\author{Hokuto Konno}
\address{Graduate School of Mathematical Sciences, the University of Tokyo, 3-8-1 Komaba, Meguro, Tokyo 153-8914, Japan \\and\\
RIKEN iTHEMS, Wako, Saitama 351-0198, Japan}
\email{konno@ms.u-tokyo.ac.jp}

\author{Jianfeng Lin}
\address{Yau Mathematical Sciences Center, Tsinghua University, Beijing, 100871, China}
\email{linjian5477@mail.tsinghua.edu.cn}

\author{Anubhav Mukherjee}
\address{Department of Mathematics, Princeton University, Princeton, 08540, USA}
\email{anubhavmaths@princeton.edu}

\author{Juan Muñoz-Echániz}
\address{Simons Center for Geometry and Physics, State University of New York, Stony Brook, 11794, USA}
\email{jmunozechaniz@scgp.stonybrook.edu}

\begin{document}
\maketitle
\setlength{\headheight}{12.0pt}
\begin{abstract} 
We establish constraints on the topology of smooth Lefschetz fibrations with $4$-dimensional fibers, by studying the family Bauer-Furuta invariant. To compute this invariant, we analyze the framed bordism class of 1-dimensional Seiberg--Witten moduli spaces using the local index theorem by Bismut--Freed. Using this, we deduce new obstructions to the smooth isotopy to the identity for compositions of Dehn twists on $(-2)$--spheres in closed $4$-manifolds. We obtain several applications: (1) We exhibit the first examples of closed simply-connected symplectic $4$-manifolds admitting Torelli symplectomorphisms which are smoothly non-trivial. In particular, their symplectic Torelli mapping class group is not generated by squared Dehn--Seidel twists on Lagrangian spheres --- providing a negative answer to a question of Donaldson. (2) We provide the first examples of irreducible closed $4$-manifolds (both symplectic and non-symplectic) that admit exotic diffeomorphisms given by Seifert-fibered Dehn twist.
\end{abstract}
\section{Introduction}



The structure of the smooth mapping class group $\pi_0 \mathrm{Diff}(X)$ of a closed oriented smooth $4$-manifold can be probed through diffeomorphisms arising from several generalizations of the classical Dehn twist. One such construction uses a smoothly embedded $2$-sphere $S \subset X$ of self-intersection $S \cdot S = -2$ (a “$(-2)$–sphere”) to define a diffeomorphism $\tau_S \in \pi_0 \mathrm{Diff}(X)$ called the \emph{Dehn twist} on $S$, which acts as the antipodal involution on $S$ and is supported in an arbitrarily small neighborhood of $S$ (see \S \ref{section: Framings of H plus} for its definition). 
Important examples of $(-2)$--spheres $S$ are the \textit{Lagrangian} spheres in symplectic $4$-manifolds $(X, \omega)$, in which case the reflection $\tau_S$ naturally lifts to the symplectic mapping class group as the \textit{Dehn--Seidel twist} $\tau_S \in \pi_0 \mathrm{Symp}(X, \omega)$ (\cite{Arnold1995,seidelknotted,seidel}).

In this article, we establish new obstructions to the smooth isotopy to the identity for compositions of Dehn twists $\tau_{S_1} \cdots \tau_{S_n}$ (Theorem \ref{thm: main}, Corollary \ref{cor: examples}), which can be interpreted as constraints on the topology of Lefschetz fibrations with four-dimensional fibers (Theorem \ref{thm: main generalized}, Corollary \ref{cor: K3}). Our results elucidate the following phenomenon: compositions of Dehn twists in a closed oriented $4$-manifold $X$ may act \textit{trivially on the homology} of $X$ yet still \textit{fail to be smoothly isotopic to the identity}. In some of these examples, the spheres $S_1, \ldots, S_n$ can even be taken to be Lagrangian for a symplectic structure on $X$ yielding, in particular, a negative answer to a well-known question by Donaldson (Question \ref{Donaldson}, Theorem \ref{thm:donaldson}). Our obstructions are not limited to the symplectic case: for instance, we shall exhibit similar phenomena in closed oriented irreducible $4$-manifolds which do not admit symplectic structures (Theorem \ref{thm: nonsymplectic}). 


These results are obtained by analyzing the \textit{framed bordism class} of the family Seiberg--Witten moduli spaces associated to the mapping torus of the diffeomorphism $\tau_{S_1} \cdots \tau_{S_n}$. Namely, we equip these moduli spaces with various stable framings with topological significance and then compare and calculate the corresponding bordism classes.

\subsection{The symplectic Torelli group and Donaldson's question}

This article provides new insights into the structure of symplectic mapping class groups in dimension $4$. For a closed symplectic $4$-manifold $(X, \omega )$, the \textit{symplectic Torelli group} is the subgroup of the symplectic mapping class group acting trivially on the cohomology:
\[
I(X, \omega ):= \mathrm{Ker}\Big( \pi_0 \mathrm{Symp}(X, \omega )\to \mathrm{Aut}H^\ast (X, \mathbb{Z} )\Big).
\]
For all symplectic manifolds of dimension a multiple of $4$, the squared Dehn--Seidel twist $\tau_L^2$ on a Lagrangian sphere is an element of $I(X, \omega )$. The following is a well-known question (\cite{sheridan-smith}):
\begin{ques}[Donaldson]\label{Donaldson}
    
For a closed simply--connected symplectic $4$-manifold $(X, \omega )$, is the symplectic Torelli group $I(X, \omega )$ generated by squared Dehn--Seidel twists on Lagrangian spheres?
\end{ques}

The answer to Question \ref{Donaldson} is known to be affirmative for positive rational surfaces \cite{li-li-wu}, but otherwise remains widely open. If one drops the simple--connectivity assumption on $M$ then examples exist for which both answers are negative \cite{arabadji-baykur,Smirnov23}. If one drops the assumption that $X$ be closed, and considers compact simply-connected symplectic $4$-manifolds with convex boundary, then the authors have also provided counterexamples \cite{KLMME}. 

On the other hand, it is a special fact in $4$ dimensions that $\tau_{L}^2$ is also \textit{smoothly} isotopic to the identity\footnote{$\tau_{L}^2$ is also smoothly trivial in dimension $12$ \cite{RW-Keating}.}, but often non-trivial in $\pi_0 \mathrm{Symp}(X, \omega )$. Thus, the \textit{smoothly trivial symplectic mapping class group}
\[
K(X, \omega ) := \mathrm{Ker}\Big( \pi_0 \mathrm{Symp}(X, \omega ) \to \pi_0 \mathrm{Diff}(X) \Big).
\]
has a rich structure in dimension $4$. Of course, $K(X, \omega ) $ is a subgroup of $ I(X, \omega )$. Besides an affirmative answer for positive rational surfaces \cite{li-li-wu}, the following natural question also remains open:
\begin{ques}\label{ques:KvsI}
    For a closed symplectic $4$-manifold $(X, \omega )$, is $K(X, \omega ) = I (X, \omega )$ ?
\end{ques}

Note that if $X$ is simply-connected and Question~\ref{ques:KvsI} has a negative answer --- that is, $K(X, \omega )$ is a proper subgroup of $I(X, \omega )$ --- then Donaldson’s Question~\ref{Donaldson} does as well, since $\tau_L^2 \in K(X,\omega)$. 
We give a \emph{negative} answer to Donaldson’s Question~\ref{Donaldson} by showing that Question~\ref{ques:KvsI} also has a negative answer:

\begin{thm}
\label{thm:donaldson}
There exist infinitely many simply-connected closed minimal symplectic 4-manifolds $(X, \omega)$ for which $K(X, \omega ) \neq I(X, \omega  )$.
\end{thm}

\begin{remark}Recently, Du--Li \cite{Li-Du2025} have also announced a counterexample to Donaldson's Question \ref{Donaldson} for a one-point blow up of a $K3$ surface. Their symplectomorphisms are the so-called ``elliptic twists'' along embedded tori with self-intersection $-1$, which are trivial in the smooth mapping class group. Thus, their examples showcase a new phenomenon (i.e., $K(X,\omega)$ is not generated by squared Dehn--Seidel twists when $X=K3\#\overline{\mathbb{CP}}^2$) essentially different from the one we study in this article. \end{remark}

As an example of Theorem \ref{thm:donaldson}, consider the $4$-manifold $X = E(4n)_{p,q}$ obtained by performing two logarithmic transformations of orders $p,q$ on the simply-connected minimal elliptic surface $E(4n)$, where $p,q \ge 1$ are odd coprime integers (excluding finitely many exceptional pairs $(p,q)$; see~(\ref{eq: exceptional p q})). Let $M = M(2,3,7) \subset \mathbb{C}^3$ be a (compact) Milnor fiber of the Brieskorn singularity $x^2 +y^3 +z^7 = 0$, equipped with the symplectic form $\omega_0$ given by restriction of the standard one in $\mathbb{C}^3$. In Theorem \ref{thm: elliptic surface examples} we construct a symplectic form $\omega$ on $X$ with certain symplectic embedding $(M,\omega_0 ) \hookrightarrow (X, \omega )$. Let $S_1 , \ldots , S_\mu$ be any distinguished basis of vanishing (Lagrangian) spheres in $(M,\omega_0 )$ (here $\mu = 12$ is the Milnor number). Then, the symplectomorphism of $(X, \omega )$ given by $(\tau_{S_1} \cdots \tau_{S_\mu})^h$ with $h = 42$ acts trivially on the cohomology of $X$, but we prove that it is smoothly non-trivial on $X$; see Corollary~\ref{cor: examples} and Example~\ref{examples_intro} below. That is, $(\tau_{S_1} \cdots \tau_{S_\mu})^h$ belongs in the symplectic Torelli group $ I(X, \omega )$ but not in $K(X, \omega )$.


\subsection{Homologically-trivial products of Dehn twists}\label{subsubsection:productsofref}

Many important classes of four-dimensional diffeomorphisms --- such as monodromies of isolated surface singularities and certain Seifert-fibered Dehn twists --- can be expressed as products of Dehn twists on $(-2)$–spheres (\cite{arnold,seidelgraded,KLMME}). This motivates the development of new techniques for studying such products of Dehn twists directly. The present article is primarily concerned with the following:
\begin{ques}\label{ques:composition}
Given a sequence of smoothly embedded $(-2)$--spheres $S_1,\cdots, S_n$ (not necessarily distinct) in a closed oriented $4$-manifold $X$, when is the product of Dehn twists $\tau_{S_1}\cdots  \tau_{S_{n}}$ smoothly isotopic to the identity?    
\end{ques}


Obviously, a necessary condition to have an affirmative answer to Question \ref{ques:composition} is that the automorphism of the cohomology $H^2 (X, \mathbb{Z} )$ induced by the product of Dehn twists be the identity:
\begin{align}
(\tau_{S_1}\cdots \tau_{S_n} )^\ast = \mathrm{Id}_{H^2 (X , \mathbb{Z} )}. \label{identitycoho}
\end{align}
Recall that each Dehn twist $\tau_{S_i}$ acts non-trivially on \(H^2(X,\mathbb{Z})\) by the Picard--Lefschetz formula $\tau_{S_i}^*\alpha = \alpha + (\alpha\cdot S_i)\,\mathrm{PD}(S_i)$,  which says that $\tau_{S_i}^\ast$ is the reflection on the hyperplane orthogonal to $S_i$ --- in particular, $(\tau_{S_i}^\ast)^2 = \mathrm{Id}$. It is natural to ask whether sufficiently intricate configurations of spheres \(S_1,\dots,S_n\), as measured by their homological intersections, could give rise to a composition \(\tau_{S_1}\cdots\tau_{S_n}\) that is not smoothly isotopic to the identity, while also satisfying \eqref{identitycoho}. To this end, we introduce the following homological invariant (which may be re-phrased in purely lattice-theoretic terms):
\begin{definition}
 Let $S_1 , \ldots , S_n $ be an ordered collection of $(-2)$--spheres in a closed oriented $4$-manifold $X$ satisfying (\ref{identitycoho}). The \textit{spin number} of $S_1, \ldots , S_n$ is the element
\[
\Delta (S_1 , \ldots S_n ) \in \pi_1 SO(b^+ (X) ) \cong \begin{cases}
    \mathbb{Z}/2 & \text{if } b^+ (X) > 2\\
    \mathbb{Z} & \text{if } b^+ (X) = 2\\
    \{ 0 \} & \text{if } b^+ (X) < 2 
\end{cases}
\]
obtained as follows. Let $\mathcal{E}$ denote the space of linear embeddings $e : \mathbb{R}^{b^+ (X)} \hookrightarrow H^2 (X , \mathbb{R} ) $ whose image $\mathrm{Im}(e)$ is a positive subspace (hence of maximal dimension $b^+ (X)$) with respect to the intersection product on $H^2 (X, \mathbb{R} )$. Fixing an embedding $e_0 \in \mathcal{E}$ yields a homotopy-equivalence $SO(b^+ (X) ) \simeq \mathcal{E} $ by reparametrisation of $e_0$. For $i = 1, \ldots , n$ let $e_i = \tau_{S_i}^\ast \circ \cdots \tau_{S_1}^\ast \circ e_0 \in \mathcal{E}$ and choose a path $\gamma_i$ in $\mathcal{E}$ from $e_{i-1}$ to the subspace $\mathcal{E}_i \subset \mathcal{E}$ consisting of embeddings whose image is orthogonal to $S_i$. Then $\Delta (S_1 , \ldots , S_n )$ is the element of $\pi_1 (\mathcal{E} , e_0 )\cong \pi_1 SO(b^+ (X) )$ given by concatenating the following $2n$ paths:
\[
\gamma_1 \, , \,  \tau_{S_1}^\ast \circ \overline{\gamma_1} \, , \, \gamma_2 \,  , \,  \tau_{S_2}^\ast\circ  \overline{\gamma_2 } \, ,\,  \cdots \, ,\,  \gamma_n \, ,\,  \tau_{S_n}^\ast \circ \overline{\gamma_n }
\]
where $\overline{\gamma_i}$ stand for the reversed. See Figure \ref{fig:spinnumber}. It can be shown that $\Delta (S_1 , \ldots , S_n )$ is independent of all auxiliary choices made (Lemma \ref{lemma:indep_choices}). 
\end{definition}

\begin{figure}[htbp]
    \centering    \includegraphics[width=0.55\textwidth]{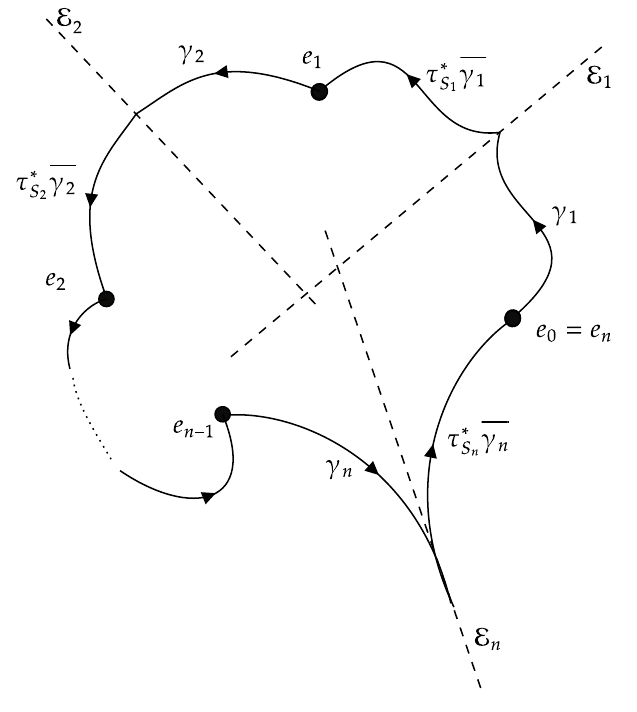}
    \caption{Schematic depiction of the spin number $\Delta (S_1 , \ldots , S_n )$ as a loop in $\mathcal{E}$ based at $e_0$.}
    \label{fig:spinnumber}
\end{figure}

The following result gives conditions on $X$ under which the non-vanishing of the spin number obstructs the smooth isotopy of $\tau_{S_1} \cdots \tau_{S_n}$ to the identity, and is a particular case of Theorem \ref{thm: main generalized} discussed below:
\begin{thm}\label{thm: main}
Let $(X, \mathfrak{s})$ be a closed simply-connected spin-c smooth $4$-manifold. 
Let $S_{1},\cdots, S_n$ be a collection of smoothly embedded $(-2)$--spheres. Assume the following conditions hold:
\begin{itemize}
    \item Both $c_{1}(\mathfrak{s})$ and $\sigma(X)$ are divisible by $32$.
    \item $d(\mathfrak{s}) := \frac{1}{4} (c_1 (\mathfrak{s} )^2 - 2 \chi (X) - 3 \sigma (X) )=0$ and the Seiberg--Witten invariant $\operatorname{SW}(X,\mathfrak{s})$ is odd.
    \item $S_i$ pairs trivially with $c_{1}(\mathfrak{s})$, i.e $ c_1(\mathfrak{s})\cdot S_i =0$.
    \item The composition $\tau_{S_1}\cdots \tau_{S_n}$ is smoothly isotopic to the identity.
\end{itemize}
Then $\Delta(S_1,\cdots, S_{n})=0$ modulo $2$. 
\end{thm}


In \S\ref{subsection:lef}, we interpret the spin number $\Delta(S_1, \ldots, S_n)$ in terms of Lefschetz fibrations on $6$-manifolds and interpret Theorem~\ref{thm: main} through this viewpoint (see Theorem~\ref{thm: main generalized}).\\


We now explain how Theorem~\ref{thm: main} can be applied to produce examples of configurations of spheres satisfying~(\ref{identitycoho}) for which Question~\ref{ques:composition} has a negative answer. 
Let $X$ be a closed oriented $4$-manifold containing a smoothly embedded copy $M \subset X$ of the Milnor fiber of an \emph{exceptional unimodal singularity} (\cite{arnold-normalforms}; see \S\ref{subsection: Configurations of spheres from exceptional unimodal singularities} for background). 
Let $S_1, \ldots, S_\mu \subset M$ be a distinguished basis of vanishing spheres of the singularity. The monodromy of the singularity is given by the composition of Dehn twists $\tau_{S_1} \cdots \tau_{S_\mu}$, which acts on $H^2(M, \mathbb{Z})$ with finite order $h$ (see Table~\ref{table:unimodal} for the corresponding values of $\mu = p+q+r$ and $h$). 
We then consider the ordered configuration of spheres in $X$
\begin{align}
\mathcal{S} := \underbrace{S_1, \ldots, S_\mu}_{h \text{ times}},\label{spheres}
\end{align}
which satisfies~(\ref{identitycoho}). For this configuration, we show that $\Delta(\mathcal{S}) \neq 0 \bmod 2$ (Proposition~\ref{prop: Milnor fiber of unimodal singularity}, Corollary~\ref{corollary:DeltaLefschetz}). Hence, Theorem~\ref{thm: main} yields:

\begin{corollary}\label{cor: examples}
Let $(X, \mathfrak{s})$ be a closed simply-connected oriented spin-c smooth 4-manifold. 
Assume the following conditions hold:
\begin{itemize}
    \item The Milnor fiber $M$ of an exceptional unimodal singularity is smoothly embedded in $X$ so that $c_1(\mathfrak{s})|_M=0$.
    \item Both $c_{1}(\mathfrak{s})$ and $\sigma(X)$ are divisible by $32$.
    \item $d(\mathfrak{s})=0$ and $\operatorname{SW}(X,\mathfrak{s})$ is odd.
\end{itemize}
Then the product of Dehn twists $(\tau_{S_1} \cdots \tau_{S_\mu })^h$, for the configuration of spheres in (\ref{spheres}), is not smoothly isotopic to the identity on $X$ (but satisfies (\ref{identitycoho})). 
\end{corollary}
\begin{example}\label{examples_intro}
Again, for example, consider the $4$-manifold $X = E(4n)_{p,q}$ obtained by performing two logarithmic transformations of orders $p,q$ on the simply connected minimal elliptic surface $E(4n)$, where $p,q \ge 1$ are odd coprime integers (excluding finitely many exceptional pairs $(p,q)$; see~(\ref{eq: exceptional p q})). 
The Milnor fiber $M = M(2,3,7)$ of the Brieskorn singularity $x^2 + y^3 + z^7 = 0$ --- an exceptional unimodal singularity with $\mu = 12$ and $h = 42$ --- admits a smooth embedding in $X = E(4n)_{p,q}$. 
Moreover, there exists a spin$^c$ structure $\mathfrak{s}$ on $X$ satisfying the required conditions, so that Corollary~\ref{cor: examples} implies the smooth nontriviality of $(\tau_{S_1} \cdots \tau_{S_\mu})^h$ in $X$ (see \S\ref{subsection:ellipticsurfaces} for details). 
An explicit picture of a loop representing the spin number $\Delta(\mathcal{S})$ in this case is given in Figure~\ref{figure:237} (and see Appendix~\ref{appendix:figures} for other exceptional unimodal singularities). We also note that similar examples can be constructed from other exceptional unimodal singularities (e.g. $x^2 + y^3 + z^8 = 0$).
Other examples obtained by knot surgery ---rather than logarithmic transformation--- on an elliptic surface are discussed in \S \ref{subsection:knowsurgery} (see also Theorem \ref{thm: nonsymplectic}).
\end{example}

When $X$ is simply-connected, such as in the aforementioned examples, the diffeomorphism $(\tau_{S_1} \cdots \tau_{S_\mu })^h$ from Corollary \ref{cor: examples} is topologically isotopic to the identity by \cite{Quinn86, gabai2023pseudoisotopies, Perron}, thus providing examples of \textit{exotic diffeomorphisms} in an irreducible closed $4$-manifold. The first examples of such were recently given by Baraglia and the first author \cite{baragliakonno2024irreducible}. On the other hand, $(\tau_{S_1} \cdots \tau_{S_\mu })^h$ agrees with the \textit{Seifert-fibered Dehn twist} on the boundary of the Milnor fiber $M \subset X$ \cite[Proposition 2.14]{KLMME}. Recently, exotic diffeomorphisms of $4$-manifolds arising as Dehn twists along Seifert fibered $3$-manifolds have been extensively studied~\cite{konno-mallick-taniguchi,KLMME,KangParkTaniguchi,miyazawa,KLMME2,KangParkTaniguchi2}. Most of these studies concern $4$-manifolds with boundary, and there have been no known examples of exotic Seifert-fibered Dehn twists on irreducible closed $4$-manifolds. 
The above examples thus provide the first instances of exotic Seifert-fibered Dehn twists on irreducible closed $4$-manifolds.

\subsection{Constraints on smooth Lefschetz fibrations in dimension 6}\label{subsection:lef}


Question \ref{ques:composition} can be re-phrased in terms of \textit{smooth Lefschetz fibrations } (\cite{donaldson-lefschetz,donaldson-lefschetz-mapping}). 
For a closed oriented $6$-manifold $E$, a \textit{smooth Lefschetz fibration} on $E$ consists of a smooth map $f : E \to \Sigma$ to a closed connected oriented surface $\Sigma$ with finitely-many critical points $p_1 , \ldots , p_n$, such that: 
\begin{itemize}
    \item $f(p_{i})\neq f(p_{j})$ for all $i\neq j$
    \item there exists oriented local coordinates at $p_{i}$ and $f(x_{i})$, such that the map $f$ is expressed as $(z_{1},\cdots, z_{k})\mapsto z^2_{1}+\cdots+z^2_{k},\ \text{ for }z_{1},\cdots,z_{k}\in \mathbb{C}$. 
\end{itemize}
From a smooth isotopy from the identity to a composition of Dehn twists $\tau_{S_1}\cdots \tau_{S_n}$ one can construct a smooth Lefschetz fibration $f : E \to \Sigma$ over $\Sigma = S^2$ with regular fiber $X$ and distinguished basis of vanishing spheres $S_1 , \ldots , S_n \subset X$; and this procedure can be reversed.

By results of Donaldson \cite{donaldson-lefschetz} and Gompf \cite{gompf-lefschetz,gompfstipsicz} the closed oriented $4$-manifolds that admit a Lefschetz fibration $X\to S^2$ are the symplectic $4$-manifolds up to blowups. On the other hand, it seems hard to characterize which closed oriented $6$-manifolds admit a smooth Lefschetz fibration. 
In this direction, the following result provides constraints on the topology of smooth Lefschetz fibrations on closed oriented $6$-manifolds:

\begin{thm}\label{thm: main generalized} Let $f:E\to \Sigma$ be a smooth Lefschetz fibration over a closed oriented surface, with regular fiber a closed connected oriented $4$-manifold $X$ with $b_{1}(X)=0$ and $b^{+}(X)\equiv 3\mod 4$. Suppose that there exists a spin-c structure $\mathfrak{s}_E$ on the $6$-manifold $E$ such that the Seiberg-Witten invariant $\operatorname{SW}(X,\mathfrak{s}_{E}|_{X})$ is odd and $d(\mathfrak{s}_{E}|_{X}) = 0$. 
Then one has 
\[
\ind(D^{+}(E,\mathfrak{s}_{E}))\equiv  w_{2}(H^{+}(f)) \cdot[\Sigma] \mod 2.
\]
\end{thm}

Here, $\ind(D^{+}(E,\mathfrak{s}_{E})) \in \mathbb{Z}$ denotes the (complex) index of the Dirac operator on the spin-c $6$-manifold $(E, \mathfrak{s} )$, which can be computed by the index formula:
\[
\ind(D^{+}(E,\mathfrak{s}_{E}))=\frac{1}{48}\big(  p_{1}(E)\cdot c_{1}(\mathfrak{s}_E)-c^3_{1}(\mathfrak{s}_E)\big) \cdot [E].
\]
On the other hand, $H^{+}(f)$ denotes the vector bundle over $\Sigma$ constructed as follows. Let $z_1 , \ldots , z_n$ denote the critical values of $f$. Then over $\Sigma \setminus \{ z_1 , \ldots , z_n \}$ there is a vector bundle whose fiber over $z$ is a maximal positive subspace of $H^{2}(f^{-1}(z) ;\mathbb{R})$. Since the monodromy around a critical value is a Dehn twist on a $(-2)$--sphere, then this monodromy is supported in a negative-definite domain in $H^2(X;\R)$. From this, it follows that the previously defined vector bundle has a canonical extension to a vector bundle $H^+ (f) \to \Sigma$ (see \S\ref{section: Framings of H plus} for details).

The spin number $\Delta (S_1 , \ldots , S_n )$ discussed earlier has a simple interpretation in terms of Lefschetz fibrations. Let $f: E \to S^2$ be a smooth Lefschetz fibration of a closed $6$-manifold with regular fiber $X$. Let $S_1 , \ldots , S_n$ be any distinguished basis of vanishing spheres in the fiber $X$. Then the composition $\tau_{S_1} \cdots \tau_{S_n}$ is smoothly isotopic to the identity, so in particular (\ref{identitycoho}) holds. The spin number $\Delta (S_1 , \ldots , S_n ) \in \pi_1 SO(b^+ (X)) \cong \pi_2 BSO(b^+ (X) )$ corresponds to the \textit{classifying map} of the vector bundle $H^+ (f) \to S^2$ (Proposition \ref{proposition:extension}); in particular $\Delta (S_1 , \ldots , S_n )$ agrees mod $2$ with the characteristic class $w_2 ( H^+ (f) )$. In fact, we will see that Theorem \ref{thm: main generalized} is a generalization of Theorem \ref{thm: main}. Theorem \ref{thm: main generalized} is also a generalization of a constraint on smooth fiber bundles with 4-manifold fiber given in \cite[Corollary 1.3]{baraglia-konno} to the setting of Lefschetz fibrations. 

We conclude with another application of Theorem \ref{thm: main generalized}. Holomorphic Lefschetz fibrations are a well-known tool for analysing the topology of complex algebraic varieties. In particular, holomorphic Lefschetz fibrations with $K3$ surface fibers are relevant in the study of Calabi--Yau $3$-folds. In the smooth category, we will establish using Theorem \ref{thm: main generalized} the following:
\begin{corollary}\label{cor: K3}
Let $f:E\to S^2$ be a smooth Lefschetz fibration with fiber $X = K3$ and vanishing cycles $S_1,\cdots, S_n$. Then the $6$-manifold $E$ is spin if and only if $\Delta(S_1,\cdots, S_n) = 0$. 
\end{corollary}

\noindent In particular, it follows that $E$ is Calabi--Yau only if $\Delta(S_1,\cdots, S_n)=0$. On the other hand, we can give examples of smooth Lefschetz fibrations with $K3$ fibers and non-spin total space:

\begin{example}
The Milnor fiber $M = M(2,3,7)$ has a smooth embedding into $K3$ (\cite[\S 8]{gompfstipsicz}), and the authors showed that the composition of Dehn twists $(\tau_{S_1}\cdots \tau_{S_\mu})^h$, for the configuration of spheres in (\ref{spheres}), is smoothly trivial in $K3$ (\cite[Proposition 2.25]{KLMME}). Since $\Delta \neq 0 $ for this configuration, by Corollary \ref{cor: K3} this yields examples of smooth Lefschetz fibrations $f : E\to S^2$ with $K3$ fibers and non-spin total space $E$.
\end{example}
\subsection{Outline and Comments}
We give an outline of the proofs of Theorems \ref{thm: main}-\ref{thm: main generalized}.\\

We sketch the proof of Theorem \ref{thm: main generalized}. Removing tubular neighborhoods of singular fibers of $f:E\to \Sigma$, we obtain a smooth bundle $f_0:E_0\to \Sigma_0$ over the punctured surface $\Sigma_0$, whose restriction to $\partial \Sigma_0 = \partial_1 \Sigma_0 \sqcup \ldots \sqcup \partial_n \Sigma_0$ is isomorphic to $\sqcup^{n}_{i=1}T(\tau_{S_i})$. Here $T(\tau_{S_i})\to S^1$ denotes the mapping torus of the Dehn twist $\tau_{S_i}$, regarded as smooth bundle over $S^1$ with fiber $X$. The theorem is proved by analyzing $\mathcal{M}T(\tau_{S_i})$, the moduli space of the family Seiberg-Witten equations on $T(\tau_{S_i}) \to S^1$. 

Note that $\mathcal{M}T(\tau_{S_i})$ is one-dimensional --- so counting points on the moduli space won't yield interesting invariants. Instead, we study the \textit{framed bordism class} of $\mathcal{M}T(\tau_{S_i})$ --- which defines an element in the framed bordism group $\Omega_{1}^{\mathrm{fr}} \cong \mathbb{Z}/2$ --- for suitable stable framings on this moduli space. A stable framing on $\mathcal{M}T(\tau_{S_i})$ can be specified by a framing $\xi_{d}$ of the bundle $H^{+}(T(\tau_{S_{i}})) \to S^1$ and a framing $\xi_{D}$ on $\Det(\widetilde{D}^{+}(T(\tau_{S_i}))) \to S^1$, the determinant line bundle for the family Dirac operator. Under the Pontryagin--Thom correspondence, this bordism class corresponds to the \textit{family Bauer--Furuta invariant}. Hence we use $\FBafu(T(\tau_{S_i}),\xi_{d},\xi_{D})\in \mathbb{Z}/2$ to denote the framed bordism class of $\mathcal{M}T(\tau_{S_i})$. 

The Dehn twist $\tau_{S_i}$ is supported in a tubular neighborhood $\nu (S_i )$ of $S_i$, and $\nu(S_i)$ is negative definite. From this we can obtain a canonical choice for the framing $\xi_{d}$, denoted by $\xi^{S_i}_{d}$. On the other hand, the spin-c structure $\mathfrak{s}$ is spin when restricted to $\nu(S_i)$, and this provides a canonical choice for $\xi_{D}$, denoted by $\xi^{S_i}_{D}$, using the quaternion-linear structure of the spin Dirac operators. We refer to $\xi_{d}^{S_i} , \xi_{D}^{S_i}$ as the \textit{Dehn twist framings}. Using excision properties of the family Bauer--Furuta invariant and computations in $\operatorname{Pin}(2)$-equivariant stable homotopy theory, we establish the following vanishing result for the Dehn twist framings (Proposition \ref{proposition: FBF vanishing original}):
\begin{align}
\FBafu(T(\tau_{S_i}),\xi^{S_i}_{d},\xi^{S_i}_{D})=0.\label{intro:vanishing1}
\end{align}

On the other hand, the bundles $T(\tau_{S_i})\to S^1$ together bound the bundle $E_0\to \Sigma_0$. We use this to obtain framings $\xi^{\partial_i}_{d}$ and $\xi^{\partial_i}_{D}$ from a choice of corresponding framings for the bundle $E_0\to \Sigma_0$. For the framings $\xi_{d}^{\partial_i }, \xi_{D}^{\partial_i }$ there is another vanishing property (Proposition \ref{prop: FBF of boundary vanishing}): 
\begin{align}
\sum^{n}_{i=1}\FBafu(T(\tau_{S_i}),\xi^{\partial_i}_{d},\xi^{\partial_i}_{D})=0.\label{intro:vanishing2}
\end{align}
In particular, (\ref{intro:vanishing1}-\ref{intro:vanishing2}) imply: 
\begin{equation}\label{eq: Bafu sum equal}
\sum^{n}_{i=1}\FBafu(T(\tau_{S_i}),\xi^{\partial_i}_{d},\xi^{\partial_i}_{D})=\sum^{n}_{i=1}\FBafu(T(\tau_{S_i}),\xi^{S_i}_{d},\xi^{S_i}_{D}).        
    \end{equation}
    
In the remainder of the argument, we analyze the dependence of the Bauer--Furuta invariant $\FBafu(T(\tau_{S_i}),\xi_{d},\xi_{D})$ on the choice of framings $\xi_{d}$ and $\xi_{D}$ and deduce a \textit{change-of-framing formula} for this invariant (Proposition \ref{prop: change of framing formula}). By this formula, and using the condition that $\operatorname{SW}(X,\mathfrak{s})$ is odd, we shall deduce from (\ref{eq: Bafu sum equal}) that 
    \begin{align}
    \sum^{n}_{i=1}(\xi^{\partial_i}_{D}-\xi^{S_i}_{D})\equiv     \sum^{n}_{i=1}(\xi^{\partial_i}_{d}-\xi^{S_i}_{d})\mod 2. \label{framings=}
    \end{align}

By definition, the difference between the Dehn twist framings $\xi_d^{S_i} , \xi_{D}^{S_i}$ and the framings $\xi_d^{\partial_i} , \xi_{D}^{\partial_i}$ is computed in terms of characteristic classes:
\begin{align*}
\sum^{n}_{i=1}(\xi^{\partial_i}_{d}-\xi^{S_i}_{d})& =w_{2}(H^{+}(Tf))\cdot [\Sigma] \in \mathbb{Z}/2 \\
\sum^{n}_{i=1}(\xi^{\partial_i}_{D}-\xi^{S_i}_{D})& =
    c_{1}(\widetilde{D}^{+}(E_0),\xi^{S_{1}}_{D},\cdots,\xi^{S_{n}}_{D})\cdot [\Sigma_0]\in \mathbb{Z}
\end{align*}
Here $c_{1}(\Det(\widetilde{D}^{+}(E_0)),\xi^{S_{1}}_{D},\cdots,\xi^{S_{n}}_{D})\cdot [\Sigma_0]$ denotes the relative Chern number of the determinant line bundle for the family Dirac operator over $E_0\to \Sigma_0$, with respect to the Dehn twist framings on the boundary of $\Sigma_0$. We compute this quantity using the \textit{local index theorem} by Bismut--Freed and we show that it equals $\operatorname{ind}(D^{+}(E,\mathfrak{s}_{E}))$, the numerical index of the $6$-dimensional Dirac operator over $E$ (Proposition \ref{prop: relative Chern class equals index}). From this and (\ref{framings=}), the proof of Theorem \ref{thm: main generalized} will be concluded.
    
The hypothesis of Theorem \ref{thm: main} ensure that the index of the $6$-dimensional Dirac operator $\mathrm{ind}D^+(E,\mathfrak{s}_{E})$ is even; so Theorem \ref{thm: main} will be a consequence of Theorem \ref{thm: main generalized}.\\

The paper is organized as follows. In Section~\ref{section: Framings of H plus}, we study the spin number $\Delta(S_1,\cdots, S_n)$ and interpret it as a difference of framings on $H^{+}$ and in terms of Lefschetz fibrations. We also show that $\Delta$ is non-vanishing on the configuration (\ref{spheres}) coming from vanishing cycles of the exceptional unimodal singularities. In Section~\ref{section: familyBF}, we interpret the family Bauer--Furuta invariant as the framed bordism class of the Seiberg--Witten moduli space. We also show that the family Bauer--Furuta invariant for the Dehn twist vanishes for the Dehn twist framing. The proofs of the main theorems are discussed in Section~\ref{sectionProof of the main theorem}. In Section~\ref{section Examples}, we construct several examples to which our theorems apply. 

\begin{acknowledgement}
We would like to thank Simon Donaldson, John Etnyre, Søren Galatius, Tian-Jun Li and Zoltán Szabó for enlightening discussions.
HK is partially supported by JSPS KAKENHI Grant Numbers 25K00908, 25H00586.  JL is partially supported by NSFC 12271281.
AM is partially supported by NSF grant- DMS 2405270. 

\end{acknowledgement}

\section{The spin number of configurations of spheres}
\label{section: Framings of H plus}

This section discusses in detail the spin number $\Delta (S_1 , \ldots , S_n )$ introduced in \S \ref{subsubsection:productsofref}. We give two interpretation of the spin number: as a difference of two framings (Proposition \ref{prop: difference of xi_d}), and as an invariant of a Lefschetz fibration (Corollary \ref{corollary:DeltaLefschetz}, Proposition \ref{proposition:extension}). We also discuss examples of configurations of spheres with non-vanishing spin number, arising from vanishing cycles of exceptional unimodal singularities (Proposition \ref{prop: Milnor fiber of unimodal singularity}).

\subsection{The Dehn twist on a $(-2)$--sphere.} We begin by recalling the construction of the Dehn twist $\tau_S \in \pi_0 \mathrm{Diff}(X)$ on a $(-2)$--sphere $S\subset X$. Since $S \cdot S = -2$, then after fixing a framing $S \cong S^2$ a tubular neighborhood of $S \subset X$ becomes identified (in a homotopically canonical fashion) with the cotangent bundle $T^*S^2$ with its symplectic orientation. The antipodal map $a$ on $S^2$ induces a diffeomorphism $a^\ast$ of $T^*S^2$ with non-compact support, which can be cut off near the zero section to obtain $\tau_S$: since $a^\ast = \phi_\pi$, where $\phi_t$ is the normalized geodesic flow on $T^*S^2 \setminus S^2$ for the standard round metric, we may set $\tau_S(q,p) = \phi_{\pi \beta(|p|)}(q,p)$, where $\beta (t)$ is a smooth bump function equal to $1$ near $t = 0$. 
(It can be shown that $\tau_S \in \pi_0 \mathrm{Diff}(X)$ is independent of all choices made; in particular of the framing $S \cong S^2$).

\subsection{Construction of framings of $H^+$}

Let $X$ be a compact oriented and connected $4$-manifold. If $\partial X \neq \emptyset$ then we suppose that $\partial X$ is a rational homology $3$-sphere, so that the intersection pairing on $H^2 (X, \mathbb{R})$ is non-degenerate. Throughout we equip $H^+ (X)$ with an orientation.

\subsubsection{Framing the $H^+$--bundle associated to a mapping torus}\label{framingH+mapping}

Let $f \in \pi_0 \mathrm{Diff}(X)$ be a diffeomorphism. Then the \textit{mapping torus} of $f$ is the smooth fiber bundle with fiber $X$ and monodromy $f$ explicitly defined as
\begin{align}
T(f) := \frac{X \times [0,1]}{(x,1) \sim (f(x), 0 ) \, , \, \forall x \in X} .\label{mappingtorus}
\end{align}

The cohomology groups of the fibers of $T(f) \to S^1$ naturally assemble into a local system of real vector spaces over $S^1$ (or flat vector bundle), which we denote $H^2 (f)\to S^1$. By (\ref{mappingtorus}), this is just is the vector bundle over $S^1$ obtained as the mapping torus of the linear isomorphism $(f^\ast)^{-1} : H^2 (X , \mathbb{R} ) \to H^2 (X, \mathbb{R})$,
\begin{align}
H^2 (f) := \frac{H^2 (X,\mathbb{R})\times [0,1]}{(\alpha, 0 ) \sim (f^\ast \alpha , 1) \, , \, \forall \alpha \in H^2 (X, \mathbb{R}) .}\label{mappingtorusH2}
\end{align}

There is a vector subbundle $H^+ (f) \subset H^2 (f)$ whose fibers are maximal positive subspaces in the fibers of $H^2 (f) \to S^1$, and which is defined up to homotopically-canonical isomorphism: indeed, such a subbundle corresponds to a section of the Grassmannian bundle of maximal positive subspaces, which has contractible fibers. 

From now on, we suppose that the vector bundle $H^+ (f) \to S^1$ is orientable, and hence can be given a \textit{framing} (i.e. a global trivialisation). In this section, our goal is to compare framings of  $H^+ (f) \to S^1$ arising in various natural ways. We shall denote by $\mathrm{Fr}(H^+ (f)) $ the set of homotopy-classes of framings of $H^+ (t)$ compatible with the fixed orientation of the fixed subspace $H^+ (X) \subset H^2 (X, \mathbb{R})$. Thus $\mathrm{Fr}(H^+ (f))$ is a torsor over $ [S^1 , SO(b^+(X) )] = H_1 ( SO(b^+ (X)) , \mathbb{Z} )$.

Using (\ref{mappingtorusH2}), a choice of subbundle $H^+ (f) \subset H^2 (f)$ can be understood plainly as a continuous path $H^+ (t)$, $0 \leq t \leq 1$, of maximal positive subspaces in the fixed vector space $H^2 (X, \mathbb{R} )$, such that 
\[
f^\ast H^+ (0) = H^+ (1) .
\]
A framing can similarly be regarded as a path $e (t) : \mathbb{R}^{b^+(X)} \hookrightarrow H^2 (X, \mathbb{R} )$ of linear embeddings such that $\mathrm{Im}\  e (t)= H^+ (t)$ and 
\[
f^\ast ( e(0) ) = e(1).
\]

\subsubsection{Gluing framings}

For each $i = 1, \ldots , n$, let $f_i \in \pi_0 \mathrm{Diff}(X)$ be a diffeomorphism with $H^+ (f_i ) \to S^1$ orientable. Let $f = f_n \circ \cdots \circ f_1$ be their composition. We now discuss how to \textit{glue} given framings of the bundles $H^+ (f_i )\to S^1$ to obtain a framing of $H^+ (f) \to S^1$, a construction that we shall repeatedly use.

Observe that there is a natural vector bundle isomorphism
\begin{align}
\Big(  \bigcup_{i =1}^{n} H^2 (X,\mathbb{R}) \times [0,1 ] \Big)/ \sim \quad \xrightarrow{\cong} H^2 (f) \label{mappingtoruscomposition}
\end{align}
where $\sim$ identifies, for each $i \in \{ 1, \ldots , n \}$, the point $(\alpha,0)$ in the $(i+1)$th component $H^2 (X,\mathbb{R}) \times [0,1]$ with the point $(f_i^\ast (\alpha) , 1 )$ in the $i$th component, with $i$ understood cyclically (i.e. this glues the $1$st component to the $n$th component by $(0 , \alpha ) \sim  (1 , f_n^\ast \alpha )$). Indeed, the isomorphism (\ref{mappingtoruscomposition}) is given by mapping the $1$st component $H^2 (X, \mathbb{R} ) \times [0,1]$ into $H^2 (f)$ (as in (\ref{mappingtorusH2})) by the identity map, and the $i$th component by the map $f_1^\ast \circ \cdots \circ f_{i-1}^\ast$ for $i = 2 , \ldots , n$.

Given framings of each $H^+ (f_i ) \to S^1$, each understood as a path of framed maximal positive subspaces $e_i (t) : \mathbb{R}^{b^+ (X)} \xrightarrow{\cong} H_i ^+ (t) \subset H^2 (X, \mathbb{R} ) $ satisfying $f_i^\ast H_i^+ (0) = H_{i}^+ (1)$ and $f_i^\ast e_i (0) = e_i (1)$, to obtain a framing of the left-hand side in (\ref{mappingtoruscomposition}) by concatenating the paths $H_{i}^+ (t)$, and therefore of $H^+ (f)$, these must satisfy a consistency condition: for $i$ running cyclically from $1$ through $n$, 
\begin{align}
H_{i}^+ (1) = f_i^\ast H^+_{i+1}(0) \quad \text{and} \quad e_i (1) = f_i^\ast e_{i+1}(0) .\label{consistency}
\end{align}

Suppose further that for a fixed framing $e$ of a fixed maximal positive subpace $H^+ (X) \subset H^2 (X,\mathbb{R} )$, the framings above satisfy $H_i^+ (0) = H^+$ and $e_i (0) = e$. Then (\ref{consistency}) is satisfied and these framings can be glued. Thus, (\ref{mappingtoruscomposition}) induces a \textit{based gluing map}
\begin{align}
\mathrm{Fr}_\ast (H^+ (f_1 )) \times \cdots \times \mathrm{Fr}_\ast (H^+ (f_n ) ) \to \mathrm{Fr}_\ast (H^+ (f) )\label{gluingmapbased}
\end{align}
where, for a diffeomorphism $g$ with $H^+ (g)$ orientable, $\mathrm{Fr}_\ast (H^+ (g))$ stands for the set of homotopy-classes of `based' framings of $H^+ (g)$: framings $e(t)$ of $H^+ (g)$ such that $e(0)$ agrees with the fixed framing $e$ of the fixed subspace $H^+$. The set $\mathrm{Fr}_\ast (H^+ (g))$ is now a torsor over the group $\pi_1 SO(b^+(X))$. But since $\pi_1 \mathrm{SO}(b^+ (X)) = H_1 (SO(b^+ (X)) )$, it follows that the natural map $\mathrm{Fr}_\ast (H^+ (g)) \to \mathrm{Fr}(H^+ (g))$ is an isomorphism of torsors. Thus, the based gluing map induces a \textit{gluing map} which is well-defined (independent of the fixed subspace $H^+(X)$ and framing $e$):
\begin{align}
\mathrm{Fr} (H^+ (f_1 )) \times \cdots \times \mathrm{Fr} (H^+ (f_n ) ) \to \mathrm{Fr} (H^+ (f) ).\label{gluingmap}
\end{align}

\subsubsection{The canonical framing}

A canonical framing of $H^+ (f)\to S^1$ arises in the situation when $f \in \pi_0 \mathrm{Diff}(X)$ is \textit{homologically-trivial}, i.e. $f^\ast $ acts on $H^2 (X, \mathbb{R} )$ as the identity:

\begin{definition}\label{definition:canonicalframingd}
    Suppose that $f \in \pi_0  \mathrm{Diff}(X)$ is a homologically-trivial diffeomorphism. Then $H^2 (f,\mathbb{R}) \to S^1$ is canonically identified with the trivial local system $H^2 (X, \mathbb{R}) \times S^1$, and one may then take $H^+ (f)$ to be the product bundle $ H^+(X) \times S^1$. Given a framing $e$ of the vector space $H^+ (X)$ compatible with the given orientation (such a choice is unique up to homotopy) 
    can thus be propagated trivially to a framing of the product bundle $H^+ (f)$. We call this the \textit{canonical framing} of $H^+ (f) \to S^1$, and we denote it by $\xi_{d}^0 \in \mathrm{Fr}(H^+ (f))$.    
    
\end{definition}

\subsubsection{The Dehn twist framing}\label{subsubsection:reflectionframing}
We now consider a diffeomorphism $f \in \pi_0 \mathrm{Diff}(X)$ of the form
\begin{align}
f = \tau_{S_n}\cdots \tau_{S_1} \label{factorisation}
\end{align}
where each $S_i$ is a smoothly embedded spheres in $X$ with self-intersection $S_i \cdot S_i = -2$ (`$-2$--spheres'), which we assume is disjoint from $\partial X$, and $\tau_{S_i} \in \pi_0 \mathrm{Diff}(X)$ denotes the Dehn twist on $S_i$. We shall now describe a framing of $H^+ (f) \to S^1$ arising from the factorisation (\ref{factorisation}).

This will be obtained by first framing each $H^+ (\tau_{S_i} ) \to S^1$, as follows. 
Each $\tau_{S_i}$ is supported in a tubular neighborhood $\nu (S_i ) \subset X$ of the sphere $S_i$, whose boundary is diffeomorphic to $\mathbb{R}P^3$. Thus, there is a canonical decomposition 
\[
H^2 (X,\mathbb{R} ) = H^2 ( \nu (S_i ) , \mathbb{R} ) \oplus H^2 (X \setminus \nu (S_i ) , \mathbb{R} )
\]
which is furthermore preserved by $\tau_{S_i}^\ast$, with 
$\tau_{S_i}^\ast$ acting as the identity on the summand $H^2 (X \setminus \nu (S_i ) , \mathbb{R} )$. It follows that $H^2 ( \tau_{S_i}|_{X \setminus \nu (S_i )}  ) \subset H^2 (\tau_{S_i}  )$ is a trivial local sub-system, and thus $H^+ (\tau_{S_i}|_{X \setminus \nu (S_i )} )$ is identified with a product bundle. On the other hand, because $S_i^2 < 0$ then we obtain a canonical isomorphism $H^+ (\tau_{S_i} )\cong H^+ (\tau_{S_i}|_{X\setminus \nu (S_i )}) $. All combined, this yields a framing of $H^+ (\tau_{S_i} ) \to S^1$, called the \textit{Dehn twist framing} of $H^+ (\tau_{S_i})$, and denote it $\xi_{d}^{S_i} \in \mathrm{Fr}(H^+ (\tau_{S_i} )$ or simply $\xi_{d}^i$. More generally:



\begin{definition}\label{definition:reflectionframingd}
Let $f = \tau_{S_1} \cdots \tau_{S_n} \in \pi_0 \mathrm{Diff}(X)$. The \textit{Dehn twist framing} of $H^+ (f) \to S^1$ associated to the given factorization of $f$ as the product of Dehn twists $\tau_{S_n} \cdots \tau_{S_1}$ is the framing obtained by gluing the Dehn twist framings on the bundles $H^+ (\tau_{S_i}) \to S^1$ using (\ref{gluingmap}). We denote this framing by $\xi_d^{S_1} \cdots \xi_{d}^{S_n}$ or simply $\xi_d^1 \cdots \xi_d^n$.
\end{definition}

\subsection{Comparison of framings}

In what follows, we make the following assumption:
\begin{align}
\text{the diffeomorphism } f:= \tau_{S_1} \cdots \tau_{S_n} \text{ acts trivially on } H^2 (X, \mathbb{Z} ) \label{assumption:trivialaction}
\end{align}
where the $S_i$ are smoothly embedded $(-2)$-spheres in $X$, disjoint from $\partial X$. There are then two framings of $H^+(f) \to S^1$: the canonical framing $\xi_{d}^0$ (Definition \ref{definition:canonicalframingd}) and the Dehn twist framing $\xi_{d}^1  \cdots  \xi_{d}^n$ (Definition \ref{definition:reflectionframingd}). The goal of this subsection is to describe the \textit{difference} of these two framings:
\[
 \xi_{d}^1  \cdots  \xi_{d}^n - \xi_{d}^0 \in [S^1 , SO(b^+ (X))] = \pi_1 SO(b^+(X))   \cong \begin{cases} \{ 1 \} & \text{ if } b^+(X) < 2\\
\mathbb{Z} & \text{if } b^+ (X) = 2\\
\mathbb{Z}/2 & \text{if } b^+(X) > 2 
\end{cases}.
\]

We shall describe the construction of an explicit loop $\Delta (S_1 , \ldots , S_n )\in \pi_1 SO(b^+ (X))$, and then establish that this represents the difference of the two framings above.



\subsubsection{The space of maximal positive embeddings}

Let $\mathcal{E}(H^2 (X,\mathbb{R}))$ be the space of linear embeddings $e : \mathbb{R}^{b^+ (X)} \hookrightarrow H^2 (X,\mathbb{R} )$ such that $\mathrm{Im}(e)$ is a \textit{positive} linear subspace with respect to the intersection form on $H^2 (X,\mathbb{R} )$, topologised as an open subset of the vector space $\mathrm{Hom}(\mathbb{R}^{b^+ (X)} , H^2 (X , \mathbb{R} ) )$. Thus, $\mathrm{Im}e \subset H^2 (X, \mathbb{R} )$ is a maximal positive subspace for the intersection form. Since the space of maximal positive subspaces of $H^2 (X, \mathbb{Z})$ is contractible, it follows that reparametrisation of a fixed embedding $e_0 \in \mathcal{E}(H^2 (X, \mathbb{R}))$ induces a homotopy-equivalence
\begin{align}
SO(b^+ (X) ) \xrightarrow{\simeq} \mathcal{E}(H^2 (X, \mathbb{R})), \quad R \mapsto e_0 \circ R .\label{homotopyeqE}
\end{align}

The Dehn twists $\tau_{S_i}$ act on $H^2 (X,\mathbb{R} )$ by pullback $\tau_{S_i}^\ast$. Recall this action is given by the \textit{Picard--Lefschetz formula}:
\begin{align}
\tau_{S_i}^\ast (\alpha ) = \alpha + \langle \alpha , [S_i]\rangle \cdot \mathrm{PD}([S_i]), \label{PLformula}
\end{align}
and hence $\tau_{S_i}^\ast$ is an involution. Because $\tau_{S_i}^\ast$ preserves the intersection form then it induces an involution $\tau_{S_i}^\vee$ of $\mathcal{E}(X)$ by
\[
\tau_{S_i}^\vee ( e ) := \tau_{S_i}^\ast \circ e : \mathbb{R}^{b^+ (X)}\hookrightarrow H^2 (X, \mathbb{R} ).
\]

By (\ref{PLformula}), the locus of fixed points of the action of $\tau_{S_i}^\ast$ on $H^2 (X, \mathbb{R})$ is the hyperplane $H_i = \{ \alpha \in H^2 (X, \mathbb{R} ) \, | \, \langle \alpha , [S_i]\rangle = 0\}$. The vector space $H_i$ inherits a non-degenerate bilinear form by restriction for which the dimension of a maximal positive subspace is also $ b^+(X) $ (since $S_i \cdot S_i < 0$). Thus the locus of fixed points of $\tau_{S_i}^\vee$ acting on $\mathcal{E}(X)$ is given by 
\[
\mathrm{Fix}(\tau^\vee_{S_i}) = \mathcal{E}(H_i ) \subset \mathcal{E}(H^2 (X, \mathbb{R} )).
\]
In particular, the space $\mathcal{E}(H^2 (X, \mathbb{R} ))$ deformation retracts onto $\mathrm{Fix}(\tau^\vee_{S_i})  $.

\subsubsection{The loop $\Delta(S_1 , \ldots , S_n )$}\label{subsubsection:loop}

The loop $\Delta(S_1 , \ldots , S_n )$ is constructed using the following auxiliary data:
\begin{itemize}
\item A `basepoint' embedding $e_0 \in \mathcal{E}(H^2 (X, \mathbb{R} ))$.
\item For each $i = 1, \ldots , n $, a path $\gamma_i$ in the space $\mathcal{E}(H^2 (X, \mathbb{R} ))$ which starts at the embedding $e_{i-1} := \tau_{i-1}^\vee \cdots \tau_{1}^\vee \cdot e_0 $ and ends at an embedding contained in the locus $\mathrm{Fix}(\tau_{S_i}^\vee )$.
\end{itemize}
For each $i= 1, \ldots, n$, we then obtain a path $\eta_i$ from $e_{i-1}$ to $e_i$ by concatenating $\gamma_i$ with the reversal of the reflected path $\tau_{S_i}^\vee (\gamma_i )$, i.e.
\[
\eta_i := \overline{\tau_{S_i}^\vee (\gamma_i )}\circ \gamma_i.
\]


\begin{definition}\label{definition:loop}
   Let \[
   \Delta(S_1 , \ldots , S_n ) \in \pi_1 \big( \mathcal{E}(H^2 (X, \mathbb{R}), e_0 \big) = \pi_1 SO(b^+ (X)) \quad \text{(cf. (\ref{homotopyeqE}))}
   \]
   be the homotopy-class of the loop in $\mathcal{E}(H^2 (X, \mathbb{R} ))$ based at $e_0$ constructed by concatenating the paths $\eta_i$ for $i = 1, \ldots , n$: 
   \[
   \Delta(S_1 , \ldots , S_n ) = [ \eta_n \circ \cdots \circ \eta_1 ] \, , 
   \]
   which is a loop because $e_n = e_0$ by (\ref{assumption:trivialaction}). See Figure \ref{fig:spinnumber} for a schematic depiction of this construction.
\end{definition}

\begin{lemma}\label{lemma:indep_choices}
The element $\Delta(S_1 , \ldots , S_n )$ is independent of the auxiliary choices made (that is, $e_0 , \gamma_1 , \ldots , \gamma_n$).
\end{lemma}

\begin{proof}
We first address the independence from the auxiliary choices of paths $\gamma_i$. Let $\gamma_i^\prime$ be another choice of path in $\mathcal{E}(H^2 (X, \mathbb{R} ))$ from $e_{i-1}$ to $\mathrm{Fix}(\tau_{S_i}^\vee )$. We show that $\eta_i = \tau_{S_i}^\vee (\overline{\gamma_i} )\circ \gamma_i$ is homotopic to $\eta_{i}^\prime = \overline{\tau_{S_i}^\vee (\gamma_i^\prime )}\circ \gamma_i^\prime$ as a path from $e_{i-1}$ to $e_i$, from which the desired independence follows. It is equivalent to show that the loop $\overline{\eta_{i}^\prime }\circ \eta_i $ in $\mathcal{E}(H^2 (X, \mathbb{R}))$ based at $e_{i-1}$ is null-homotopic. Since $\mathcal{E}(H^2 (X, \mathbb{R})) \simeq SO(b^+ (X))$ then $\pi_1 ( \mathcal{E}(H^2 (X, \mathbb{R})), e_{i-1})$ is abelian; hence it suffices to show that the loop $\overline{\eta_{i}^\prime }\circ \eta_i $ is null-\textit{homologous}. Choose any path $\kappa$ in $\mathrm{Fix}(\tau_{S_i}^\vee)$ from $\gamma_i (1)$ to $\gamma_i^\prime (1)$ (this is possible since $\mathrm{Fix}(\tau_{S_i}^\vee ) \simeq SO(b^+(X))$ is connected), and form the loop $\overline{\gamma_{i}^\prime} \circ \kappa \circ \gamma_i$ based at $e_{i-1}$. Clearly, the homology class given by the difference of the cycles $\overline{\gamma_{i}^\prime} \circ \kappa \circ \gamma_i$ and $\tau_{S_i}^\vee (\overline{\gamma_{i}^\prime} ) \circ \kappa \circ \tau_{S_i}^\vee (\gamma_i )$ is represented by the loop $\overline{\eta_{i}^\prime }\circ \eta_i $. But since $\mathcal{E}(H^2 (X, \mathbb{R} ))$ deformation retracts onto $\mathrm{Fix}(\tau_{S_i}^\vee )$, then the automorphism of the homology of $\mathcal{E}(H^2 (X, \mathbb{R} ))$ induced by $\tau^\vee_{S_i}$ is the identity; and thus the cycles $\overline{\gamma_{i}^\prime } \circ \kappa \circ \gamma_i$ and $ \tau_{S_i}^\vee (\overline{\gamma_{i}^\prime} \circ \kappa \circ \gamma_i) =  \tau_{S_i}^\vee (\overline{\gamma_{i}^\prime} ) \circ \kappa \circ \tau_{S_i}^\vee (\gamma_i )$ are homologous, so their difference is null-homologous. Thus, $\overline{\eta_{i}^\prime }\circ \eta_i $ is null-homologous, as required.

Finally, we discuss the independence of $e_0$. For this, note that fixing $\gamma_1 , \ldots , \gamma_n$ and varying the basepoint $e_0$ in a continuous path $e_0 (t)$, $0 \leq t \leq 1$, induces a corresponding path of loops $\eta_n (t) \circ \cdots \circ \eta_1 (t)$ based at $e_0 (t)$, as follows. 
For $i = 1 , \ldots , n$, let $e_i (t)$ be the path $\tau_{S_i}^\vee \circ \cdots \circ \tau_{S_1}^\vee (e_0 (t))$, which satisfies $e_0 (t) = e_n (t)$. Let $\gamma_i (t)$ be the path of paths from $e_{i-1}(t)$ to $\mathrm{Fix}(\tau_{S_i}^\vee )$ given by first travelling $e_{i-1}(s)$ from $s = t$ to $s = 0$, then $\gamma_i$, and then reparametrising to unit length. The path $\eta_i (t)$ from $e_{i-1} (t)$ to $e_{i} (t)$ is then constructed using $\gamma_i (t)$, as before. It follows from this that $\Delta (S_1 , \ldots , S_n )$ is independent of choices as an element in the first homology of $\mathcal{E}(H^2 (X, \mathbb{R} )) \simeq SO(b^+ (X) )$; and since this space has Abelian fundamental group then also in $\pi_1 SO(b^+ (X) )$, as required.
\end{proof}

\begin{proposition}\label{prop: difference of xi_d}Let $f = \tau_{S_n} \cdots \tau_{S_1}$ be as above. Then the difference between the Dehn twist framing and the canonical framing of $H^+ (f)\to S^1$ agrees with minus $\Delta (S_1 , \ldots, S_n )$: 
\[
-\Delta (S_1 , \ldots , S_n ) = \xi_{d}^{1}\cdots \xi_{d}^n - \xi_d^0 \in \pi_1 SO(b^+ (X)).
\]
\end{proposition}

Before proving Proposition \ref{prop: difference of xi_d} we need to discuss some preliminary results. 

\subsubsection{Constructing an explicit representative of $\Delta (S_1 , \ldots , S_n )$}

We now give a construction of an explicit representative for $\Delta (S_1 , \ldots , S_n ) \in \pi_1 SO(b^+ (X))$, by describing a canonical choice of paths $\gamma_i$ in Definition \ref{definition:loop}. This construction will be used both when calculating $\Delta (S_1 , \ldots , S_n )$ in examples and also in the proof of Proposition \ref{prop: difference of xi_d}.

For each $i =1 , \ldots , n$, consider the following path of linear maps: for $0 \leq t\leq 1$
\begin{align}
\rho^{i}_t : H^2 (X, \mathbb{R} )&  \to H^2 (X , \mathbb{R} ) \label{pathrho}\\
\alpha & \mapsto \alpha + t \langle \alpha , [S_i ]\rangle \mathrm{PD}([S_i]).\nonumber
\end{align}
The path $\rho_{i}^t$ interpolates between the identity $\rho_{0}^i = \mathrm{Id}$ and the Dehn twist $\rho_{1}^0 = \tau_{S_i}^\ast$. At $t =1/2$, we have $\rho^{i}_{1/2} = \Pi_i$, where $\Pi_i$ denotes the orthogonal projection onto the orthogonal complement of $\mathrm{PD}([S_i])$.

\begin{lemma}\label{lemma:rho1} If $e  \in \mathcal{E}(H^2 (X , \mathbb{R} ))$, then $\rho_{t}^i \circ e \in \mathcal{E}(H^2 (X, \mathbb{R} ))$ for all $t \in [0,1]$.
\end{lemma}
\begin{proof}
It suffices to check that if $\alpha \in H^2 ( X , \mathbb{R} )$ has $\alpha \cdot \alpha >0$, then $\rho_{t}^i (\alpha ) \cdot \rho_{t}^i (\alpha ) > 0$ for $t \in [0,1]$. For this we compute: using $S_i\cdot S_i = -2$,
\begin{align*}
\rho_{t}^i (\alpha ) \cdot \rho_{t}^i (\alpha ) &= (\alpha + t \langle \alpha , [S_i ]\rangle \mathrm{PD}([S_i]) ) \cdot (\alpha + t \langle \alpha , [S_i ]\rangle \mathrm{PD}([S_i]) )\\
& = \alpha\cdot \alpha + 2t( 1-t) \langle \alpha ,  \mathrm{PD}([S_i])\rangle ^2 \\
& > 0 .
\end{align*}
\end{proof}

\begin{lemma}\label{lemma:rho2}
For all $t \in [0,1]$, $\rho_{t}^i \circ \tau_{S_i}^\ast = \tau_{S_i}^\ast \circ \rho_{t}^i = \rho_{1-t}^i$.
\end{lemma}
\begin{proof}
Using (\ref{PLformula}) we compute
\begin{align*}
(\tau_{S_i}^\ast \circ \rho_{t}^i ) (\alpha ) &= \tau_{S_i}^\ast (\alpha + t \langle \alpha , [S_i]\rangle \mathrm{PD}([S_i]) ) \\
& = \alpha + \langle \alpha , [S_i]\rangle \mathrm{PD}([S_i]) - t \langle \alpha , [S_i]\rangle \mathrm{PD}([S_i])\\
& = \rho_{1-t}^i (\alpha ).
\end{align*}
The identity $\rho_{t}^i \circ \tau_{S_i}^\ast = \rho_{1-t}^i$ follows similarly.
\end{proof}

Fix now a basepoint $e_0 \in \mathcal{E}(H^2 (X , \mathbb{R} ))$. We now discuss how to make canonical choices for the paths $\gamma_i$, $i =1 , \ldots , n$, in Definition \ref{definition:loop}. Let $e_i = \tau_{S_i}^\ast \cdots \tau_{S_1}^\ast e_0$, as before. From Lemma \ref{lemma:rho1}, we have that $\rho_{t}^i \circ e_{i-1}$, traveled from $t = 0$ to $t =1/2$, is a path in $\mathcal{E}(H^2 (X , \mathbb{R} ))$ connecting $e_{i-1}$ to $\Pi_i (e_{i-1}) \in \mathrm{Fix}(\tau_{S_i}^\vee )$, and we set $\gamma_i$ equal to this path. By Lemma \ref{lemma:rho2}, the reflected path $\overline{\tau_{S_i}^\vee (\gamma_i )}$ is just $\rho_{t}^i \circ e_{i-1}$ traveled from $t = 1/2$ to $t =1$. Thus, we have shown:

\begin{proposition}\label{proposition:representative}
    $\Delta (S_1 , \ldots , S_n ) \in \pi_1 \mathcal{E}(H^2(X, \mathbb{R} ),e_0 )$ is the loop obtained by concatenating the following $n$ paths:
    \[
\rho_{t}^{i}\circ e_{i-1} \quad , \quad 0 \leq t \leq 1 \, , \, i = 1 , \ldots , n 
    \]
    where $e_{i-1} = \tau_{S_{i-1}}^\ast \cdots \tau_{S_1}^\ast e_0$.
\end{proposition}

\subsubsection{Proof of Proposition \ref{prop: difference of xi_d}}

The Dehn twist framing $\xi_{d}^i$ of $H^+ (\tau_{S_i} )$ is represented by the constant path of framings based at a framing $e_{i}^{\mathrm{fix}} \in \mathcal{E}(H^2 (X , \mathbb{R} ))$ whose image is a maximal positive subspace contained in $\mathrm{Fix}(\tau_{S_i}^\ast ) = \langle \mathrm{PD}([S_i ])\rangle^{\perp}$. In order to describe the glued framing (cf. (\ref{gluingmap})), we first want to homotope this framing of $H^+ (\tau_{S_i} )$ to a based framing. Fix a maximal positive subspace $H^+ (X ) \subset H^2 (X , \mathbb{R} )$ with a framing $e_0 $ (compatible with the given orientation of $H^+ (X)$). Then the path $\rho_{t}^i ( e_0 )$, $0 \leq t \leq 1$, represents a framing based at $e_0$. This framing is homotopic (through framings) to the Dehn twist framing. Indeed, choosing a path $\beta_i (s)$ in $\mathcal{E}(H^2 (X , \mathbb{R}))$ from $e_i^{\mathrm{fix}}$ to $e_0$, we obtain a homotopy of framings $\rho_{t}^i ( \beta_i (s))$ (by Lemma \ref{lemma:rho1}) from the Dehn twist framing to the new based framing. In summary, we have shown that the path $\rho_{t}^i ( e_0 )$ in $\mathcal{E}(H^2 (X , \mathbb{R}))$ based at $e_0$ represents the Dehn twist framing $\xi_{d}^i$ of $H^+ (\tau_{S_i } )$.

Gluing the based framings $\rho_{t}^i ( e_0 )$ using the based gluing map (\ref{gluingmapbased}) we obtain a framing of $H^+ (f)$ where $f = \tau_{S_n} \circ \cdots \tau_{S_1}$. By the identification (\ref{mappingtoruscomposition}), this framing of $H^+ (f)$ is represented by the loop in $\mathcal{E}(H^2 (X , \mathbb{R} ))$ based at $e_0$ obtained by concatenating the following $n$ paths:
\begin{align}
\tau_{S_1}^\ast \cdots \tau_{S_{i-1}}^\ast \rho_{t}^i (e_0  ) \quad , \quad 0 \leq t \leq 1 \, , \, i = 1 , \ldots  , n . \label{paths}
\end{align}

The following steps show that the loop obtained by concatenating the paths in  (\ref{paths}) represents $- \Delta (S_1 , \ldots , S_n )$.\\ 

\textit{Connecting the paths $i =1$.} The path $i =1$ in (\ref{paths}) agrees with the path $i =1$ in Proposition \ref{proposition:representative}. This is clear. \\

\textit{Connecting the paths $i =2$.} The path $i =2$ from (\ref{paths}) is $\tau_{S_1}^\ast \rho_{t}^2 (e_0 )$, which connects $\tau_{S_1}^\ast e_0 = e_1$ (at $t = 0$) to $\tau_{S_1}^\ast \tau_{S_2}^\ast e_0$ (at $t =1$). From Lemma \ref{lemma:rho2}, we have the following homotopy of this path:
\begin{align}
\rho_{1-s}^1 \rho_{t}^2 \rho_{s}^1 (e_0 ) \quad , \quad 0 \leq s \leq 1 .\label{homotopy:i=2}
\end{align}
At $s = 0$, this agrees with $\tau_{S_1}^\ast \rho_{t}^2 (e_0 )$. At $s =1$, we have the path $\rho_{t}^2 \tau_{S_1}^\ast (e_0 ) = \rho_{t}^2 (e_1 )$, which is the path $i =2$ from Proposition \ref{proposition:representative}. 

However, note that both starting point ($t = 0$) and the ending point ($t =1$) of each path in the homotopy (\ref{homotopy:i=2}) do \textit{not} remain constant. Still, we can modify the homotopy (\ref{homotopy:i=2}) so that the starting point remains constant at $\tau_{S_1}^\ast (e_0 )$. For this, note that the starting point is $\rho_{1-s}^1 \rho_{s}^1 (e_0 ) = \tau_{S_1}^\ast \rho_{s}^1 \rho_{s}^1 (e_0 )$ (by Lemma \ref{lemma:rho2}), which describes a loop based at $\tau_{S_1}^\ast (e_0 )$ as $s$ goes from $0$ to $1$. We have the following identity, which follows easily from (\ref{PLformula}) and $S_1\cdot S_1 = -2$,
\[
\rho_{s}^1 \rho_{s}^1  = \rho_{2s (1-s)}^1
\]
and hence the loop $\tau_{S_1}^\ast \rho_{s}^1  \rho_{s}^1 (e_0 )$ can be homotoped to the constant loop at $\tau_{S_1}^\ast e_0 = e_1$ through the following based homotopy (using also Lemma \ref{lemma:rho1}):
\begin{align}
\tau_{S_1}^\ast \rho_{2sr (1-s)}^1 \quad , \quad 0 \leq r \leq 1 .\label{homotopystarti=2}
\end{align}

In conclusion, by applying the homotopy (\ref{homotopy:i=2}) and modifying its starting point using the homotopy (\ref{homotopystarti=2}), we have homotoped the path $i =2$ in (\ref{paths}) to the path $i =2$ in Proposition \ref{proposition:representative} through paths which remain fixed at the starting point $\tau_{S_1}^\ast e_0 = e_1$. Gluing this homotopy with the paths from the step $i =1$ above yields a homotopy $H_{s}^{2}(t)$ of paths from the concatenation of the paths $i =1,2$ in (\ref{paths}) to the concatenation of the paths $i =1,2$ in Proposition \ref{proposition:representative}, which stays constant at the starting point $e_0$ but possibly varies the endpoint. \\

\textit{Connecting the paths $i =3$.} The path $i =3$ from (\ref{paths}) is now $\tau_{S_1}^\ast \tau_{S_2}^\ast \rho_{t}^3 (e_0 )$. We consider the following two homotopies:
\begin{align}
\rho_{1-s}^1 \tau_{S_2}^\ast \rho_{t}^3 \rho_{s}^1 (e_0 ) \quad , \quad 0 \leq s \leq 1 \label{homotopyi=3,1}\\
\rho_{1-s}^2\rho_{t}^3 \rho_{s}^2 \tau_{S_1}^\ast (e_0 )\quad , \quad 0 \leq s \leq 1 .\label{homotopyi=3,2}
\end{align}
The homotopy (\ref{homotopyi=3,1}) interpolates from the path $i =3$ in (\ref{paths}) (i.e. $\tau_{S_1}^\ast \tau_{S_2}^\ast \rho_{t}^3 (e_0 )$) to the path $\tau_{2}^\ast \rho_{t}^3 \tau_{S_1}^\ast (e_0 )$, and the homotopy (\ref{homotopyi=3,2}) interpolates from the latter path to the path $i =3$ in Proposition \ref{proposition:representative} (i.e. $\rho_{t}^3 (e_2 )$ where $e_2 = \tau_{S_2}^\ast \tau_{S_1}^\ast e_0 $).

The starting points of the paths in the homotopy (\ref{homotopyi=3,1}) are given by
\[
\rho_{1-s}^1 \tau_{S_2}^\ast \rho_{s}^1 (e_0 )
\]
which coincide with the ending points of the homotopy (\ref{homotopy:i=2}) from the previous step.

We would like the starting point of the paths in the homotopy (\ref{homotopyi=3,2}) to remain fixed at $e_2 = \tau_{S_2}^\ast \tau_{S_1}^\ast e_0 $. However, this is not the case. But the starting points are given by the loop
\[
\rho_{1-s}^2 \rho_{s}^2 \tau_{S_1}^\ast (e_0 ) = \tau_{S_2}^\ast \rho_{s}^2 \rho_{s}^2 \tau_{S_1}^\ast (e_0 )
\]
which can be deformed to the constant path at $e_2$ through a based homotopy constructed in a similar way as in the previous step.

Thus, we can glue the homotopy $H^{2}_s (t)$ constructed in the previous step with the homotopies (\ref{homotopyi=3,1}-\ref{homotopyi=3,2}) after making the starting point in the homotopy (\ref{homotopyi=3,2}) constant, as explained above. The new homotopy thus obtained, denoted $H_{s}^3 (t)$ interpolates from the concatenation of the paths $i =1,2,3$ in (\ref{paths}) to the concatenation of the paths $i =1,2,3$ in Proposition \ref{proposition:representative}. As before, the starting point of the paths in the homotopy $H_{s}^3 (t)$ remain fixed at $e_0$ but the endpoints are varying possibly.\\

Clearly, the procedure described in the previous steps can be carried on for $i =1 , 2 , 3, \ldots , n$. It results in a homotopy $H_{s}^n (t)$ from the loop obtained by the concatenation of the paths (\ref{paths}) to the loop from Proposition \ref{proposition:representative}. This homotopy is through paths with fixed starting point, but the endpoint possibly varying. The endpoints form a loop $H_{s}^n (1)$ based at $e_0$, and we now describe this loop:\\

\textit{Claim: the loop $H_{s}^n (1)$ represents $2\Delta (S_1 , \ldots , S_n ) \in \pi_1 \mathcal{E}(H^2 (X, \mathbb{R} ))$.}\\

Explicitly, the loop $H_{s}^n (1)$ is given by the concatenation of the following $n$ paths:
\begin{align*}
\rho_{1-s}^i \tau_{S_{i+1}}^\ast \cdots \tau_{S_n}^\ast  \rho_{s}^i \tau_{S_{i-1}}^\ast \cdots \tau_{S_1}^\ast e_0 \quad , \quad i = 1 , \ldots , n .
\end{align*}

Using $\rho_{1-s}^i = \rho_{s}^i \circ \tau_{S_i}^\ast$ (Lemma \ref{lemma:rho2}) and the identity $\tau_{S_{i}}^\ast \cdots \tau_{S_n}^\ast  = \tau_{S_{i-1}}^\ast \cdots \tau_{S_1}^\ast$ (coming from the fact that $\tau_{S_n}^\ast \cdots \tau_{S_1}^\ast = \mathrm{Id}$ and $(\tau_{S_i}^\ast)^2 =\mathrm{Id}$), we can rewrite the above $n$ paths as
\begin{align*}
(\rho_{s}^i \tau_{S_{i-1}}^\ast \cdots \tau_{S_1}^\ast )^2 e_0 \quad , \quad i = 1 , \ldots , n .
\end{align*}

The concatenation of the paths of linear maps $\rho_{s}^i \tau_{S_{i-1}}^\ast \cdots \tau_{S_1}^\ast $ for $i = 1 , \ldots , n$ gives a loop $L(s)$ of linear maps based at the identity. The loop $H^n_s(1)$ from above is $L(s)^2 e_0$. In the same vein of the proof that the fundamental group of a topological group is abelian, one can show that the loop $L(s)^2 e_0$ is based homotopic to the twice concatenation of $L(s)e_0$. Namely, one exhibits the following based homotopy between the two:
\[
K(s,r) = \begin{cases} L\big( \frac{2s}{1+r}\big) L(rs) &  s\in [0 , \frac{t+1}{2}]\\
L\big( (s-\frac{r+1}{2} )(2-r) + sr \big) L\big(\frac{2s}{1+r} + (\frac{1+r}{2} -s )(2 - 2r ) \big) & s\in [\frac{r+1}{2} , 1]
\end{cases}.
\]
On the other hand, $L(s)e_0$ is the loop representing $\Delta (S_1 , \ldots , S_n )$ given in Proposition \ref{proposition:representative}. This concludes the proof of the Claim.\\

From this the proof of Proposition \ref{prop: difference of xi_d} is completed, as putting all together shows:
\[
\xi_{d}^{1}\cdots \xi_{d}^n - \xi_d^0 = -2\Delta (S_1 , \ldots , S_n ) + \Delta(S_1 ,\ldots , S_n ) = -\Delta(S_1 , \ldots , S_n ). \qed
\]

\subsection{Calculations of $\Delta(S_1 , \ldots , S_n )$}

In this subsection we discuss how to compute the element $\Delta (S_1 , \ldots , S_n )$ in concrete examples.\\

First, we make some preliminary remarks. If we make a fixed choice of maximal positive subspace $H^+ (X) \subset H^2 (X, \mathbb{R} )$ with an orientation. Then there is a canonical orthogonal projection map
\[
\Pi : H^2 (X, \mathbb{R} ) \to H^+ (X).
\]
Orthogonal projection induces a well-defined map
\begin{align}
\Pi : \mathcal{E}(H^2 (X,\mathbb{R} )) \to \mathcal{E} ( H^+ (X) ) \quad , \quad e \mapsto \Pi \circ e . \label{projH+}
\end{align}
Here, $\mathcal{E} (H^+ (X))$ is simply the set of orientation-preserving linear isomorphisms $\mathbb{R}^{b^+ (X)} \xrightarrow{\cong} H^+ (X)$, and is therefore homeomorphic to the group $GL_{+}(b^+(X) , \mathbb{R} )$. Furthermore, (\ref{projH+}) is a homotopy-equivalence, and in what follows we describe how to compute $\Pi (\Delta(S_1 , \ldots , S_n )) \in \pi_1 \mathcal{E}(H^+ (X))$ when $b^+ (X) =2$. 

\subsubsection{Computing $\Delta (S_1 , \ldots , S_n )$ in the case $b^+  =2$}\label{subsubsection:algorithm}

The following describes an `algorithm' for calculating $\Delta (S_1 , \ldots , S_n )$ in a special situation. The input of this algorithm is the following:
\begin{itemize}
\item A compact smooth $4$-manifold $M$ with $b^+ (M) =2$ and with rational homology sphere boundary (possibly empty)
\item  A finite collection of smoothly embedded $(-2)$--spheres  $S_1 , \ldots , S_n $  in $M$ (disjoint from $\partial M$)
\item A maximal positive subspace $H^+ (M) \subset H^2 (M, \mathbb{R} )$ together with a framing $e_0 : \mathbb{R}^{b^+(M)} \xrightarrow{\cong} H^+ (M)$. Since $b^+ (M) =2$, then this framing corresponds to a choice of two linearly independent vectors $a ,b \in H^+ (M)$.
\end{itemize}

The framing $a,b$ of $H^+ (M)$ identifies $\mathcal{E} (H^+ (M)) = GL_{+}(2,\mathbb{R} )$. Let $p : GL_+ (2,\mathbb{R} ) \to \mathbb{R}^2 \setminus 0$ be the map which projects a matrix to its first column. Then, with $\Pi$ as defined in (\ref{projH+}), we obtain a homotopy-equivalence
\[
\pi := p \circ \Pi : \mathcal{E}(H^2 (M, \mathbb{R} )) \xrightarrow{\simeq} \mathbb{R}^2 \setminus 0 .
\]

\begin{lemma}\label{lemma:computing}
For $i =0 , \ldots , n$, let $v_i  = \pi ( \tau_{S_i}^\ast \cdots \tau_{S_1}^\ast a ) \in \mathbb{R}^2 \setminus 0$, and note that these satisfy $v_0 = v_n = (1,0)$. Let $\eta$ be the loop in $\mathbb{R}^2 \setminus 0$ based at $v_0 = (1,0)$ obtained by concatenating the straight line segment from $v_{i-1}$ to $v_i$ for $i =1 , \ldots , n$. Then $\eta$ represents the element 
\[
\pi_\ast \Delta (S_1 , \ldots , S_n ) \in \pi_1 ( \mathbb{R}^2 \setminus 0 , v_0 ) = \mathbb{Z}.
\]
\end{lemma}
\begin{proof}
In the representative of $\Delta(S_1 , \ldots , S_n )$ constructed in Proposition \ref{proposition:representative}, it is clear that the $i$th path contained in it projects to a straight line segment in $\mathbb{R}^2$ under $\pi$.
\end{proof}

Since $\tau_{S_i}^\ast$ is given by the Picard--Lefschetz formula (\ref{PLformula}), Lemma \ref{lemma:computing} gives an algorithm for computing $\Delta (S_1 , \ldots , S_n ) \in \pi_1 \mathcal{E}(H^2 (M, \mathbb{R}))$ which can easily be implemented with a computer. 

In our case of interest we will have a closed $4$-manifold $X$ with $b^+ (X) >2$, but the spheres $S_1 , \ldots , S_n$ will all be contained in the interior of a compact $4$-dimensional submanifold $M \subset X$ with rational homology sphere boundary and $b^+ (M) = 2$. By pushforward we obtain an embedding $H^2 (M,\mathbb{R}) \subset H^2 (X,\mathbb{R})$, and an associated orthogonal decomposition $H^2 (X, \mathbb{R} ) = H^2 (M , \mathbb{R} ) \oplus H^2 (M , \mathbb{R} )^{\perp}$. If $H^+ (M) \subset H^2 (M, \mathbb{R} )$ and $V \subset H^2 (M,\mathbb{R})^{\perp}$ are maximal positive subspaces then so is their sum in $H^2 (X, \mathbb{R} )$. Thus, a choice of framed subspace $V$ yields a stabilization map
\[
s_V :  \mathcal{E}(H^2 (M,\mathbb{R} )) \to  \mathcal{E} (H^2 (X, \mathbb{R})) \quad , \quad e \mapsto e \oplus V .
\] 
The induced map
\[
(s_{V})_\ast : \pi_1 \mathcal{E} ( H^2 (M , \mathbb{R} )) \to \pi_1 \mathcal{E} (H^2 (X, \mathbb{R}))
\]
is the unique surjection between these two groups. Hence, we can use Lemma \ref{lemma:computing} to compute $\Delta (S_1 , \ldots , S_n )\in \pi_1 \mathcal{E}(H^2 (X, \mathbb{R} )) $ in this situation also: the element $\Delta (S_1 , \ldots , S_n ) $ is non-trivial in $\pi_1 \mathcal{E}(H^2 ( X, \mathbb{R} )) \cong \mathbb{Z}/2$ if and only if the loop $\eta \in \pi_1 (\mathbb{R}^2 \setminus 0 )$ from Lemma \ref{lemma:computing} is an \textit{odd} multiple of the standard generator. \\

Next, we describe explicit calculations in a class of examples using the above algorithm.

\subsubsection{Configurations of spheres from exceptional unimodal singularities}
\label{subsection: Configurations of spheres from exceptional unimodal singularities}

We now calculate the element $\Delta $ for certain configurations of $(-2)$--spheres arising from vanishing cycles.\\

First, we briefly recall the notion of distinguished basis of vanishing spheres associated to an isolated hypersurface singularity (see \cite{arnold,ebeling-book} for details). 
Let $f : (\mathbb{C}^{3} , 0 ) \to (\mathbb{C},0) $ be the germ of a complex-analytic function with an isolated singular point at $0 \in \mathbb{C}^3$. By Milnor's Fibration Theorem \cite{milnor}, there exists $\varepsilon_0 >0$ such that for each $0 < \varepsilon \leq \varepsilon_0$ there exists $\delta = \delta (\varepsilon )>0$ for which the mapping
\begin{align}
f : B_{\varepsilon}(0) \cap f^{-1} (B_\delta (0) ) \subset \mathbb{C}^3 \to B_\delta (0) \subset \mathbb{C}\label{milnorfibration1}
\end{align}
is a smooth fibration over the complement of $0 \in B_\delta (0)$ with fibers given by compact $4$-manifolds-with-boundary $M$ --- the `Milnor fibers'. 

Let $\mu = \mu (f) = b^2 (M)$ denote the Milnor number of $f$. For small generic parameters $a,b,c \in \mathbb{C}$, the perturbation $\widetilde{f} = f + ax + by +cz$ has only non-degenerate (i.e. Morse) critical points in $B_\varepsilon (0)$, and has exactly $\mu $ of them, with pairwise distinct critical values all contained in the interior of $B_\delta (0) \subset \mathbb{C}$. Such an $\widetilde{f}$ is called a Morsification of $f$, and the mapping 
\begin{align}
\widetilde{f} : B_{\varepsilon}(0) \cap \widetilde{f}^{-1} (B_\delta (0) ) \subset \mathbb{C}^3 \to B_\delta (0)\subset \mathbb{C} \label{milnorfibration2}
\end{align}
is a smooth fibration over the complement of the $\mu$ critical values. Fixing a point $z_0 \in \partial B_{\delta}(0)$, we may identify the fiber $\widetilde{f}^{-1}(z_0 )$ of (\ref{milnorfibration2}) with the fiber $M$ of the fibration (\ref{milnorfibration1}), and the monodromy $\psi \in \pi_0 \mathrm{Diff}(M )$ along the boundary circle $\partial B_\delta (0)$ is the same for both fibrations (\ref{milnorfibration1}-\ref{milnorfibration2}).

A \textit{distinguished basis of vanishing paths} $\gamma_1 , \ldots , \gamma_\mu$ for the Morsification $\widetilde{f}$ consists of an ordered collection of smoothly embedded paths in $B_\delta (0)$ such that:
\begin{itemize}
\item $\gamma_i (0) = z_0$, and for each critical value $z$ of $\widetilde{f}$ there is a (unique) $i$ with $\gamma_i (1) = z$
\item two different paths $\gamma_i , \gamma_j$ meet only at $z_0$
\item the derivatives $\gamma_{1}^\prime (0), \ldots , \gamma_{\mu}^\prime (0)$ are pairwise distinct, and the ordering of the paths $\gamma_1 , \ldots , \gamma_\mu$ is by clockwise outgoing order from $z_0$.
\end{itemize}
For $i = 1 , \ldots , \mu$, let $z_i := \gamma_i (1)$ and let $p_i \in \widetilde{f}^{-1}(z_i )$ be the unique critical point over $z_i$. Associated to the path $\gamma_i$ from $z_0$ to the critical value $z_i$, there is an associated smoothly embedded sphere $S_i \subset M = \widetilde{f}^{-1} (z_0 )$ with $S_i \cdot S_i =-2$ called the \textit{vanishing sphere} of $\gamma_i$, and well-defined up to isotopy: in the local model for a non-degenerate critical point, namely $x^2 + y^2 +z^2 : (\mathbb{C}^3 , 0 ) \to (\mathbb{C}, 0 )$, the non-singular fibers are diffeomorphic to $T^\ast S^2$ and the vanishing sphere is given by the zero section in this cotangent bundle; in general, one uses parallel transport along the path $\gamma_i$ to transport the vanishing sphere from the local model to $M = \widetilde{f}^{-1} (z_0 )$.
The \textit{distinguished basis of vanishing spheres} associated to a distinguished basis of vanishing paths $\gamma_1 , \ldots , \gamma_\mu$ of $\widetilde{f}$ is the collection of $(-2)$--spheres $S_1, \ldots , S_\mu$ smoothly embedded in $M = \widetilde{f}^{-1} (z_0)$, each well-defined up to isotopy, where $S_i$ is the vanishing sphere of $\gamma_i$. The \textit{Dynkin diagram} of a distinguished basis of vanishing paths $\gamma_1 , \ldots , \gamma_\mu$ is the graph with vertices labelled $i =1 , \ldots , \mu$, with an edge connecting two different $i$ and $j$ whenever $S_i$ and $S_j$ have non-trivial homological intersection $S_i \cdot S_j$, in which case the edge is weighted by the integer $S_i \cdot S_j$.\\ 

In Arnold's classification of the unimodal isolated singularities \cite{arnold-normalforms}, he identified a subclass of these consisting of 14 families of singularities $f_\lambda$ known as the \textit{exceptional unimodal singularities}, listed in Table \ref{table:unimodal}. Here $f_\lambda : (\mathbb{C}^3 , 0 ) \to (\mathbb{C} , 0 )$ is a family of isolated singularity germs indexed by a parameter $\lambda \in \mathbb{C}$, such that the Milnor number of $f_\lambda$ stays constant in $\lambda$. 

\begin{table}[h!] 
\centering
\begin{tabular}{|c|c|c|c|}
\hline
\textbf{Name} & \textbf{Equation $f_\lambda$} & \textbf{Gabrielov numbers $(p,q,r)$} & \textbf{Monodromy order $h$} \\
\hline
$E_{12}$ & $x^3 + y^7 + z^2 + \lambda xy^5 $ & (2,3,7) & 42 \\
\hline
$E_{13}$ & $x^3 + x y^5 + z^2 + \lambda y^8 $ & (2,3,8) & 30 \\
\hline
$E_{14}$ & $x^3 + y^8 + z^2 + \lambda xy^6 $ & (2,3,9) & 24 \\
\hline
$Z_{11}$ & $x^3 y + y^5 + z^2 + \lambda xy^4 $ & (2,4,5) & 30 \\
\hline
$Z_{12}$ & $x^3 y + xy^4 + z^2 + \lambda y^6 $ & (2,4,6) & 22 \\
\hline
$Z_{13}$ & $x^3 y + y^6 + z^2 + \lambda xy^5 $ & (2,4,7) & 18 \\
\hline
$Q_{10}$ & $x^2 z + y^3 + z^4 + \lambda yz^3 $ & (3,3,4) & 24 \\
\hline
$Q_{11}$ & $x^2 z + y^3 + yz^3 + \lambda z^5 $ & (3,3,5) & 18 \\
\hline
$Q_{12}$ & $x^2 y + y^3 + z^5 + \lambda yz^4 $ & (3,3,6) & 15 \\
\hline
$W_{12}$ & $x^4 + y^5 + z^2 + \lambda x^2 y^3 $ & (2,5,5) & 20 \\
\hline
$W_{13}$ & $x^4 + x y^4 + z^2 + \lambda y^6 $ & (2,5,6) & 16 \\
\hline
$S_{11}$ & $x^4 + y^2 z + xz^2 + \lambda x^3 z $ & (3,4,4) & 16 \\
\hline
$S_{12}$ & $x^2 y + y^2 z + xz^3 + \lambda z^5 $ & (3,4,5) & 13 \\
\hline
$U_{12}$ & $x^3 + y^3 + z^4 + \lambda xyz^4 $ & (4,4,4) & 12 \\
\hline
\end{tabular}
\caption{The exceptional unimodal hypersurface singularities, with their Gabrielov numbers and order of homological monodromy}
\label{table:unimodal}
\end{table}

For each of the exceptional unimodal singularities $f_\lambda$, by work of Gabrielov \cite{gabrielov-unimodal} there exists a distinguished basis of vanishing spheres in the corresponding Milnor fiber $M$ such that the Dynkin diagram is given by Figure \ref{fig:dynkin}, where $(p,q,r)$ in that Figure is a triple of integers known as the \textit{Gabrielov numbers} of $f_\lambda$ (see Table \ref{table:unimodal} for the list of Gabrielov numbers). From the Dynkin diagram in Figure \ref{fig:dynkin}, we have that the Milnor number of $f_\lambda$ is $\mu = p+q+r$. One can also see from the Dynkin diagram that $b^+(M) =2$ and that $\partial M$ is a rational homology sphere for all the exceptional unimodal singularities (see e.g. \cite[\S 5.47]{ebeling-book}).

\begin{figure}[htbp]
    \centering    \includegraphics[width=0.7\textwidth]{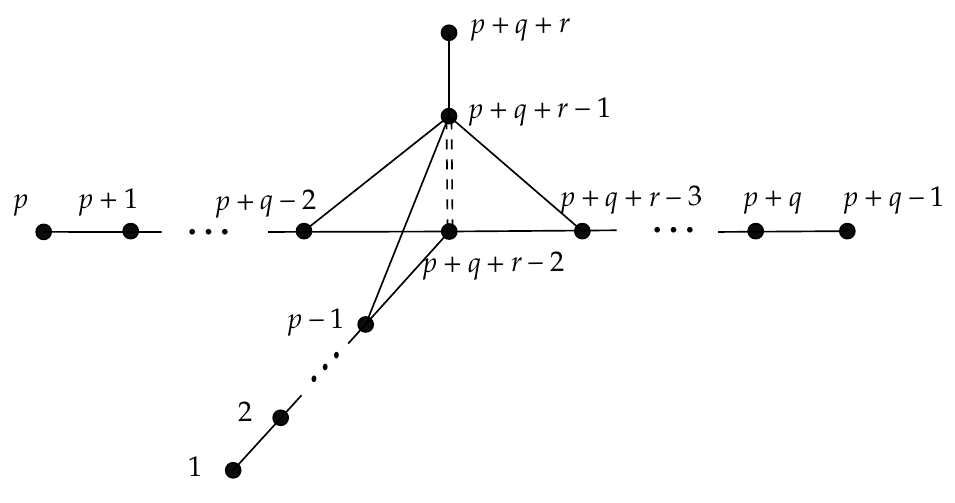}
    \caption{Gabrielov's Dynkin diagram for the exceptional unimodal singularities.}
    \label{fig:dynkin}
\end{figure}

The singularities $f_\lambda$ are all weighted-homogeneous when $\lambda = 0$. From this one sees that the monodromy $\psi$ over $\partial B_\delta (0)$ in (\ref{milnorfibration1}-\ref{milnorfibration2}) induces an automorphism $\psi^\ast \in \mathrm{Aut}H^2 (M,\mathbb{R})$ of the intersection form with finite order \cite{milnor}; the order, which we denote by $h$, is just the weighted-degree of the polynomial $f_0$ (see Table \ref{table:unimodal} for the list of orders $h$). Furthermore, since the family $f_\lambda$ has constant Milnor number, $\psi^\ast$ will have order $h$ in $\mathrm{Aut}H^2 (M, \mathbb{R})$ even when $\lambda \neq 0$ (\cite[Lemma 2.4]{KLMME2}).\\

We now describe a collection of spheres for which we compute the element $\Delta$. Fix any exceptional unimodal singularity $f_\lambda$, and let $S_1 , \ldots , S_\mu$ be a distinguished configuration of vanishing spheres, associated to a distinguished basis of vanishing paths of a Morsification of $f_\lambda$ with the Gabrielov Dynkin diagram from Figure \ref{fig:dynkin}. It is well-known that the symplectic monodromy around the loop based at $z_0$ and encircling $z_i$ using the simple loop determined by the path $\gamma_i$ is given by the Dehn twist $\tau_{S_i}$ \cite{seidel:exactseq}. Thus, the monodromy $\psi$ factors (both in the smooth and the symplectic mapping class group) as a product of Dehn twists on the spheres $S_i$:
\[
\psi = \tau_{S_1} \cdots  \tau_{S_\mu}.
\]
The order of $\psi$ when acting on $H^2 (M , \mathbb{R} )$ is given by $h$ (see Table \ref{table:unimodal} for the list of orders). Thus, the following product of Dehn twists on spheres is homologically trivial:
\begin{align}
\label{eq: power Dehn twist}
\psi^h = (\tau_{S_1}\cdots \tau_{S_\mu})^h .
\end{align}
Furthermore, $\psi^h \in \mathrm{Diff}(M)$ agrees with (the inverse of) the \textit{boundary Seifert-fibered Dehn twist} on the Milnor fiber $M$ \cite[Proposition 2.14]{KLMME}.
\begin{definition}\label{definition:Sexceptional}
Let $\mathcal{S}$ be the ordered collection of $(-2)$--spheres in $M$ given by $\mu h$ spheres:
    \begin{align*}
    \underbrace{S_\mu , \ldots , S_1}_{h \text{ times}}.
    \end{align*}
 \end{definition}

Let $\delta_1 , \ldots , \delta_\mu \in H^2 (M , \mathbb{R})$ denote the Poincaré duals of the fundamental classes of the spheres $S_1 , \ldots , S_\mu$, and then $H^2 (M, \mathbb{R})$ has a basis given by the $\delta_{i}$'s. From the Dynkin diagram, we find a nice choice of maximal positive subspace $H^+ (M) \subset H^2 (M, \mathbb{R} )$. Namely, take $H^+ (M) $ to be the span of $a , b \in H^2 (M , \mathbb{R} )$ where
\begin{align*}
a :=& 2 \delta_{\mu-2} - 2 \delta_{\mu-1} - \delta_\mu \\
b := & \frac{1}{p}\delta_1 + \frac{2}{p}\delta_2 + \cdots + \frac{p-1}{p}\delta_{p-1}\\
+ &  \frac{1}{q}\delta_p + \frac{2}{q}\delta_{p+1} + \cdots + \frac{q-1}{q}\delta_{p+q-2}\\
+ & \frac{1}{r}\delta_{p+q-1} + \frac{2}{r}\delta_{p+q} + \cdots + \frac{r-1}{r}\delta_{\mu-3}\\
+ &  \delta_{\mu -2 }.
\end{align*}
Indeed, one easily checks $a^2 >0$ and $b^2 >0$. Furthermore, $a\cdot b = 0$. 

\begin{figure}[htbp] 
    \centering
    \begin{subfigure}[b]{0.45\textwidth}
        \includegraphics[width=\textwidth]{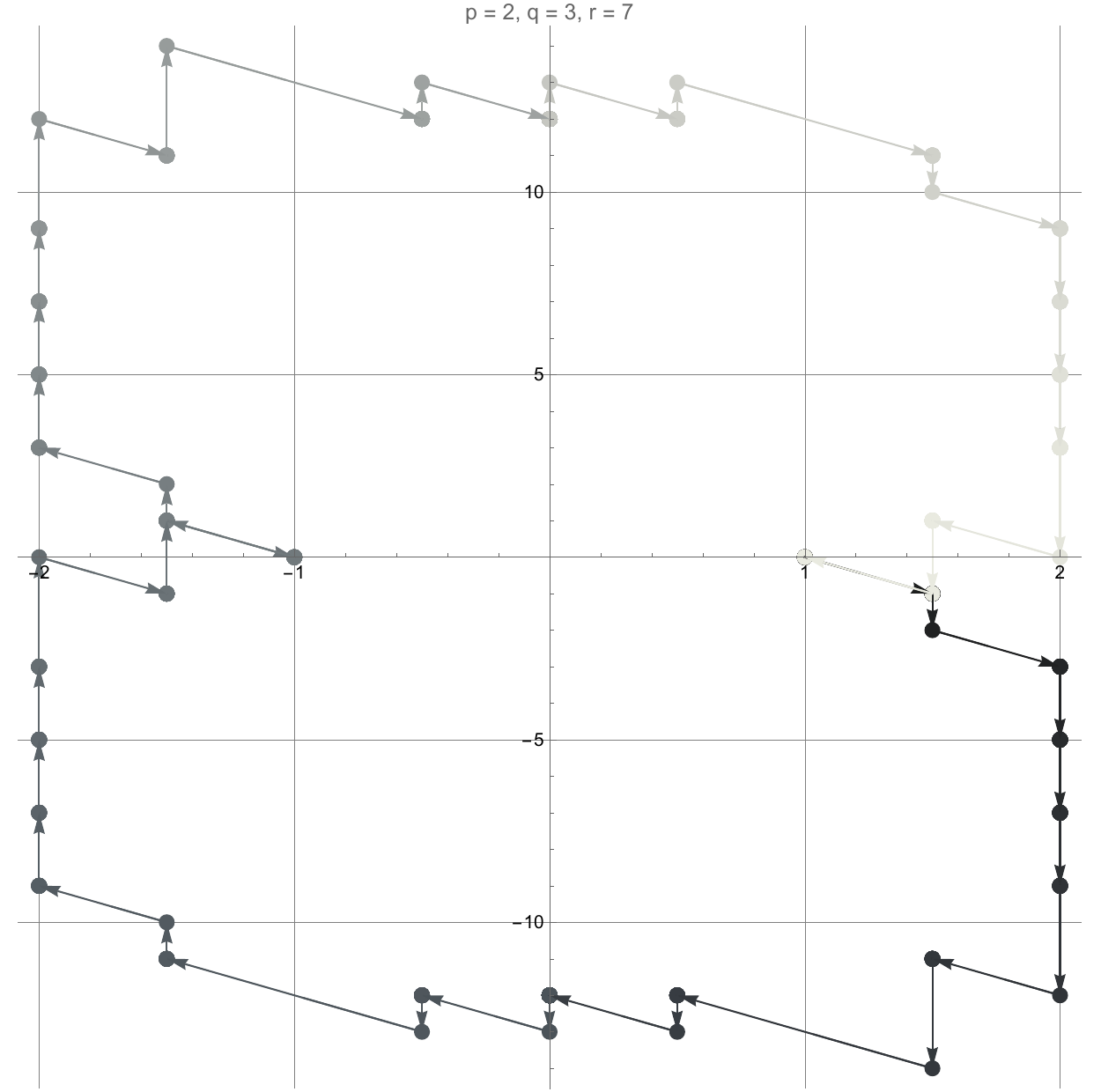}
    \end{subfigure}
    \hfill
    \begin{subfigure}[b]{0.45\textwidth}
        \includegraphics[width=\textwidth]{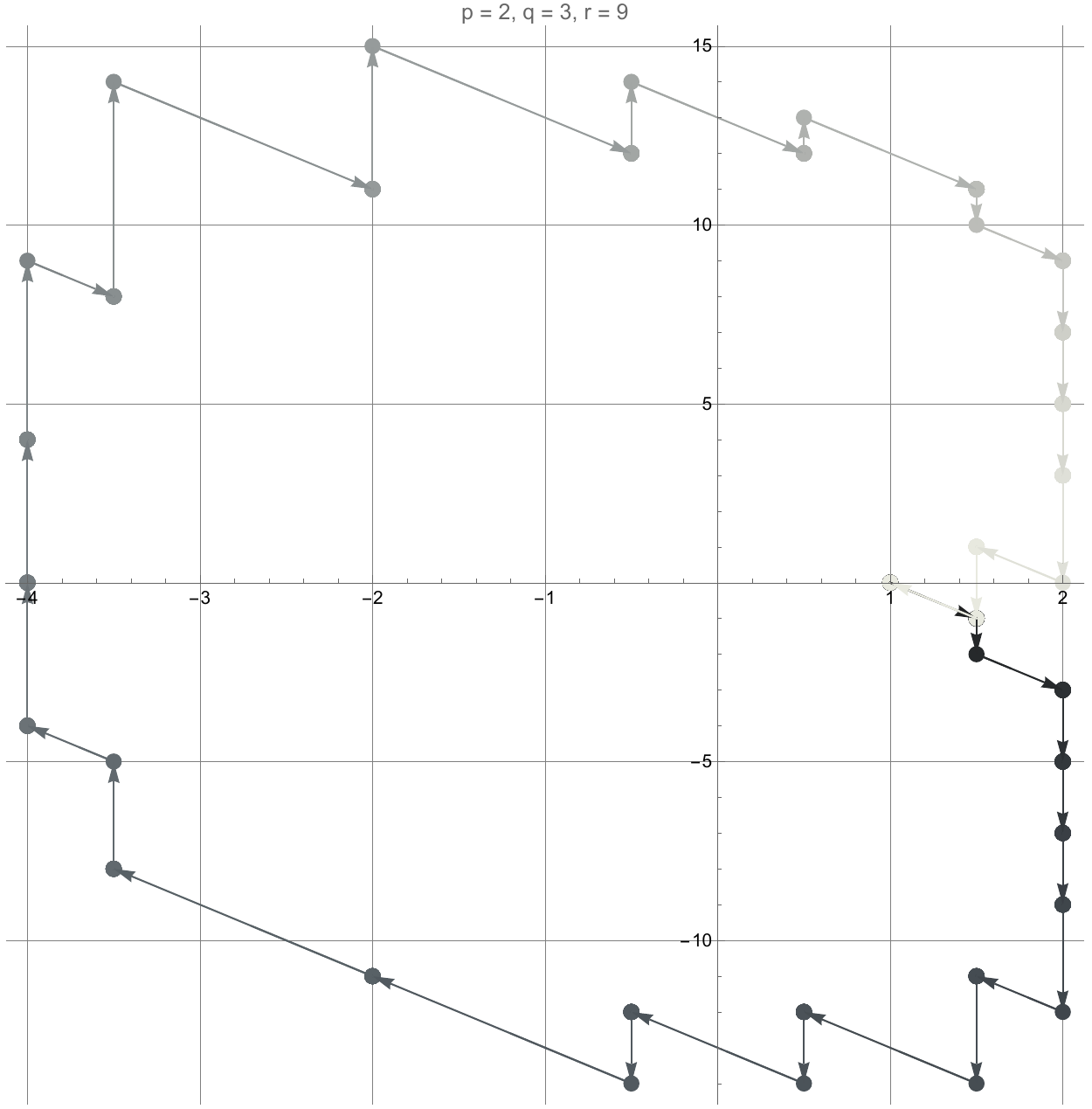}
    \end{subfigure}

    \caption{$\Delta = [\eta ]$ for the exceptional unimodal singularities $E_{12}$ (e.g. $x^2 + y^3 + y^7 = 0$) and $E_{14}$ (e.g. $x^2 + y^3 + z^8 = 0$).}
    \label{figure:237}
    \end{figure}

We can then apply the algorithm from \S \ref{subsubsection:algorithm} to the collection of spheres $\mathcal{S}$ and the aforementioned choice of $H^+ (M,\mathbb{R})$ subspace equipped with the framing given by the basis $a,b$. In Appendix \ref{appendix:code} we include code for \textit{Mathematica}~\cite{Mathematica} which implements this for all the exceptional unimodal singularities, producing a plot of the loop $\eta \in \pi_1 (\mathbb{R}^2\setminus \{0\} , v_0 )$ described in Lemma \ref{lemma:computing}. The plot for the singularities $x^2 + y^3 +z^7 =0$ and $x^2 + y^3 +z^8 = 0$ is shown in Figure \ref{figure:237}, and the plots for the remaining exceptional unimodal singularities are shown in Appendix \ref{appendix:figures}. From these plots, we find that the winding number of the loop $\eta$ is $-1$ for each exceptional unimodal singularity. In particular, we deduce:

\begin{proposition}
\label{prop: Milnor fiber of unimodal singularity}
Let $X$ is a closed $4$-manifold with a smooth embedding $M \subset X$ of the Milnor fiber of an exceptional unimodal singularity. Let $\mathcal{S}$ be the configuration of spheres in $M$ from Definition \ref{definition:Sexceptional}, which we regard as spheres in $X$. Then the element $\Delta (\mathcal{S} )$ (cf. Definition \ref{definition:loop}) is a generator of the group $\pi_1 SO(b^+ (X)) \, (\cong \mathbb{Z} \text{  or  } \mathbb{Z}/2 )$, and in particular it is a non-trivial element.
\end{proposition}

\subsection{Interpretation of $\Delta (S_1 , \ldots , S_n )$ in terms of Lefschetz fibrations}

We now discuss a geometric viewpoint on $\Delta (S_1 , \ldots , S_n )$, from the point of view of Picard Lefschetz theory. Throughout, let $X$ be a closed, oriented $4$-manifold with a collection of smoothly embedded $2$-spheres $S_1, \ldots , S_n$ each with self-intersection $-2$. We will often assume the condition that
\begin{align}
\tau_{S_n}^\ast \cdots \tau_{S_1}^\ast =1 \in \mathrm{Aut}H^2 (X, \mathbb{R}) . \label{trivialhomology}
\end{align}
We also fix an orientation of $H^+ (X)$. 

\subsubsection{Invariance of $\Delta$ under mutations}

For each $1\leq j < n$, consider modifying the collection $(S_1 , \ldots , S_n )$ by the following two operations:
\begin{align*}
\alpha_j : (S_1 , \ldots , S_n ) &\mapsto (S_1 , \ldots , S_{j-1} , \tau_{S_j} (S_{j+1} ) , S_j , S_{j+2} , \ldots , S_n )\\
\beta_j : (S_1 , \ldots , S_n ) &\mapsto (S_1 , \ldots , S_{j-1}, S_{j+1} , \tau_{S_{j+1}}^{-1} (S_j ) , S_{j+2} , \ldots , S_n ).
\end{align*}
One can easily check the following properties:
\begin{enumerate}
\item $\alpha_j $ and $\beta_j$ are inverses to each other.
\item If $1\leq i , j < n$ and and $|i - j|>1$ then $\alpha_i \circ \alpha_j = \alpha_j \circ \alpha_i$.
\item If $1 \leq j < n-1$ then $\alpha_j \circ \alpha_{j+1}\circ \alpha_j = \alpha_{j+1}\circ \alpha_j \circ \alpha_{j+1}$.
\end{enumerate}
By (2) and (3), the operations $\alpha_j$, $1 \leq j < n$, define an action of the $n$-strand Braid group $B_n$ on the set of collections of spheres $(S_1 , \ldots , S_n )$. The operations $\alpha_j, \beta_j =\alpha_j^{-1}$ are referred to as \textit{mutations} of the collection of spheres $(S_1 , \ldots , S_n )$ (\cite{arnold,ebeling-book}). Note that the condition (\ref{trivialhomology}) is invariant under mutations of $(S_1 , \ldots , S_n )$.

\begin{proposition}\label{proposition:mutation}
Suppose the spheres $S_1, \ldots , S_n $ satisfy (\ref{trivialhomology}). Then the element $\Delta (S_1 , \ldots , S_n ) \in \pi_1 SO(b^+ (X))$ is invariant under mutations of $(S_1 , \ldots , S_n )$.
\end{proposition}

\begin{proof}
We show that $\Delta (S_1 , \ldots , S_n ) = \Delta (\alpha_j (S_1 , \ldots , S_n ))=: (\widetilde{S}_1 , \ldots , \widetilde{S}_n)$. We use the terminology of \S \ref{subsubsection:loop}. In particular, we have the endpoints $e_i = \tau_{S_{i-1}}^\vee \cdots \tau_{S_1}^\vee (e_0 ) \in \mathcal{E}(H^2 (X, \mathbb{R})$ and the paths $\gamma_i$ from $e_{i-1}$ to $\mathrm{Fix}(\tau_{S_i}^\vee) \subset \mathcal{E}(H^2 (X, \mathbb{R}))$, from which we obtain the paths $\eta_i = \overline{\tau_{S_j}^\vee (\gamma_i )} \circ \gamma_i$ from $e_{i-1}$ to $e_i$, which make up $\Delta(S_1 , \ldots , S_n ) = [\eta_n \circ \cdots \circ \eta_1 ]$. 

We set $\widetilde{e}_{0} = e_0$, from which we obtain corresponding endpoints $\widetilde{e}_i$, $i =1 , \ldots , n$, for $\Delta (\widetilde{S}_1 , \ldots , \widetilde{S}_n )$. Observe that $\widetilde{e}_i = e_i$ for $i \neq j$. We now build corresponding paths $\widetilde{\gamma_i}$ and $\widetilde{\eta_i}$, $i =1, \ldots , n$ for $\Delta (\widetilde{S}_1 , \ldots , \widetilde{S}_n )$. If $i \neq j , j+1$ then we set $\widetilde{\gamma_i} = \gamma_i$. We set
\[
\widetilde{\gamma}_j := \tau_{S_j}^\vee ( \gamma_{j+1}) 
\]
which is a path from $ \tau_{S_j}^\vee e_j = e_{j-1} = \widetilde{e}_{j-1}$ to $\tau_{S_j}^\vee \cdot \mathrm{Fix} (  \tau_{S_{j+1}}^\vee  ) = \mathrm{Fix}(\tau_{\widetilde{S}_j}^\vee )$. We set 
\[
\widetilde{\gamma}_{j+1} := \gamma_j \circ \tau_{S_j}^\vee \overline{\eta}_{j+1}
= \gamma_j \circ \tau_{S_j}^\vee \big( \overline{\gamma}_{j+1} \circ \tau_{S_{j+1}}^\vee \gamma_{j+1} \big) 
\]
which is a path from $\tau_{S_j}^\vee \tau_{S_{j+1}}^\vee e_j =  \tau_{S_j}^\vee \tau_{S_{j+1}}^\vee \tau_{S_j}^\vee e_{j-1} = \widetilde{e}_j$ to $\mathrm{Fix}(\tau_{S_j}^\vee ) = \mathrm{Fix}(\tau_{\widetilde{S}_{j+1}}^\vee )$. 

Finally, we show that the paths $\widetilde{\eta}_{j+1} \circ \widetilde{\eta}_j $ and $\eta_{j+1} \circ \eta_j$, both going from $e_{j-1} = \widetilde{e}_{j-1}$ to $e_{j+1} = \widetilde{e}_{j+1}$, are homotopic relative to the endpoints, from which the required result will follow. For this, note that
\begin{align*}
\widetilde{\eta_j} & = \tau_{\widetilde{S}_j} \overline{\widetilde{\gamma}_j} \circ \widetilde{\gamma}_j  = \tau_{S_j}^\vee \tau_{S_{j+1}}^\vee \underbrace{\tau_{S_j}^\vee \tau_{S_j}^\vee }_{=1}\overline{\gamma_{j+1}} \circ \tau_{S_j}^\vee \gamma_{j+1}\\
& = \tau_{S_j}^\vee \tau_{S_{j+1}}^\vee \overline{\gamma_{j+1}} \circ \tau_{S_j}^\vee \gamma_{j+1}\\
\widetilde{\eta}_{j+1} &= \tau_{\widetilde{S}_{j+1}}^\vee \overline{\widetilde{\gamma}_{j+1}} \circ \widetilde{\gamma}_{j+1} = \underbrace{\tau_{S_{j+1}}^\vee \overline{\gamma_{j+1}}  \circ \gamma_{j+1} }_{\eta_{j+1}} \circ \underbrace{\tau_{S_j}^\vee \overline{\gamma_j} \circ \gamma_j}_{\eta_j} \circ \tau_{S_j}^\vee \overline{\gamma_{j+1}} \circ \tau_{S_j}^\vee \tau_{S_{j+1}}^\vee \gamma_{j+1}\\
& = \eta_{j+1} \circ \eta_j \circ \tau_{S_j}^\vee \overline{\gamma_{j+1}} \circ \tau_{S_j}^\vee \tau_{S_{j+1}}^\vee \gamma_{j+1}.
\end{align*}
Hence, we have
\begin{align*}
\widetilde{\eta}_{j+1} \circ \widetilde{\eta}_j & = \eta_{j+1} \circ \eta_j \circ \tau_{S_j}^\vee \overline{\gamma_{j+1}} \circ \underbrace{\tau_{S_j}^\vee \tau_{S_{j+1}}^\vee \gamma_{j+1} \circ \tau_{S_j}^\vee \tau_{S_{j+1}}^\vee \overline{\gamma_{j+1}}}_{\simeq \ast} \circ \tau_{S_j}^\vee \gamma_{j+1} \\
& = \eta_{j+1} \circ \eta_j \circ \underbrace{\tau_{S_j}^\vee \overline{\gamma_{j+1}} \circ \tau_{S_j}^\vee \gamma_{j+1}}_{\simeq \ast}\\
& \simeq \eta_{j+1} \circ \eta_j .
 \end{align*}
\end{proof}

We recall the following definition (\cite{donaldson-lefschetz-mapping}):

\begin{definition}
Let $X$ be a closed oriented smooth $4$-manifold. A \textit{smooth Lefschetz fibration} with fiber $X$ consists of data $(E, \Sigma , f , z_0 , z_1 , \ldots , z_n , \phi )$ where $E$ is a compact oriented smooth $6$-manifold-with-boundary $E$, $\Sigma$ is an compact oriented connected surface-with-boundary, $f: E \to \Sigma$ is a smooth map with $f ( \partial E) = \partial \Sigma$, $z_0 \in \Sigma$ is a regular value of $f$,  $ z_1 , \ldots , z_n  \in \Sigma\setminus \partial \Sigma$ is an ordered collection of distinct points comprising the set of critical values of $f$, and $\phi : X \cong f^{-1} (z_0 )$ is an orientation-preserving diffeomorphism, such that for each $i =1, \ldots ,n$:
\begin{enumerate}
\item $f^{-1} (z_i )$ contains a unique critical point of $f$, denoted $p_i$
\item there exists oriented smooth charts on $E$ at $p_i$ (with coordinates denoted in complex notation by $x,y,z$) and $\Sigma$ at $z_i$, such that in those coordinates we have $\pi = x^2 + y^2 + z^2$. 
\end{enumerate}
Two smooth Lefschetz fibrations $(E,\Sigma , f, z_0 , z_1 , \ldots , z_n , \phi)$ and $(E^\prime , \Sigma^\prime , f^\prime , z_{0}^\prime , z_{1}^\prime , \ldots , z_{n}^\prime , \phi^\prime)$ with fiber $X$ are \textit{equivalent} if there exists orientation-preserving diffeomorphisms $E \cong E^\prime$ and $\Sigma \cong \Sigma^\prime$ sending $z_i$ to $z_{i}^\prime$ for each $i = 0,1, \ldots , n$ and compatible with the projections $f , f^\prime$ and the diffeomorphisms $\phi, \phi^\prime$. In what follows we denote a smooth Lefschetz fibration plainly as $f : E\to \Sigma$.
\end{definition}

Proposition \ref{proposition:mutation} gives an interpretation of $\Delta(S_1 , \ldots , S_n )$ in terms of smooth Lefschetz fibrations. Indeed, it is a well-known fact (see e.g. \cite{donaldson-lefschetz-mapping}) that there is a one-to-one correspondence between:
\begin{itemize}
 \item Smooth Lefschetz fibrations $f:  E \to D^2$ with fiber $X$ over a disk $D^2$, up to equivalence
 \item Ordered collections $(S_1 , \ldots , S_n )$ of isotopy classes of smoothly embedded $(-2)$--spheres in $X$, up to mutation.
 \end{itemize}
Namely, to a Lefschetz fibration $E \to D^2$ we associate the distinguished basis of vanishing spheres $S_1 , \ldots , S_n $ in $X = f^{-1}(z_0 )$ obtained from a choice of distinguished basis of vanishing paths $\gamma_1 , \ldots , \gamma_n$ in $D^2$ from $z_0$ to the critical values $z_1 , \ldots , z_n$  (this is defined similarly as in \S \ref{subsection: Configurations of spheres from exceptional unimodal singularities}). Each $\gamma_i$ determines a simple loop based at $z_0$ travelling around the critical value $z_i$ counterclockwise once, and the monodromy of this loop is the Dehn twist $\tau_{S_i}$ on the sphere $S_i \subset X = f^{-1} (z_0 )$. In particular, the boundary monodromy of $f : E \to D^2$ is given by $\tau_{S_1} \cdots \tau_{S_n}$. Thus, under the above correspondence, those Lefschetz fibrations over $D^2$ whose boundary monodromy acts as the identity on $H^2 (X, \mathbb{R} )$ correspond to configurations $(S_1, \ldots , S_n )$ satisfying (\ref{trivialhomology}). Thus, by Proposition \ref{proposition:mutation}, we deduce:
\begin{corollary}\label{corollary:DeltaLefschetz}
Let $f : E \to D^2$ be a smooth Lefschetz fibration with fiber $X$ and boundary monodromy acting as the identity on $H^2 (X, \mathbb{R})$. The element $\Delta (S_1 , \ldots , S_n ) \in \pi_1 SO(b^+ (X))$, where $S_1 , \ldots, S_n$ is any choice of distinguished basis of vanishing spheres in $f^{-1} (z_0 ) =X$, is an invariant of the Lefschetz fibration $f : E \to D^2$. We denote it by $\Delta (f: E\to D^2) \in \pi_1 SO(b^+ (X))$.
\end{corollary}

\subsubsection{$\Delta (S_1 , \ldots , S_n )$ as a characteristic class}

We now interpret $\Delta$ as a suitable characteristic class associated to Lefschetz fibration $E \to D^2$ with fiber $X$ and boundary monodromy acting as the identity on $H^2 (X, \mathbb{R} )$.\\ 

Let $f : E \to D^2$ be a smooth Lefschetz fibration with fiber $X = f^{-1} (z_0 )$ and critical values $z_1 , \ldots , z_n \in D^2 \setminus \partial D^2$. Similarly as in \S \ref{framingH+mapping}, we have an oriented vector bundle denoted
\[
H^+ (f) \to D^2 \setminus \{z_1 , \ldots , z_n \}
\]
whose fiber over $z \in D^2 \setminus \{z_1 , \ldots , z_n \}$ is a maximal positive subspace in $H^2 (f^{-1} (z) , \mathbb{R} )$, and this bundle is unique up to canonical isomorphisms. The monodromy around a small circle around the critical value $z_i$ is given by a Dehn twist on the vanishing cycle, hence using the Dehn twist framing (\S \ref{subsubsection:reflectionframing}) we obtain an extension of the bundle $H^+ (f) \to D^2 \setminus \{ z_1 , \ldots , z_n \}$ over to the critical values, and this extension is well-defined up to isomorphisms. We denote this oriented vector bundle plainly as
\begin{align}
H^+ (f) \to D^2 .\label{canonicalextension1}
\end{align}

Suppose further that the boundary monodromy of $f: E\to D^2$ acts as the identity on $H^2 (X, \mathbb{R})$. Then we have the canonical framing (Definition \ref{definition:canonicalframingd}) of the restriction of $H^+ (f) $ to $\partial D^2$. From this we obtain a canonical (up to isomorphism) extension of (\ref{canonicalextension1}) to an oriented vector bundle
\begin{align}
H^+ (f) \to S^2 = D^2 / \partial D^2 . \label{canonicalextension2}
\end{align}

\begin{proposition}\label{proposition:extension}
Let $f : E \to D^2$ be a smooth Lefschetz fibration with fiber $X$ and boundary monodromy acting as the identity on $H^2 (X, \mathbb{R})$. The invariant $\Delta (f : E \to D^2 )$ (cf. Corollary \ref{corollary:DeltaLefschetz}) agrees with the element in $\pi_1 SO(b^+ (X) ) \cong \pi_2 BSO( b^+ (X))$ represented by the classifying map of the oriented vector bundle (\ref{canonicalextension2}). 
\end{proposition}

\begin{proof}
Make a choice of distinguished basis of vanishing paths $\gamma_1 , \ldots , \gamma_n$ in $D^2$ from $z_0$ to the critical values $z_1 , \ldots , z_n$, and let $S_1 , \ldots , S_n $ be the corresponding vanishing cycles in $f^{-1} (z_0 ) = X$. A neighborhood $D_1$ of $\bigcup_{i =1}^n \gamma_i \subset D^2$ is homeomorphic to a disk, and we thus obtain a corresponding decomposition of $S^2 = D^2 /\partial D^2$ as the union of two disks $D_1 \cup D_2$ along their common boundary $S^1 := \partial D_1 = \partial D_2$. The Dehn twist framings induce a canonical identification of (\ref{canonicalextension2}) over $D_1$ with the product bundle $H^+ (X) \times D_1$. On the other hand, the canonical framing (given by the fact that the boundary monodromy acts as the identity on $H^2 (X, \mathbb{R})$) induces a similar identification with a product bundle over $D_2$. The map $S^1 \to SO(b^+ (X))$ given by the difference of the two trivialisations over $S^1$ coincides under $\pi_1 SO(b^+ (X)) \cong \pi_2 BSO(b^+ (X))$ with the element representing the classifying map of the vector bundle (\ref{canonicalextension2}). On the other hand, this framing difference was shown to agree with the element $\Delta (S_1 , \ldots , S_n )$ in Proposition \ref{prop: difference of xi_d}.
\end{proof}

\section{The family Bauer-Furuta invariant}\label{section: familyBF}
\subsection{The approximated Seiberg-Witten map} We start with some notation. Let $X\hookrightarrow E\to B$ be a smooth bundle whose base $B$ is a smooth, compact manifold (possibly with boundary) and whose fiber is a closed 4-manifold $X$. We assume $b_1 (X) =0$ and $b^{+}(X)>\operatorname{dim}(B)$.  We use $X_{b}$ to denote the fiber over $b\in B$. Pick a base point $b_0$ and fix a diffeomorphism $X_{b_0}\cong X$. We also fix a homological orientation on $X$ (i.e. an orientation of $H^+ (X)$). $\widetilde{\mathfrak{s}}$ be a family spin-c structure: i.e. a spin-c structure on $T^{v}E:=\ker (TE\to TB)$. We assume that the restriction of $\mathfrak{s}$ to a fiber $X$ is a spin-c structure $\mathfrak{s}$ that satisfies 
\[
d(\mathfrak{s}):=\frac{c^2_{1}(\mathfrak{s})-2\chi(X)+3\sigma(X)}{4}=0.
\]

Let $\widetilde{A}_{0}=\{A_{0,b}\}_{b \in B}$ be a family spin-c connection. Then we have the family Dirac operator 
\[
\widetilde{D}^{+}_{\widetilde{A}_0}(E)=\{D^+_{A_{0,b}}: \Gamma(S^{+}_{b})\to \Gamma(S^{-}_{b})\}_{b\in B}
\]
When $E$ and $\widetilde{A}_0$ are obvious from the context, we just write $\widetilde{D}^{+}$ instead of $\widetilde{D}^{+}_{\widetilde{A}_0}(E)$. We use $\ind(\widetilde{D}^{+}_{\widetilde{A}_0}(E))$ to denote the (complex) index bundle of  $\widetilde{D}^{+}_{\widetilde{A}_0}(E)$ and use $\Det(\widetilde{D}^{+}_{\widetilde{A}_0}(E))$ to denote the determinant line bundle of $\ind(\widetilde{D}^{+}_{\widetilde{A}_0}(E))$. 

Associated to the family $E\to B$, we also have the family operator 
\[
\widetilde{d}=\widetilde{d}(E):=\{(d^+,d^*):\Omega^{1}(X_{b})\to \Omega^2_{+}(X_{b})\oplus\Omega^{0}_0(X_{b})\}_{b\in E}. 
\]
The index bundle $\ind(\widetilde{d}(E))$ is exactly the bundle $H^{+}(E)$. 
 Consider the family Seiberg-Witten map 
\[
\operatorname{SW}: \mathcal{U}^{+}\oplus \mathcal{V}^{+}\to \mathcal{U}^{-}\oplus \mathcal{V}^{-}.
\]
Here $\mathcal{U}^{\pm}$ are complex Hilbert spaces over $B$. And $\mathcal{V}^{\pm}$ are real Hilbert spaces over $B$. After taking finite dimensional approximations, we obtain the approximated Seiberg-Witten map 
\begin{equation}\label{eq: sw-apr}
\operatorname{SW_{apr}}:U^{+}\oplus V^{+}\to U^{-}\oplus V^{-}
\end{equation}
Here $U^{\pm}$ are finite dimensional complex vector bundles over $B$ that satisfies $U^{+}-U^{-}= \ind(\widetilde{D})\in K(B)$. And $V^{\pm}$ are real vector bundles over $B$ with that satisfy $V^{-}\cong V^{+}\oplus H^{+}$. This map is $S^1$-equivariant, where $S^1$ acts as scalar multiplication on $U^{\pm}$ and acts trivially on $V^{\pm}$. The map $\operatorname{SW_{apr}}$ satisfies the following additional properties:
\begin{enumerate}
    \item The restriction $\operatorname{SW_{apr}}|_{V^{+}}$ is the standard inclusion $V^{+}\hookrightarrow V^{-}$.
    \item There exists large $R$ and small $\epsilon$ such that 
    \begin{equation}\label{eq: bounded}
    \operatorname{SW_{apr}}(S_{R}(U^{+}\oplus V^{+}))\cap D_{\epsilon}(U^{-}\oplus V^{-})=\emptyset.    
    \end{equation}
    Here $S_{R}(-)$ denotes the sphere bundle of radius $R$ and $D_{\epsilon}(-)$ denotes the disk bundle of radius $\epsilon$.
    \item There exists a section $\mathfrak{p}: B\to  D_{\epsilon}(V^{-}\setminus V^{+})$ such that the $\mathfrak{p}$-perturbed family Seiberg-Witten equations satisfy the transversality condition. After a finite dimensional approximation, this implies that $\mathfrak{p}(B)$ is transverse to $\operatorname{SW_{apr}}|_{D_{R}(U^{+}\oplus V^{+})}$. 
\end{enumerate}    
Take the Thom spaces 
\[
\operatorname{Th}(U^{+}\oplus V^{+})=D_{R}(U^{+}\oplus V^{+}) / S_R( U^+\oplus V^+)\]
and 
\[\operatorname{Th}(U^{-}\oplus V^{-})=(U^{-}\oplus V^{-})/((U^{-}\oplus V^{-})\setminus \mathring{D}_{\epsilon}(U^{-}\oplus V^{-}))
\]
By the boundedness condition (\ref{eq: bounded}), the map (\ref{eq: sw-apr}) induces an $S^1$-equivariant map 
\begin{equation}\label{eq: swapr+}
\swapr: \operatorname{Th}(U^{+}\oplus V^{+})\to \operatorname{Th}(U^{-}\oplus V^{-})
\end{equation}
All our invariants will be extracted from this approximated Seiberg-Witten map.

\subsection{The family Bauer-Furuta invariant of a diffeomorphism}
Now we let $E=T(f)\to B=S^1$ for some diffeomorphism $f:X\to X$. We assume $f$ fix a spin-c structure $\mathfrak{s}$ and the homological orientation. We also assume $b^+_{2}(X)\equiv 3\mod 4$. There is a unique family spin-c structure $\widetilde{s}$ that restricts to $\mathfrak{s}$ to fibers. To define the family Bauer-Furuta invariant, we pick a framing $\xi_{D}$ on the complex line bundle $\Det(\widetilde{D}^{+}_{\widetilde{A}_0}(E))$ and a framing $\xi_{d}$ on the vector bundle $H^{+}(E)$. Consider the approximated Seiberg-Witten map (\ref{eq: swapr+}). We pick trivializations of $U^{\pm}$ and $V^{\pm}$ that are compatible with $\xi_{D}$ and $\xi_{d}$ (up to homotopy). Such trivializations induce the identification 
\[
\operatorname{Th}(U^{+}\oplus V^{+})\cong S^1_{+}\wedge S^{(M+2k+2)\mathbb{C}+N\mathbb{R}}, \quad \operatorname{Th}(U^{-}\oplus V^{-})\cong S^1_{+}\wedge S^{M\mathbb{C}+(N+4k+3)\mathbb{R}}
\]
Here we use $S^V$ and $S(V)$ to denote the representation sphere and the unit sphere of a representation space $V$ and $b^+ = 4k+3$.
Now we define the family Bauer-Furuta invariant \[\FBafu(f,\mathfrak{s},\xi_{D},\xi_{d})\in \mathbb{Z}/2.\]

Consider the composition 
\[
S^1_{+}\wedge S^{(M+2k+2)\mathbb{C}+N\mathbb{R}}\xrightarrow{\swapr} S^1_{+}\wedge S^{M\mathbb{C}+(N+4k+3)\mathbb{R}}\xrightarrow{\operatorname{pj}} S^{M\mathbb{C}+(N+4k+3)\mathbb{R}},
\]
where $\operatorname{pj}$ denotes the projection to the second component. This map represents an element in the $S^1$-equivariant stable homotopy group 
\[
[\operatorname{pj}\circ \swapr]\in [S^{(2k+2)\mathbb{C}}\wedge B_{+}, S^{(4k+3)\mathbb{R}}]^{S^1}
\]
\begin{lemma} We have a canonical isomorphism  
\begin{equation}\label{eq: computation stable homotopy group}
 [S^{(2k+2)\mathbb{C}}\wedge B_{+}, S^{(4k+3)\mathbb{R}}]^{S^1}\cong \mathbb{Z}\oplus \mathbb{Z}/2    
\end{equation}
\end{lemma}
\begin{proof}
We have a natural inclusion map $S^0\hookrightarrow B_{+}$ and a natural projection map $B_{+}\to S^0$, which are stably homotopy inverse to each other. So they induce a splitting $B_{+}\cong S^0\vee S^1$. 
This gives  a canonical isomorphism 
\[
[S^{(2k+2)\mathbb{C}}\wedge B_{+}, S^{(4k+3)\mathbb{R}}]^{S^1}\cong [S^{(2k+2)\mathbb{C}}, S^{(4k+3)\mathbb{R}}]^{S^1}\oplus [S^{(2k+2)\mathbb{C}}, S^{(4k+2)\mathbb{R}}]^{S^1}
\]
By the equivariant Hopf theorem, we have $[S^0,S^{a}]^{S^1}=0$ for any $a>0$. By the long exact sequence of stable cohomotopy groups induced by the cofiber sequences \[S^0\to S^{(2k+2)\mathbb{C}}\to S^1\wedge S((2k+2)\mathbb{C})_{+},\] we have  
\[
[S^{(2k+2)\mathbb{C}}, S^{(4k+3)\mathbb{R}}]^{S^1}\cong [S((2k+2)\mathbb{C})_{+}, S^{(4k+2)\mathbb{R}}]^{S^1}
\]
and 
\[
[S^{(2k+2)\mathbb{C}}, S^{(4k+2)\mathbb{R}}]^{S^1}\cong [S((2k+2)\mathbb{C})_{+}, S^{(4k+1)\mathbb{R}}]^{S^1}
\]
Note that the $S^1$-action on $S((2k+2)\mathbb{C})_{+}$ as complex multiplication and trivial on $S^{(4k+1)\mathbb{R}}$, so we have
\[
[S((2k+2)\mathbb{C})_{+}, S^{(4k+1)\mathbb{R}}]^{S^1}\cong [\mathbb{CP}^{2k+1}_{+}, S^{(4k+1)\mathbb{R}}]\]
By the CW approximation theorem, we have 
\[
[\mathbb{CP}^{2k+1}_{+}, S^{(4k+1)\mathbb{R}}]\cong [\mathbb{CP}^{2k+1}/\mathbb{CP}^{2k-1}, S^{4k+1}]\cong [S^{4k+2}\vee S^{4k}, S^{4k+1}]\cong \mathbb{Z}/2.
\]
Similarly, we have 
\[
[S((2k+2)\mathbb{C})_{+}, S^{(4k+2)\mathbb{R}}]^{S^1}\cong \mathbb{Z}.
\]\end{proof}

\begin{lemma}
Under the decomposition (\ref{eq: computation stable homotopy group}), the first component of $[\operatorname{pj}\circ \swapr]$ equals the Seiberg-Witten invariant $\operatorname{SW}(X,\mathfrak{s})$.    
\end{lemma}
\begin{proof}
If we restrict (\ref{eq: sw-apr}) to a single point in $B$, we recover the approximated Seiberg-Witten map for $(X,\mathfrak{s})$. Therefore, the first component    $[\operatorname{pj}\circ\swapr]$ represents the Bauer-Furuta invairant of $(X,\mathfrak{s})$, which is equivalent to the Seiberg-Witten invariant $\operatorname{SW}(X,\mathfrak{s})$ because $d(\mathfrak{s})=0$. 
\end{proof}
\begin{definition}
The family Bauer-Furuta invariant $\FBafu(f,\mathfrak{s},\xi_{D},\xi_{d})\in \mathbb{Z}/2$ is defined as the second component of $[\operatorname{pj}\circ\swapr]$ under the decomposition (\ref{eq: computation stable homotopy group}). 
\end{definition}

Via the classical Pontryagin-Thom construction, we can translate the  $\FBafu(f,\mathfrak{s},\xi_{D},\xi_{d})$ in terms of the framed cobordism class of the 1-dimensional Seiberg-Witten moduli space. Consider the vector bundle $\pi:\widetilde{W}\to W$, where 
\[
\widetilde{W}=((D_{R}(U^{+}\oplus V^{+})\setminus( \{0\}\times V^+))\times_{B} (U^{-}\oplus V^{-}))/S^{1}
\]

and \[W=(D_{R}(U^{+}\oplus V^{+})\setminus (\{0\}\times V^{+}))/S^1.  \] 
Then the sections of $\pi$ are one-to-one corresponding to $S^1$-equivariant maps. \[D_{R}(U^{+}\oplus V^{+})\setminus (\{0\}\times V^+)\to U^{-}\oplus V^{-}\]
that cover the identity map on $B$. 
In particular, we have a section $s_{sw}:W\to \widetilde{W}$ that corresponds to the map $\operatorname{SW_{apr}}$ and a section $s_{\mathfrak{p}}:W\to \widetilde{W}$ that corresponds to the perturbation $\mathfrak{p}$. By our choice of $\mathfrak{p}$, these two sections intersect transversely. The transverse intersection $\mathcal{M}_{sw}:=s_{sw}(W)\pitchfork s_{\mathfrak{p}}(W)$ is an embedded 1-dimensional submanifold of $\widetilde{W}$. The manifold $\mathcal{M}_{sw}$ is compact because \[(0,\mathfrak{p}(B))\notin \operatorname{SW_{apr}}((\{0\}\times V^{+})\cup S_{R}(U^{+}\oplus V^{+})).\]
Furthermore, note the isomorphisms 
\[
Ns_{sw}(W)\cong (\pi|_{s_{sw}(W)})^{*}(\widetilde{W})\text{ and } Ns_{\mathfrak{p}}(W)\cong (\pi|_{s_{\mathfrak{p}}(W)})^{*}(\widetilde{W}).
\]
So we have a canonical isomorphism 
\[
N\mathcal{M}_{sw}\cong Ns_{sw}(W)|_{\mathcal{M}_{sw}}\oplus Ns_{\mathfrak{p}}(W)|_{\mathcal{M}_{sw}}\cong (\pi|_{\mathcal{M}_{sw}})^{*}(\widetilde{W})\oplus (\pi|_{\mathcal{M}_{sw}})^{*}(\widetilde{W}). 
\]
So up to homotopy, $N\mathcal{M}_{sw}$ has a canonical trivialization $\xi_{c}$ . Note that $\widetilde{W}$ is canonically oriented as a manifold by the homological orientation (and the orientation of $B$). So $\xi$ induces canonical orientation on $\mathcal{M}_{sw}$. Thus, we obtain an element $[\mathcal{M}_{sw},\xi_{c}]\in \Omega_{1}^{\operatorname{fr}}(\widetilde{W})$, the one-dimensional framed bordism group of $\widetilde{W}$. 

Now, we pick trivializations of $U^{\pm},V^{\pm}$ that are compatible with $\xi_{D}$ and $\xi_{d}$. Such trivializations will induce a homeomorphism 
\begin{equation}\label{eq: trivialization of tildeW}
\widetilde{W}\cong \widetilde{W}':=S^1\times \mathbb{R}^{2N+4k+3}\times  (0,R)\times (S(\mathbb{C}^{M+2k+2})\times \mathbb{C}^{M})/S^1    
\end{equation}
\begin{lemma}\label{lem: cobordism computation} We have a canonical isomorphism
\begin{equation}\label{eq: cobordism group}
\Omega_{1}^{\operatorname{fr}}(\widetilde{W}')\cong \mathbb{Z}\oplus\mathbb{Z}/2. 
\end{equation}
\end{lemma}

\begin{proof} Take any element $[Y,\xi] \in \Omega^{\mathrm{fr}}_1 (\widetilde{W}^\prime )$, and let $[Y]=m \in  \mathbb{Z} =  H_{1}(\widetilde{W}')$. 
We take $|m|$ disjoint points $p_{1},\cdots, p_{m}\in \mathbb{R}^{2N+4k+3}\times (0,R)\times (S(\mathbb{C}^{M+2k+2})\times \mathbb{C}^{M})/S^1$ and consider the submanifolds $Y(p_i)=S^1\times  \{p_{i}\}$ for $1\leq i\leq m$, oriented such that $[Y]=\sum^{|m|}_{i=1}[Y(p_i)]$. The manifold $Y(p_{i})$ has a canonical framing $\xi_{i}$, obtained by pulling back a trivilization of $T_{p_{i}}(\mathbb{R}^{2N+4k+3}\times(0,R)\times (S(\mathbb{C}^{M+2k+2})\times \mathbb{C}^{M})/S^1)$. Let $F\hookrightarrow \widetilde{W}'$ be any cobordism from $Y$ to $\cup Y(p_i)$. Then $\xi\cup (\cup \xi_{i})$ can be extended to $F$ if and only if
the relative Stiefel-Whitney class 
\[
w_{2}(NF, \xi,\cup \xi_{i})\in \mathbb{Z}/2\cong H^{2}(F, \partial F;\mathbb{Z}/2).
\]
vanishes. 

Note that $\widetilde{W}'$ is homotopy equivalent to $S^1\times \oplus_{M}O(1)$, where $O(1)$ denotes the dual of the tautological line bundle over $\mathbb{CP}^{2M+2k+1}$. Moreover, 
\[
w_{2}(T(\oplus_{M}O(1)))=w_{2}(T\mathbb{CP}^{2M+2k+1})+Mw_{2}(O(1))=2M+2k+2\equiv 0\mod 2.
\]
So $w_{2}(T\widetilde{W}')=0$. From this, it follows easily that the number $w_{2}(NF, \xi,\cup \xi_{i})$ is independent of the chosen cobordism $F$. The desired isomorphism is given by 
\[
[Y,\xi]\mapsto (m, w_{2}(NF, \xi,\cup \xi_{i}))\in \mathbb{Z}\oplus \mathbb{Z}/2.
\]
\end{proof}

Via the homeomorphism (\ref{eq: trivialization of tildeW}) and the isomorphism (\ref{eq: cobordism group}), we can treat the framed moduli space as \[[\mathcal{M}_{sw},\xi_{c}]\in \mathbb{Z}\oplus \mathbb{Z}/2.\]
By the Pontryagin-Thom correspondence, this is exactly the element $[\operatorname{pj}\circ \swapr]$ coming from (\ref{eq: computation stable homotopy group}). Thus, the second component of $[\mathcal{M}_{sw} , \xi_c ]$ equals the family Bauer--Furuta invariant $\FBafu(f,\mathfrak{s},\xi_{D},\xi_{d})$.\\

Now we study the dependence of the family Bauer--Furuta invariant from the choice of framings $\xi_{D}$ and $\xi_{d}$. First, we can express a point in $\widetilde{W}'$ as $(\theta, r, w, [u,v])$, where  $\theta\in S^1$, $r\in (0,R)$, $w\in \mathbb{R}^{2N+4k+3}$, $u\in S(\mathbb{C}^{M+2k+2})$ and $v\in \mathbb{C}^{M}$. Given a loop $\gamma_{1}$ in $SO(2N+4k+3)$ and a loop $\gamma_2$ in $U(M)$, we can define a homeomorphism 
$f_{\gamma_1,\gamma_2}: \widetilde{W}'\to \widetilde{W}'$
by 
\[
f_{\gamma_1,\gamma_2}(\theta, r, w, [u,v]):=(\theta, r, \gamma_{1}(\theta)w, [u,\gamma_{2}(\theta)v]).
\]
\begin{lemma}\label{lem: change of framing}
Assume $[\gamma_1]=a\in \mathbb{Z}/2 \cong \pi_{1}(SO(2N+4k+3))$ and $[\gamma_2]=b\in \mathbb{Z}\cong \pi_{1}(U(M))$. Then under the isomorphism (\ref{eq: cobordism group}), the induced map $f^*_{\gamma_1,\gamma_2}: \Omega_{1}^{\operatorname{fr}}(\widetilde{W}')\to \Omega_{1}^{\operatorname{fr}}(\widetilde{W}')$ is given by 
\[f^{*}_{\gamma_1,\gamma_2}(x,y)=(x,y+x(a+b))\in \mathbb{Z}\oplus \mathbb{Z}/2.\] 
\end{lemma}
\begin{proof}
Let $Y=S^1\hookrightarrow \widetilde{W}'$ be the submanifold defined by $\theta\mapsto (\theta,r,0,[u,0])$ for fixed $r,u$. Then we have a canonical framing $\xi_{c}$ on $Y$, pulled back from a trivialization of $T_{[u,0]}(S(\mathbb{C}^{M+2k+2})\times \mathbb{C}^{M})/S^1)$. Let $\xi'_{c}$ be the other framing on $Y$. Then under the isomorphism (\ref{eq: cobordism group}), we have 
\[
[Y,\xi_{c}]=(1,0) \quad \text{and} \quad [Y,\xi'_{c}]=(1,1).
\]
Note that the homeomorphism $f_{\gamma_{1},\gamma_{2}}$ fixes $Y$ pointwisely. It is straightforward to check that the differential of $f_{\gamma_{1},\gamma_{2}}$ preserves $\xi_{c}$ if and only if $a+b$ is even, which finishes the proof.
\end{proof}
Up to homotopy, any two framings of $\Det(\widetilde{D})$ differ by an integer $b\in \mathbb{Z}$, and any two framings of $\Det(\widetilde{d})$ differ by an element $a\in \mathbb{Z}/2$. So it makes sense to write $\xi_{d}+a$ and $\xi_{D}+b$.
\begin{proposition}\label{prop: change of framing formula} For any $a,b$, one has 
\[
\FBafu(f,\mathfrak{s},\xi_{D}+b,\xi_{d}+a)\equiv \FBafu(f,\mathfrak{s},\xi_{D},\xi_{d})+(a+b)\cdot \operatorname{SW}(X,\mathfrak{s})\mod 2.
\]
\end{proposition}
\begin{proof}
If we change the framing from $(\xi_{d},\xi_{D})$ to $(\xi_{d}+a,\xi_{D}+b)$, we need to compose the homeomorphism (\ref{eq: trivialization of tildeW}) by $f_{\gamma_{1},\gamma_{2}}$ for $[\gamma_{1}]=a$ and $[\gamma_{2}]=b$. So the lemma follows directly from  Lemma \ref{lem: change of framing}.
\end{proof}

The following vanishing theorem will be useful later.

\begin{proposition}\label{prop: FBF of boundary vanishing}
Let $X\hookrightarrow E_0\to \Sigma_0$ be a smooth bundle over a compact oriented surface $\Sigma_0$ with boundary components $\partial_{1}\Sigma_0,\cdots,\partial_{n}\Sigma_0$. Suppose $E_0|_{\partial_{i} \Sigma_0}$ is isomorphic to $T(f_{i})\to S^1$ as a smooth bundle. And suppose the family spin-c structure $\widetilde{\mathfrak{s}}_{i}$ and the framings $\xi^{i}_{d},\xi^{i}_{D}$ on $T(f_{i})$ can be extended a family spin-c structure $\widehat{\mathfrak{s}}$ and the framings $\widehat{\xi}_{d},\widehat{\xi}_{D}$ on $E_0$. Then one has 
\[
\sum^n_{i=1}\FBafu(f_{i},\mathfrak{s}_{i},\xi^{i}_{D},\xi^{i}_{d})=0.
\]
\end{proposition}
\begin{proof} To simplify the notation, we focus on the case $n=1$ and use $f,\mathfrak{s},\xi_{D},\xi_{d}$ to denote $f_1,\mathfrak{s}_1,\xi_{D}^1,\xi^1_{d}$. The general case is similar.

By repeating our constructions of $[\mathcal{M}_{sw},\xi_{c}]\in \Omega^{\operatorname{fr}}_{1}(\widetilde{W}')$, we see that the Seiberg-Witten moduli space for the family $E_0\to \Sigma_0$, denoted by $\widehat{\mathcal{M}}_{sw}$, is
an embedded submanifold of 
\[
\widehat{W}':=\Sigma_0\times \mathbb{R}^{2N+4k+3}\times  (0,R)\times (S(\mathbb{C}^{M+2k+2})\times \mathbb{C}^{M})/S^1   
\]

bounded by $\mathcal{M}_{sw}\hookrightarrow \partial \widehat{W}'=\widetilde{W}'$.  
To compute $\FBafu(f,\mathfrak{s},\xi_{D},\xi_{d})$, we repeat the construction in the proof of Lemma \ref{lem: cobordism computation}. Consider the embedded cobordism $F\hookrightarrow \widetilde{W}'$ from $\mathcal{M}_{sw}$ to $\sqcup Y(p_{i})$. 
Then we have 
\[
\FBafu(f,\mathfrak{s},\xi_{D},\xi_{d})=\langle w_{2}(NF, \xi_{c},\cup \xi_{i})\rangle.\]
Let $\widehat{Y}(p_{i})=F\times \{p_{i}\}$. Then we have a closed, oriented surface 
\[
\widehat{F}:=F\cup \widehat{\mathcal{M}}_{sw}\cup (\sqcup_{i} \widehat{Y}(p_{i}))\hookrightarrow \widehat{W}'.
\]
The canonical framing $\xi_{c}$ on $\mathcal{M}_{sw}$ can be extended to a canonical framing  $\widehat{\xi}_{c}$ on $\widehat{\mathcal{M}}_{sw}$. And its straightforward to see that the framing $\xi_{i}$ on $Y(p_{i})$ extends over $\widehat{Y}(p_{i})$. So we have 
\[
\langle w_{2}(NF, \xi_{c},\cup \xi_{i}), [F]\rangle =\langle w_{2}(N\widehat{F}),[\widehat{F}]\rangle\in \mathbb{Z}/2.
\]
On the other hand $w_{2}(T\widehat{W}')=w_{2}(T\widehat{F})=0$. So $ w_{2}(N\widehat{\Sigma})=0$ and the proof is finished.
\end{proof}

\subsection{The family Bauer-Furuta invariant of $\tau_{S}$}\label{section: FBF of tau} Let $S$ be a $(-2)$-sphere that pairs trivially with $c_{1}(\mathfrak{s})$, i.e. $\langle c_1(\mathfrak{s}), [S]\rangle= 0$. In this subsection, we study the family Bauer-Furuta invariant of the Dehn twist $\tau=\tau_S$. Consider 
\[
X\hookrightarrow E=T(\tau)\to S^1
\]
We now define canonical framings, denoted $\xi_{D}^S$ and $\xi_{D}^d$, on $\mathrm{det} (\widetilde{D}_{\widetilde{A}_0}^+ (E))$ and $H^+ (E)$.

Then we have a decomposition $E=E_1\cup E_2$ as families over $S^1$, where $E_1=S^1\times (X\setminus \nu(S))$ and $E_2=T(\tau|_{\nu(S)})$. We pick a family metric $\widetilde{g}$ that is trivial on $E_1$. 
We pick a family spin-c connection $\widetilde{A}_0$ that is constant on $E_1$ and spin on $E_2$. 

Consider the family Dirac operators $\widetilde{D}^+|_{E_1}$ and $\widetilde{D}^+|_{E_2}$, both equipped with Atiyah--Patodi--Singer (APS) boundary conditions. Then we have an isomorphism (natural up to homotopy):
\[
\Det(\widetilde{D}^{+})\cong \Det(\widetilde{D}^{+}|_{E_1})\otimes_{\mathbb{C}}  \Det(\widetilde{D}^{+}|_{E_2})
\]
Note that $\widetilde{D}^{+}|_{E_1}$ is a constant family of operators, so the index bundle has a canonical trivialization. On the other hand, there is a unique family spin structure on $E_{2}$ whose restriction on $\partial E_2$ is pulled back from $\partial \nu(S)$. This family spin structure induces $\mathfrak{s}|_{\nu(S)}$ on the fiber. Hence the family operator $\widetilde{D}^+|_{E_2}$ is canonically a family of quaternionic linear operators, so its index bundle  $\ind(\widetilde{D}^{+}|_{E_2})$ has structure group $\operatorname{Sp}(n)$. Since $\pi_{1}(\operatorname{Sp}(n))=0$,   the bundle $\ind(\widetilde{D}^+|_{E_2})$ also has a canonical trivialization up to homotopy.
Combining these two trivializations together, we obtain the canonical framing $\xi^{S}_{D}$.

To define the canonical framing $\xi^{S}_{d}$, we consider the family operators $\widetilde{d}|_{E_1}$ and $\widetilde{d}|_{E_2}$ and the isomorphism 
\[
\Det(\widetilde{d})\cong \Det(\widetilde{d}|_{E_1})\otimes  \Det(\widetilde{d}|_{E_2}).
\]
Again, $\widetilde{d}|_{E_1}$ is a constant family so its index bundle has a canonical trivialization. On the other hand, since $b^{+}_{2}(\nu(S))=b_{1}(\nu(S))=0$, the operator $(d^*,d^{+})$, with the APS boundary condition is invertible. So $\Det(\widetilde{d}|_{E_2})$ has a canonical trivialization. These two trivializations together give $\xi^{S}_{d}$. Of course, this is just an index-theoretic interpretation of the Dehn twist framing constructed in Definition \ref{definition:reflectionframingd}.\\

The aim of this subsection is to prove the following:

\begin{proposition}
\label{proposition: FBF vanishing original}
We have   $\FBafu(\tau_S,\mathfrak{s},\xi_{D}^S,\xi_{d}^S)=0$. 
\end{proposition}

Proposition~\ref{proposition: FBF vanishing original} follows from a gluing result, for which we need some preliminaries.
Put $W = D(\nu(S))$, which is diffeomorphic to the disk bundle of the complex line bundle $\mathcal{O}(-2) \to \CP^1$ of degree $-2$.
By our assumption, the restriction of $\mathfrak{s}$ to $W$ is the (unique) spin structure on $W$.
We consider the family relative Bauer-Furuta invariant of $(W, \mathfrak{s}, \tau)$. 
First, we recall the ordinary (i.e.\ non-family) relative Bauer-Furuta invariant $\Bafu(W, \mathfrak{s})$ of $(W, \mathfrak{s})$ defined by Manolescu \cite{ManolescuStablehomotopytype}.
Recalling that $\sigma(W) = -1$ and $b^+(W) = 0$, this invariant is given by an $S^1$-equivariant stable map of the form
\begin{align}
\label{eq: relative BF}
\Bafu(W, \mathfrak{s}) : S^{{M+1/8}\bc} \wedge S^{N\R}
\to S^{M\bc} \wedge S^{N\R} \wedge \SWF(\RP^3, \mathfrak{s}),
\end{align}
for $M, N > 0$, where $\SWF(\RP^3, \mathfrak{s})$ denotes the Seiberg-Witten stable Floer homotopy type defined in \cite{ManolescuStablehomotopytype} of $\RP^3$ with the spin structure obtained by restricting $\mathfrak{s}$ (denoted by the same symbol).

In fact, the existence of a positive scalar curvature metric $g_{\RP^3}$ allows us to construct the relative Bauer-Furuta invariant rather directly, without using the Seiberg-Witten stable Floer homotopy type. 
We shall describe the construction in the next subsection. 
In this subsection, let us simply clarify which stable homotopy set the relative Bauer–Furuta invariant lies in, and prove Proposition~\ref{proposition: FBF vanishing original}, assuming a few formal properties of the relative Bauer–Furuta invariant.
First, it follows that the domain and codomain of the relative Bauer-Furuta invariant of $(W, \mathfrak{s})$ are representation spheres of the same dimension.
Namely, we have
\begin{align}
\label{eq: set BF lies in}
\Bafu(W, \mathfrak{s}) \in [S^{0}, S^{0}]^{S^1}.
\end{align}

\begin{remark}
For readers who are familiar with the definition of the relative Bauer–Furuta invariant given in \cite{Manolescugluing}, \eqref{eq: set BF lies in} can be verified as follows.
First, since $g_{\RP^3}$ is a positive scalar curvature metric, the Floer homotopy type is given by
\[
\SWF(\RP^3, \mathfrak{s}) = [(S^0, 0, n(\RP^3, \mathfrak{s}, g_{\RP^3}))]
\]
in the notation of \cite{ManolescuStablehomotopytype}.
Here $n(\RP^3, \mathfrak{s},g_{\RP^3}) \in \mathbb{Q}$ is a quantity defined in \cite[Equation (6)]{ManolescuStablehomotopytype}, which is given by 
\[
n(\RP^3, \mathfrak{s}, g_{\RP^3})=1/8
\]
as explained in \cite[Subsection 7.1]{Manolescugluing}.
(In the notation of \cite[Subsection 7.1]{Manolescugluing}, $(\RP^3, \mathfrak{s})$ corresponds to $n=2$ and $k=1$.)
Thus \eqref{eq: set BF lies in} follows from \eqref{eq: relative BF}.
\end{remark}

Now we consider the family version. 
First, note that $\tau$ has exactly two lifts to automorphisms of the spin 4-manifold $(W, \mathfrak{s})$. 
Among these, there is exactly one lift that restricts to the identity on $(\partial W, \mathfrak{s})$. 
We denote this lift by $\tilde{\tau}$.
Then the mapping torus $T\tilde{\tau} \to S^1$ is a family of spin 4-manifolds with fiber $(W, \mathfrak{s})$, and the restriction of this family to the fiberwise boundary is the trivialized family $(\partial W, \mathfrak{s}) \times S^1$.
Associated to this family, we obtain the family relative Bauer-Furuta invariant, formulated as an $S^1$-equivariant stable map
\[
\FBafu(W, \tau, \mathfrak{s}, \xi_{D}^S,\xi_{d}^S) \in [S^{0}\wedge B_+, S^{0}]^{S^1},
\]
where $B=S^1$ is the base circle.
Just as in the closed 4-manifold case, the framings $\xi_{D}^S$ and $\xi_{d}^S$ are needed to regard the family relative Bauer-Furuta invariant as a map between spheres that are already trivialized.

Using this invariant, we can formulate the following gluing formula.
To record which 4-manifold we consider, let us denote $\FBafu(\tau,\mathfrak{s},\xi_{D}^S,\xi_{d}^S)$ in Proposition~\ref{proposition: FBF vanishing original} by $\FBafu(X,\tau,\mathfrak{s},\xi_{D}^S,\xi_{d}^S)$.

\begin{lemma}
\label{lem: gluing}
We have
\[
\FBafu(X,\tau,\mathfrak{s},\xi_{D}^S,\xi_{d}^S) = 
\FBafu(W, \tau, \mathfrak{s}, \xi_{D}^S,\xi_{d}^S) \wedge \baf(X,\mathfrak{s}).
\]
\end{lemma}


Another formal property is the following vanishing result:

\begin{lemma}
\label{lem: BF type vanishing}
We have $\FBafu(W, \mathfrak{s},\tau, \xi_{D}^S,\xi_{d}^S)=0$.    
\end{lemma}

\begin{proof}[Proof of Proposition~\ref{proposition: FBF vanishing original}]
This follows immediately from Lemmas~\ref{lem: gluing} and \ref{lem: BF type vanishing}.
\end{proof}

Thus, to establish Proposition~\ref{proposition: FBF vanishing original}, it remains to prove Lemmas~\ref{lem: gluing} and \ref{lem: BF type vanishing}.
We prove Lemma~\ref{lem: BF type vanishing} in the next subsection, and Lemma~\ref{lem: gluing} in Subsection~\ref{subsection Excision}.

\subsection{Relative Bauer-Furuta invariant for psc boundary}

Let $(Z, \mathfrak{s})$ be a compact smooth spin-c 4-manifold with $b_1(Z) = 0$ and $b_1(\del Z) = 0$, whose boundary $Y=\partial Z$ is equipped with a positive scalar curvature metric $g$. 
The relative Bauer-Furuta invariant for $(Z, \mathfrak{s})$ is then constructed by slightly modifying the construction of the Bauer-Furuta invariant for closed 4-manifolds \cite{BauerFurutaI}.
We describe the necessary modifications below. 
Our construction follows a common procedure for obtaining a finite-dimensional approximation on a non-compact 4-manifold, provided that the moduli space is compact. 
Specifically, we follow the construction of the Bauer-Furuta counterpart of Kronheimer--Mrowka's invariant for 4-manifolds with contact boundary, due to Iida \cite{Iida}. 
As in \cite{Iida}, we construct a finite-dimensional approximation following Furuta's argument~\cite{Furuta-10-8}.

Let $\hat{Z}$ be a cylindrical 4-manifold obtained from $Z$:
\[
\hat{Z} = Z \cup (Y \times [0,\infty)).
\]
Fix a metric on $\hat{Z}$ that restricts on the cylindrical end to the product of $g$ with the standard metric on $[0, \infty)$.
On $\hat{Z}$, rather than the ordinary Sobolev spaces, we work with weighted Sobolev spaces. 
This is to make the quadratic term in the Seiberg–Witten equations a compact operator (see \cite[Lemma 2.1]{Iida}), despite the absence of Rellich's theorem.
Take a smooth function $\sigma \colon \hat{Z} \to \R$ that restricts to $\sigma(y, t) = t$ on $Y \times [0, \infty)$.
Let $\alpha > 0$ be a real number such that there are no eigenvalues in $(0, \alpha)$ for the Dirac operator and the operator $d^\ast$ on $\partial Z$.
Fix $k>3$, and consider the weighted Sobolev space $L^2_{k,\alpha}(\hat{Z})$, defined as $e^{-\alpha \sigma} L^2_{k}(\hat{Z})$.
The Seiberg-Witten map in the weighted Sobolev setup is a map of the form
\[
SW : L^2_{k,\alpha}(\hat{Z};\Lambda^1\oplus S^+)
\to L^2_{k-1,\alpha}(\hat{Z};\Lambda^+\oplus S^-).
\]

Using the assumption that $g$ is a positive scalar curvature metric on $Y$, the Seiberg-Witten moduli space for $\hat{Z}$ under the $L^2$-decay condition is compact \cite[Corollary 4.4.16]{Nicolaescu-book}.
By a standard elliptic regularity argument, this implies that the moduli space defined in the weighted Sobolev setup is compact as well.

Also in the weighted Sobolev setup, we have the global slice for the based gauge group, given by
\[
\mathcal{W}^+ = 
\ker(d^{\ast,\alpha} : L^2_{k,\alpha}(\hat{Z};\Lambda^1) \to L^2_{k-1,\alpha}(\hat{Z};\Lambda^0)) \oplus L^2_{k,\alpha}(S^+),
\]
just as in \cite[Proposition 3.5]{Iida}, where $d^{\ast,\alpha}$ is the adjoint of $d$ with respect to $L^2_{\alpha}$.
Consider the Seiberg-Witten map restricted to this global slice.
The zero set $(SW|_{\mathcal{W}^+})^{-1}(0)$, i.e. the framed moduli space, is compact, thanks to the compactness of the moduli space mentioned above.

By this compactness, there exists $R > 0$ such that
\[
(SW|_{\mathcal{W}^+})^{-1}(0) \subset B_R(\mathcal{W}^+),
\]
where $B_R(\mathcal{W}^+)$ denotes the ball in $\mathcal{W}^+$ of radius $R$ centered at the origin in $\mathcal{W}^+$.
Set
\[
\mathcal{W}^- = L^2_{k-1,\alpha}(\hat{Z}; \Lambda^+ \oplus S^-)
\]
and denote by $S_R(\mathcal{W}^+)$ the sphere of radius $R$ centered at the origin in $\mathcal{W}^+$.
Then we have:

\begin{lemma}
\label{lem: does not hit the ball}
There exists a small $\epsilon>0$ such that
$SW(S_R(\mathcal{W}^+)) \cap B_\epsilon(\mathcal{W}^-) = \emptyset$.
\end{lemma}

\begin{proof}
The proof is completely analogous to that of \cite[Proposition 3.12]{Iida}.
The fact that the quadratic part of the Seiberg–Witten map is a compact operator is used in the proof of this lemma.
\end{proof}

Let $L : \mathcal{W}^+ \to \mathcal{W}^-$ denote the linear Fredholm operator given by the linear part of the map $SW|_{\mathcal{W}^+}$, and let $C = SW|_{\mathcal{W}^+} - L$ be the quadratic part.
Let $\{W_n\}_n$ be an increasing sequence
\[
W_1 \subset W_2 \subset \cdots \subset \mathcal{W}^-
\]
of finite-dimensional subspaces of $\mathcal{W}^-$ with $(\mathrm{Im}(L))^{\perp} \subset W_n$, where $(\mathrm{Im}(L))^{\perp}$ denotes the orthogonal complement of $\mathrm{Im}(L)$ in $\mathcal{W}^-$ with respect to the $L^2_{k-1,\alpha}$-inner product.
For each $n$, let $p_n \colon \mathcal{W}^- \to W_n$ denote the $L^2_{k-1,\alpha}$-orthogonal projection.

\begin{lemma}
\label{lem: quad term}
Assume that 
$p_n$ regarded as maps $p_n :  \mathcal{W}^- \to \mathcal{W}^-$ converge to the identity map on $\mathcal{W}^-$ in the strong operator topology as $n \to +\infty$.
Then there exists $N>0$ such that, for every $n \geq N$, we have 
\[
\|(\id - p_n)(C(v))\|_{L^2_{k-1,\alpha}} < \epsilon
\]
for any $v \in S_{R}(\mathcal{W}^+)$.
\end{lemma}

\begin{proof}
The proof is completely analogous to that of \cite[Proposition 3.13]{Iida}\footnote{In the proof of \cite[Proposition 3.13]{Iida}, for a sequence $\{v_n\}_n$ in the sphere in the domain of the Seiberg–Witten map, it is asserted that $C(v_n)$ converges strongly to $C(v_\infty)$ after passing to a subsequence, using the weak convergence of $\{v_n\}$. 
This does not follow in general for a \emph{non-linear} compact operator $C$. 
However, the compactness of $C$ and the boundedness of $\{v_n\}_n$ do imply that, after passing to a subsequence, $C(v_n)$ converges to some element in the codomain of the Seiberg–Witten map. 
The proof of \cite[Proposition 3.13]{Iida} relies only on this fact, so the argument is correct once this modification is made.}
The fact that $C$ is a compact operator is used in the proof of this lemma as well.
\end{proof}

By Lemmas~\ref{lem: does not hit the ball} and \ref{lem: quad term}, we can repeat the construction of a finite-dimensional approximation as in the closed case \cite{Furuta-10-8}:
for a sufficiently large finite-dimensional subspace $W_n \subset \mathcal{W}^-$ discussed in Lemma~\ref{lem: quad term}, it follows from Lemmas~\ref{lem: does not hit the ball} and \ref{lem: quad term} that
\[
(L+p_nC)(S_R(L^{-1}(W_n))) \neq 0.
\]
Thus, we obtain a map of pairs
\[
L+p_nC : (B_R(L^{-1}(W_n)), S_R(L^{-1}(W_n))) \to (W_n,W_n \setminus \{0\}),
\]
which is equivalent (up to homotopy) to a based map
\[
S^{L^{-1}(W_n)} \to S^{W_n}
\]
between representation spheres.
The stable homotopy class of this map is the relative Bauer-Furuta invariant $\Bafu(Z, \mathfrak{s})$ of $(Z,\mathfrak{s})$.
By computing the index of $L$,  we have that this invariant lies in the following stable homotopy set:
\begin{align*}
\Bafu(Z, \mathfrak{s}) \in [S^{\frac{c_1(\mathfrak{s}^2)-\sigma(Z)}{8}\bc},\, S^{b^+(Z) \R} \wedge S^{n(Y, \mathfrak{s}, g)\bc}]^{S^1},
\end{align*}
where $n(Y, \mathfrak{s},g) \in \mathbb{Q}$ is a quantity that appears in \cite[Equation (6)]{ManolescuStablehomotopytype}.

\begin{remark}
We need this relative invariant for a gluing along $\RP^3$ equipped with the standard positive scalar curvature metric. 
Therefore, we do not need the independence of $\Bafu(Z, \mathfrak{s})$ with respect to the boundary metric $g$, and it suffices to treat $\Bafu(Z, \mathfrak{s})$ as an invariant of the triple $(Z, \mathfrak{s}, g)$.
In this case, the proof of the invariance of $\Bafu(Z, \mathfrak{s})$ (i.e. the independence of the choice of metric on $Z$ extending $g$ and of finite-dimensional approximation) is completely analogous to the closed 4-manifold case \cite{BauerFurutaI}.
\end{remark}

Given a diffeomorphism $f \colon Z \to Z$ with $f|_{\partial Z} = \id$ and $f^\ast \mathfrak{s} = \mathfrak{s}$, if we pick framings $\xi_D$ and $\xi_d$ for the mapping torus $Tf \to S^1$, it is evident that the family version of the relative Bauer-Furuta invariant 
\[
\FBafu(Z, f, \mathfrak{s}, \xi_D, \xi_d)
\in [S^{\frac{c_1(\mathfrak{s})^2 - \sigma(Z)}{8}\bc} \wedge B_+,\, S^{b^+(Z)\R} \wedge S^{n(Y, \mathfrak{s}, g)\bc}]^{S^1}
\]
is defined in the same way as in the closed 4-manifold case, where $B = S^1$.

Now we give the proof of the vanishing result, Lemma~\ref{lem: BF type vanishing}: 

\begin{proof}[Proof of Lemma~\ref{lem: BF type vanishing}]
Recall that $\FBafu(W, \mathfrak{s}, \tau, \xi_{D}^S, \xi_{d}^S)$ is the invariant associated to the spin family $T\tilde{\tau} \to S^1$. 
Thus, the family relative Bauer-Furuta invariant is in fact $\Pin(2)$-equivariant, not just $S^1$-equivariant.
For brevity and distinction, let
\[
\Psi^{S^1} \in [S^0 \wedge B_+, S^0]^{S^1}
\]
and
\[
\Psi^{\Pin(2)} \in [S^0 \wedge B_+, S^0]^{\Pin(2)}
\]
denote the $S^1$- and $\Pin(2)$-equivariant family relative Bauer-Furuta invariants of $(W, \mathfrak{s},\tau, \xi_{D}^S,\xi_{d}^S)$, respectively.
We see that $\Psi^{\Pin(2)}$ is of BF-type in the terminology of \cite{LinMukherjee} by repeating the proof of \cite[Lemma 5.2]{LinMukherjee} in the relative setup.
The proof of Case (1) of \cite[Theorem 1.9]{LinMukherjee} shows that a map of BF-type lying in $[S^0\wedge B_+,S^0]^{\Pin(2)}$ restricts to the trivial element in $[S^0\wedge B_+,S^0]^{S^1}$, showing that $\Psi^{S^1} = 0$ in $[S^0 \wedge B_+, S^0]^{S^1}$.
This proves the lemma.
\end{proof}

\subsection{Excision along $\RP^3$}
\label{subsection Excision}

To prove the desired gluing result, Lemma~\ref{lem: gluing}, we need to consider a gluing along $\RP^3$ in the family and relative setting.
It is a straightforward generalization of a gluing (or excision) result along $\RP^3$ due to Bauer~\cite{Bauer-survey} in the unparameterized and closed setting, which is a variant of his connected sum formula~\cite{Bauer-gluing} for the Bauer-Furuta invariant.

First, we review Bauer's excision.
Let $Z_0, Z_1$ be compact oriented smooth 4-manifolds with boundary $\del Z_0 \cong \del Z_1 \cong \RP^3$ as oriented manifolds.
By using an orientation-reversing diffeomorphism $\varphi : \RP^3 \to \RP^3$, one can glue $Z_0$ and $Z_1$ along $\RP^3$. 
Let $Z_0 \#_{P} Z_1$ denote the resulting 4-manifold.

An important example of such $Z_i$ is the disk bundle $W=D(\mathcal{O}(-2))$ of the complex line bundle $\mathcal{O}(-2) \to \CP^1$ of degree $-2$.
Note that $\RP^3$ admits two spin structures, $\mathfrak{t}_0$ and $\mathfrak{t}_1$. 
One of them, say $\mathfrak{t}_0$, extends to a spin structure on $W$, while the other, $\mathfrak{t}_1$, does not extend to a spin structure on $W$. 
The orientation-reversing diffeomorphism $\varphi$ interchanges these two spin structures. 
In fact, the manifold $W \#_{P} W$ is not spin and is diffeomorphic to $\#_2 \overline{\CP}^2$.

Assume that spin-c structures $\mathfrak{s}_i$ are given on $Z_i$, and suppose that $\mathfrak{s}_i|_{\del Z_i} = \mathfrak{t}_i$.
Then we can glue $(Z_0,\mathfrak{s}_0)$ with $(Z_1,\mathfrak{s}_1)$ using $\varphi$, and obtain a closed spin-c 4-manifold
$(Z_0\#_{P} Z_1, \mathfrak{s}_0\#_P\mathfrak{s}_1)$.
Now suppose further that we are given similar tuples 
\[
Z_0',\ Z_1',\ \mathfrak{s}_0',\ \mathfrak{s}_1'
\]
and 
\[
Z_0'',\ Z_1'',\ \mathfrak{s}_0'',\ \mathfrak{s}_1''
\]
as in the version without primes.
Under this setup, the excision along $\RP^3$ can be stated as follows:

\begin{theorem}[{\cite[Proof of Theorem 8.4]{Bauer-survey}}]
\label{thm: Bauer gluing RP3}
We have 
\begin{align*}
&\Bafu(Z_0\#_{P} Z_1, \mathfrak{s}_0\#_P\mathfrak{s}_1)
\wedge \Bafu(Z_0'\#_{P} Z_1', \mathfrak{s}_0'\#_P\mathfrak{s}_1')
\wedge \Bafu(Z_0''\#_{P} Z_1'', \mathfrak{s}_0''\#_P\mathfrak{s}_1'')\\
=&
\Bafu(Z_0\#_{P} Z_1', \mathfrak{s}_0\#_P\mathfrak{s}_1')
\wedge \Bafu(Z_0'\#_{P} Z_1'', \mathfrak{s}_0'\#_P\mathfrak{s}_1'')
\wedge \Bafu(Z_0''\#_{P} Z_1, \mathfrak{s}_0''\#_P\mathfrak{s}_1).
\end{align*}
\end{theorem}

For readers' convenience, we briefly review the proof of Theorem~\ref{thm: Bauer gluing RP3}.
The central step is to construct an ($S^1$-equivariant) homotopy from a finite-dimensional approximation $\operatorname{SW_{apr}^+}$ for
\begin{align}
\label{eq: sums 1}
(Z_0\#_{P} Z_1, \mathfrak{s}_0\#_P\mathfrak{s}_1)
\sqcup (Z_0'\#_{P} Z_1', \mathfrak{s}_0'\#_P\mathfrak{s}_1')
\sqcup (Z_0''\#_{P} Z_1'', \mathfrak{s}_0''\#_P\mathfrak{s}_1'')
\end{align}
to a finite-dimensional approximation for
\begin{align}
\label{eq: sums 2}
(Z_0\#_{P} Z_1', \mathfrak{s}_0\#_P\mathfrak{s}_1')
\sqcup (Z_0'\#_{P} Z_1'', \mathfrak{s}_0'\#_P\mathfrak{s}_1'')
\sqcup (Z_0''\#_{P} Z_1, \mathfrak{s}_0''\#_P\mathfrak{s}_1).
\end{align}
The components are switched by the cyclic permutation $\sigma$ of order 3.
This homotopy is constructed by cutting and pasting the configurations (i.e. differential forms and spinors) using a cut-off function and a path in $SO(3)$ from the identity to $\sigma$, regarded as an element of $SO(3)$.
(Note that $\sigma$ is an even permutation, so it lies in $SO(3)$.)
Precisely, we isometrically embed a neck $\RP^3 \times [-L,L]$ for $L>0$ into each sum along $\RP^3$, and let $r : \RP^3 \times [-L,L] \to [0,1]$ be a smooth function with
\begin{align*}
r|_{\RP^3 \times [-L,-1]}\equiv 0,\quad r|_{\RP^3 \times [1,L]}\equiv 1.
\end{align*}
Let $\psi : [0,1] \to SO(3)$ be a path from the identity to the permutation $\sigma$.
For $\vec{e} = (e_1, e_2, e_3) \in \bigoplus_3 \Gamma(\RP^3; \Lambda^\ast T\RP^3 \oplus S)$, where $S$ is the spinor bundle, set
\[
\vec{e^\sigma} = (\psi \circ r) \cdot \vec{e}.
\]
Applying this construction to the configurations on the cylinder $\RP^3 \times [-L, L]$ while keeping the other parts unchanged, we obtain an isomorphism from the configuration space for \eqref{eq: sums 1} to the configuration space for \eqref{eq: sums 2}.
The main assertion of the excision is that this isomorphism induces an identification of the Bauer-Furuta invariant for \eqref{eq: sums 1} with that for \eqref{eq: sums 2}, which is proved by making explicit homotopies.
Positivity of scalar and Ricci curvature of $\RP^3$ along the neck provides the necessary estimates during the homotopy.

From Theorem~\ref{thm: Bauer gluing RP3}, Bauer deduced a sum formula for the Bauer-Furuta invariant along $\RP^3$ \cite[Theorem 8.4]{Bauer-survey}.
In that deduction, the following fact is used, which is easily deduced from $b^+(\#_2\overline{\CP}^2)=0$ together with a homotopy-theoretic lemma \cite[Lemma 3.8]{BauerFurutaI} that determines the homotopy class of an $S^1$-equivariant map from the $S^1$-invariant-part map in this setting.
Let $\mathfrak{s}_i^W$ be spin-c structures that are extensions of $\mathfrak{t}_i$ to $W$ respectively, with $\mathfrak{s}_0^W$ spin and  $\mathfrak{s}_1^W$ non-spin.
As we noted, $W\#_PW \cong \#_2\overline{\CP}^2$.

\begin{lemma}
\label{lem: CPbar id}
We have
\[
\Bafu(\#_2\overline{\CP}^2, \mathfrak{s}_0^W\#_P \mathfrak{s}_1^W)
=[\mathrm{id}].
\]
\end{lemma}

There is also a relative version of Lemma~\ref{lem: CPbar id}:

\begin{lemma}
\label{lem: identity}
We have 
\[
\Bafu(W, \mathfrak{s}_0^W) =[\mathrm{id}].
\]
\end{lemma}

\begin{proof}
This follows from the fact that the $S^1$-invariant-part map $\Bafu(W, \mathfrak{s}_0^W)^{S^1}$ is represented by the identity map since $b^+(W) = 0$, together with \cite[Lemma 3.8]{BauerFurutaI}.
\end{proof}

We need relative and family versions of Theorem~\ref{thm: Bauer gluing RP3}. 
Let us begin with the relative version, which is formulated as follows:

\begin{theorem}
\label{thm: Bauer gluing RP3 relative}
We have 
\begin{align*}
&\Bafu(Z_0, \mathfrak{s}_0)
\wedge \Bafu(Z_0'\#_{P} Z_1', \mathfrak{s}_0'\#_P\mathfrak{s}_1')
\wedge \Bafu(Z_0''\#_{P} Z_1'', \mathfrak{s}_0''\#_P\mathfrak{s}_1'')\\
=&
\Bafu(Z_0\#_{P} Z_1', \mathfrak{s}_0\#_P\mathfrak{s}_1')
\wedge \Bafu(Z_0'\#_{P} Z_1'', \mathfrak{s}_0'\#_P\mathfrak{s}_1'')
\wedge \Bafu(Z_0'', \mathfrak{s}_0'').
\end{align*}
\end{theorem}

\begin{proof}
Consider the cylindrical-end manifolds $\hat{Z}_0$ and $\hat{Z}_0''$ constructed from $Z_0$ and $Z_0''$, respectively.
We can regard the neck $\RP^3 \times [-L, L]$ as embedded into $\hat{Z}_0$ and $\hat{Z}_0''$ by identifying $[-L, L]$ with $[0, 2L] \subset [0, \infty)$.
Then the excision process used in the proof of Theorem~\ref{thm: Bauer gluing RP3} described above works without any change, formally by simply putting $0$ as a configuration for $Z_1 = \emptyset$.
Thus we obtain a homotopy from a finite-dimensional approximation for
\begin{align*}
(\hat{Z}_0, \mathfrak{s}_0)
\sqcup (Z_0'\#_{P} Z_1', \mathfrak{s}_0'\#_P\mathfrak{s}_1')
\sqcup (Z_0''\#_{P} Z_1'', \mathfrak{s}_0''\#_P\mathfrak{s}_1'')
\end{align*}
to a finite-dimensional approximation for
\begin{align*}
(Z_0\#_{P} Z_1', \mathfrak{s}_0\#_P\mathfrak{s}_1')
\sqcup (Z_0'\#_{P} Z_1'', \mathfrak{s}_0'\#_P\mathfrak{s}_1'')
\sqcup (\hat{Z}_0'', \mathfrak{s}_0'').
\end{align*}
This proves the assertion.
\end{proof}

Next, let us consider a family version of Theorem~\ref{thm: Bauer gluing RP3 relative}.
Given a diffeomorphism $f \colon Z_0 \to Z$ with $f|_{\partial Z_0} = \id$ and $f^\ast \mathfrak{s}_0 = \mathfrak{s}_0$, pick framings $\xi_D$ and $\xi_d$ for the mapping torus $Tf \to S^1$.

\begin{theorem}
\label{thm: family gluing}
We have 
\begin{align*}
&\FBafu(Z_0, f, \mathfrak{s}_0,\xi_{D},\xi_{d})
\wedge \Bafu(Z_0'\#_{P} Z_1', \mathfrak{s}_0'\#_P\mathfrak{s}_1')
\wedge \Bafu(Z_0''\#_{P} Z_1'', \mathfrak{s}_0''\#_P\mathfrak{s}_1'')\\
=&
\FBafu(Z_0\#_{P} Z_1', f, \mathfrak{s}_0\#_P\mathfrak{s}_1', \xi_{D},\xi_{d})
\wedge \Bafu(Z_0'\#_{P} Z_1'', \mathfrak{s}_0'\#_P\mathfrak{s}_1'')
\wedge \Bafu(Z_0'', \mathfrak{s}_0'').
\end{align*}
\end{theorem}

\begin{proof}
As described in \cite[Proof of Proposition 5.1]{KM-dehn}, there is no difficulty to generalize Bauer's connected sum formula for a families setup.
Similarly, the proof of the assertion is a straightforward generalization of the proof of Theorem~\ref{thm: Bauer gluing RP3 relative}, so we just briefly summarize the argument following \cite[Proof of Proposition 5.1]{KM-dehn}.

For a disjoint union, the Seiberg-Witten map and its finite-dimensional approximation are defined to be the fiber product over $B=S^1$.
The homotopy between finite-dimensional approximations used in the proof of Theorem~\ref{thm: Bauer gluing RP3 relative} can be applied fiberwise over $B$, since all the estimates in \cite{Bauer-gluing} can be made uniformly over the compact base.
Thus, we get a proper homotopy between finite-dimensional approximations, regarded as bundle maps over $B$.
This gives the desired equality in the assertion. 
\end{proof}

Now we can deduce the desired gluing, Lemma~\ref{lem: gluing}:

\begin{proof}[Proof of Lemma~\ref{lem: gluing}]
Applying Theorem~\ref{thm: family gluing} to
\begin{align*}
&Z_0 = W,\quad \mathfrak{s}_0 = \mathfrak{s}_0^W,\quad f=\tau,\quad \xi_D=\xi_D^S,\quad \xi_d=\xi_d^S,\\
&Z_0' = W,\quad \mathfrak{s}_0 = \mathfrak{s}_0^W,\quad Z_1' = X \setminus \nu(S), \quad \mathfrak{s}_1' = \mathfrak{s}|_{Z_1'},\\
&Z_0'' = W,\quad \mathfrak{s}_0 = \mathfrak{s}_0^W,\quad Z_1'' = W, \quad \mathfrak{s}_1''=\mathfrak{s}_1^W,
\end{align*}
the assertion then follows immediately from Lemmas~\ref{lem: CPbar id} and \ref{lem: identity}.
\end{proof}

To this end, we have established Lemma~\ref{lem: gluing}, hence Proposition~\ref{proposition: FBF vanishing original}.


\section{Proof of the main theorem}
\label{sectionProof of the main theorem}
In this section, we prove the main Theorems (Theorem \ref{thm: main}, Theorem \ref{thm: main generalized}).\\

We start by fixing some geometric data. Let $f: E\to\Sigma$ be a smooth Lefschetz fibration over a closed, oriented surface $\Sigma$. For $b\in \Sigma$, we use $X_b$ to denote the fiber over $b$. We pick a regular value $b_0\in \Sigma$ and use $X$ to denote the fiber $X_{b_0}$. We use $p_{1},\cdots,p_{n}\in E$ to denote the singular points. We fix a path $\gamma_{i}:I\to \Sigma$ from $b_0$ to $f( p_{i} )$.

Under a local chart  $\psi: U_{i}\cong \{|z_{1}|^2+|z_{2}|^2 +|z_3|^2 <1\}$ near $p_i$ and a local chart $\phi:V_i\cong \{|z|<1\}$ near $f(p_i)$, the map $f$ can be written as $f(z_{1},z_2, z_3)=z^2_1+z^2_2 +z_{3}^2$. We pick small $\epsilon$ and let $\Sigma_{i}=\psi^{-1}(\{|z|<\epsilon\})\subset \Sigma$.  Then we have the decomposition 
\[
\Sigma=\Sigma_0\cup \Sigma_1\cup\cdots \cup \Sigma_n,
\]
where $\Sigma_0=\Sigma\setminus (\cup_{1\leq i\leq n}\mathring{\Sigma}_{i})$. For each $0\leq i\leq k$, we let $E_{i}=f^{-1}(\Sigma_i)$. 
We let $f_{i}:E_{i}\to \Sigma_i$ be the restriction of $f$. Then $f_0: E_0\to\Sigma_0$ is a smooth bundle with fiber $X$. 
For $1\leq i\leq n$, let $D_{i}=E_{i}\cap U_{i}$, $E^{\circ }_{i}=E_{i}\setminus \mathring{D}_{i}$ and $E^{\circ}=E\setminus \cup^{n}_{i=1}\mathring{D}_{i}$. Then $f_{i}|_{E^{\circ}_{i}}: E^{\circ }_i\to \Sigma_i$ is also a smooth bundle. Since $\Sigma_i$ is contractible, we have a trivialization 
\begin{equation}\label{eq: product Ei}
E^{\circ }_{i}\cong \Sigma_{i}\times X_i \end{equation}
Here $X_{i}=X\setminus \nu(S_{i})$, where $S_{i}\hookrightarrow X$ denotes the vanishing cycle for $p_{i}$ (along $\gamma_{i}$). 
Note that $E^{\circ}$ is a smooth manifold-with-corners, and $f|_{E^{\circ}}$ is a submersion. Let $T^{V}E^{\circ}:=\ker((f|_{E^{\circ}})_*: TE^{\circ}\to T\Sigma)$ be the vertical tangent bundle, and let $T^{H}E^{\circ}: =(f|_{E^{\circ}})^*(T\Sigma)$. We fix a splitting 
\begin{equation}\label{eq: splitting of tangent bundles}
\mathcal{H}_{E^{\circ}}:TE^{\circ}\xrightarrow{\cong} T^{V}E^{\circ}\oplus T^{H}E^{\circ}.    
\end{equation}
that is compatible with the trivialization (\ref{eq: product Ei}). Next, we pick a Riemmannian metric $g_E$ on $E$ that satisfies the following conditions: (i) $\mathcal{H}$ is an orthogonal decomposition. (ii) $g_E|_{T^{H}E^{\circ}}$ is pulled back from a metric on $\Sigma$. (iii) $g_E|_{E^{\circ}_{i}}$ is a product metric with respect to (\ref{eq: product Ei}). 

Next, we fix a spin structure $\mathfrak{s}_{\Sigma}$ on $\Sigma$ and a spin-c structure $\mathfrak{s}_{E}$ on $E$. 
Note the pull-back square 
\begin{equation}\label{diagram: spinc}
\xymatrix{Spin^{c}(4)\times Spin(2)\ar[d]\ar[rr] & & Spin^{c}(6)\ar[d]\\SO(4)\times Spin(2)\ar[r] & SO(4)\times SO(2)\ar[r] &SO(6)}.
\end{equation}
Thus, the spin structure on $T^{H}E^{\circ }$ given by pulling back $\mathfrak{s}_\Sigma$ together with the spin-c structure $\mathfrak{s}_{E^{\circ}}$ restricted from $\mathfrak{s}_{E}$ determine a spin-c structure $\widetilde{\mathfrak{s}}$ on $T^{V}E^{\circ}$. Let $S^{+}_{E}\to E$ be the spinor bundle over $E$. We fix a unitary connection $A^{t}_{E}$ on the line bundle $\operatorname{det}(S^+_{E})$. This induces a spin-c connection $A_{E}$ on $S^{+}_{E}$.
For various submanifolds $M\hookrightarrow E$, we use  $A^{t}_{M}$ to denote the restriction $A^{t}_{E}|_{M}$. We pick $A^{t}_{E}$
such that the following two conditions hold: (1) $A^{t}_{D_{i}}$ is flat for any $1\leq i\leq n$; (2) $A^{t}_{E^{\circ}_i}$ is pulled back from some connection on  $X_{i}$ under the decomposition (\ref{eq: product Ei}). 

For each $b\in \Sigma_0$, the spin-c structure $\widetilde{\mathfrak{s}}$ restricts to a spin-c structure $\mathfrak{s}_{b}$ on $X_b$. We denote the spinor bundle by $S^{+}_{X_b}$. We use $\mathfrak{s}_{X}$ to denote $\mathfrak{s}_{X_{b_0}}$. Since $\mathfrak{s}_{\Sigma}$ is spin, we have  have a canonical isomorphism $\operatorname{det}(S^{+}_{X_b})\cong \operatorname{det}(S^+_{E})|_{X_b}$. Hence the connection $A^{t}_{X_{b}}$ on $\operatorname{det}(S^+_{E})|_{X_b}$ can also be viewed as a connection on $\operatorname{det}(S^{+}_{X_b})$. It further induces a spin-c connection $A_{b}$ on $(X_{b},\mathfrak{s}_{b})$. Thus, we obtain a family of Dirac operators 
\[
\widetilde{D}^{+}(E_0)=\{D^{+}_{A_{b}}(X_b):\Gamma(S^{+}_{b})\to \Gamma(S^{-}_{b})\}_{b\in \Sigma_0}.
\]
We will be interested in its determinant line bundle:
\begin{equation}\label{eq: determinant bundle over Sigma0}
\Det(\widetilde{D}^{+}(E_0))\to \Sigma_0 .   
\end{equation}
 Note the decomposition \[\partial \Sigma_0=\bigsqcup_{1\leq i\leq n}\partial_{i} \Sigma_0,\]
where $\partial_{i}\Sigma_0=\partial \Sigma_i$. For each $1\leq k\leq n$, the restriction of the bundle $E_0\to \Sigma_0$ to $\partial_i \Sigma_0$ is isomorphic to the mapping torus $T(\tau_{S_i})\to \partial_{i}\Sigma_0$ of the Dehn twist $\tau_{S_i}$. So we also use $\Det(\widetilde{D}^{+}(T(\tau_{S_i})))$ to denote the restriction of the bundle (\ref{eq: determinant bundle over Sigma0}) to $\partial_i\Sigma_0$. By our discussion in Section \ref{section: FBF of tau}, the bundle $\Det(\widetilde{D}^{+}(T(\tau_{S_i})))$ has a canonical framing $\xi^{S_i}_{D}$. The following proposition is a key step in our proof.

\begin{proposition}\label{prop: relative Chern class equals index}
We have $\langle c_{1}(\operatorname{det}(\widetilde{D}^{+}(E_0)), \xi^{S_1}_{D},\cdots, \xi^{S_n}_{D}),[\Sigma_0]\rangle=\operatorname{ind}(D^{+}(E,\mathfrak{s}_E))$
\end{proposition}
To prove Proposition \ref{prop: relative Chern class equals index}, we use the \textit{Local Index Theorem} proved by Bismut--Freed\cite{Bismut-Freed}. As discussed in \cite{freed1987determinant}, there is a canonical Hermitian metric on $\Det(\widetilde{D}^{+}(E_0))$. Furthermore, the splitting (\ref{eq: splitting of tangent bundles}) induces a canonically defined unitary connection $\nabla$ on $\Det(\widetilde{D}^{+}(E_0))$, called the \textit{Bismut connection}. The only property about the Bismut connection that we shall need is the following local index theorem.
Let 
\[c_1(\mathfrak{s}_{E_0})=\frac{i}{2\pi}F_{A^{t}_{E_0}}\in\Omega^{2}(E_0)\] be the Chern form of $\Det(S^{+}_{E_0})$ and let \[p_{1}(T^{V}E)=p_{1}(T^{V}E_0,g|_{T^{V}E_0})\in \Omega^4(E_0)\] be the Pontryagin form of the vertical tangent bundle. Then the local index formula states
\begin{equation}\label{eq: local index formula}
\begin{split}
\frac{i}{2\pi}F_{\nabla}&=\int_{E_0 /\Sigma_0}\hat{A}(T^{V}E_0)\cdot e^{\frac{c_1(\mathfrak{s}_{E_0})}{2}}\\&=\frac{1}{48}\int_{E_0/\Sigma_0}(p_{1}(T^{V}E)\wedge c_1(\mathfrak{s}_{E_0})-c_1(\mathfrak{s}_{E_0})\wedge c_1(\mathfrak{s}_{E_0})\wedge c_1(\mathfrak{s}_{E_0}))\in \Omega^2(\Sigma_0)    
\end{split}
\end{equation}
where integration is along the fibers of $E_0 \to \Sigma_0$.

For $1\leq i\leq n$, we use $\operatorname{hol}_{\nabla}(\xi^{S_i}_{D})\in \mathbb{R}$ to denote the holonomy of the Bismut connection $\nabla$ on  $\Det(\widetilde{D}^{+}(T(\tau_{S_i})))$, under the framing $\xi^{S_i}_{D}$.

\begin{lemma}\label{lem: holonomy equals 0} For $1\leq i \leq n$, we have $\operatorname{hol}_{\nabla}(\xi^{S_i}_{D})=0$.
\end{lemma}
\begin{proof} We fix $i$ and use $W/S^1$ to denote the bundle $T(\tau_{S_i})/\partial_i\Sigma_0$. We pull back $W/ S^1$ via a degree-$2$ map $S^1\to S^1$. This gives  the family $W'/ S^1$ with $W'=T(\tau^2_{S_i})$. Note the decomposition 
$W'=W'_{1}\cup W'_{2}$, where $W'_{1}=S^1\times X_{i} $ and $W'_{2}=T(\tau^2_{S_i}|_{\nu(S_{i})})$. 

We pull back the following geometric data from $W\to S^1$ to $W'\to S^1$:
\begin{itemize}
    \item The metric $g$ on $T^{V}W$, restricted from the metric $g_E$ on $TE$.
    \item The splitting $\mathcal{H}:TW\cong T^{V}W\oplus T^{H}W$, induced by the splitting (\ref{eq: splitting of tangent bundles}).
    \item The family spin-c connections $\widetilde{A}=\{A_{b}\}_{b\in \partial_i\Sigma_0}$. By our choice of $A^{t}_{E}$, the restriction  $\widetilde{A}$ to ${T(\tau_{S_i}|_{\nu(S_i)})}$ is a family of spin connection. 
\end{itemize}
We denote the pulled back metric, splitting and connections by $g',\mathcal{H}'$ and $\widetilde{A}'$ respectively. The Dirac operator for the family $W'/S^1$ is pulled back from $W/S^1$. So the determinant line bundle $\Det(\widetilde{D}^{+}(W'))$ is the pull back of $\Det(\widetilde{D}^{+}(W))$. The Bismut connection $\nabla'$ on $\Det(\widetilde{D}^{+}(W'))$ is the pull back of $\nabla$. And the canonical framing $\xi'$ on $\Det(\widetilde{D}^{+}(W'))$ is also pulled back from the canonical framing $\xi^{S_i}_{D}$. As a result, we have $\operatorname{hol}_{\nabla'}(\xi')=2\operatorname{hol}_{\nabla}(\xi)$. It remains to show that $\operatorname{hol}_{\nabla'}(\xi')=0$.

Since $\tau^2_{S_i}|_{\nu(S_i)}$ is isotopic to the identity relative to $\partial \nu(S_i)$, we have $W'_{2}\cong S^1\times \nu(S_i)$. So $W'/S^1$ is really a product family obtained by gluing two product families $W'_1/S^1$ and $W'_{2}/S^1$ together.   Let $g^c, \widetilde{A}^c, \mathcal{H}^c$ be constant families on $W'$ that  equal $\widetilde{g}', \widetilde{A}', \mathcal{H}'$ on the product piece $W'_{1}$. We further assume that $\widetilde{A}^c|_{W'_2}$ is spin. Let $\nabla^c$ be the Bismut connection on  $\Det(\widetilde{D}^{+}(W'))$. Since $\widetilde{D}^{+}(W')$ is a constant family of operators, the bundle $\Det(\widetilde{D}^{+}_{\widetilde{A}^c}(W'))$ has a canonical framing $\xi^c$, which is parallel with respect to $\nabla^c$. In particular, we have $\operatorname{hol}_{\nabla^c}(\xi^c)=0$. 

To prove $\operatorname{hol}_{\nabla^c}(\xi^c)=\operatorname{hol}_{\nabla'}(\xi')$, we pick a homotopy $\hat{\mathcal{H}}$ from $\mathcal{H}'$ to $\mathcal{H}^c$, a homotopy $\hat{g}$ from $g'$ to $g^c$, and a homotopy $\hat{A}$ from $\widetilde{A}'$ to $\widetilde{A}^c$. We assume that $\hat{g}$ and $\hat{A}$ are constant on the product piece $W'_{1}$. We further assume that $\hat{A}$ is spin on $W'_2$. We treat them as family objects associated to the bundle  $(W\times I)/B$, where $B=S^1\times I$. Consider the family Dirac operator $\hat{D}^{+}:=\widetilde{D}^{+}_{\hat{A}}(W'\times I)$ over $B$. We use $\hat{\nabla}$ to denote the Bismut connection on its determinant line bundle $\Det(\hat{D}^{+})$.

Note the decomposition $W'\times I=(W'_{1}\times I)\cup (W'_{2}\times I)$ of families. As before, we can deform the family operator $\hat{D}^{+}$ into a direct sum $\hat{D}_{1}^{+}\oplus \hat{D}_2^{+}$. Here $\hat{D}^{+}_{1}=\widetilde{D}^{+}_{\hat{A}}(W'_{1}\times I)$ and $\hat{D}^{+}_{2}=\widetilde{D}^{+}_{\hat{A}}(W'_{2}\times I)$. Both are families over $B$ and both are regarded as Fredholm operators equipped with the Atiyah--Patodi--Singer boundary conditions. Note that $\Det(\hat{D}^{+}_{1})$ has a canonical framing $\hat{\xi}_{1}$ because $\hat{D}^{+}_{1}$ is a constant family. On the other hand, $\Det(\hat{D}^{+}_{2})$ has a canonical framing $\hat{\xi}_{2}$ because $\hat{D}^{+}_{2}$ is a quaternionic linear family. Under the deformation, the framing $\hat{\xi}_{1}\oplus \hat{\xi}_{2}$ induces a framing $\hat{\xi}$ on  $\Det(\hat{D}^{+})$. By its construction, we have 
\[
\xi'=\hat{\xi}|_{S^1\times \{0\}}, \xi^c=\hat{\xi}|_{S^1\times \{1\}}.
\]
So the relative Chern class $c_1 ( \Det(\hat{D}^{+}), \xi'\cup \xi^{c}) $ vanishes. This implies 

\[\operatorname{hol}_{\nabla^{c}}(\xi^c)-\operatorname{hol}_{\nabla'}(\xi')+ \int_{B}\frac{i}{2\pi}F_{\hat{\nabla}}=c_{1}(\Det(\hat{D}^{+}),\xi'\cup \xi^c)=0.\] So we have
\[
\operatorname{hol}_{\nabla'}(\xi')=\operatorname{hol}_{\nabla'}(\xi')-\operatorname{hol}_{\nabla''}(\xi^c)=\int_{B}\frac{i}{2\pi}F_{\hat{\nabla}}=\frac{1}{48}\int_{W'\times I}(p_{1}(T^{V}(W'\times I))\wedge c_1-c_1\wedge c_1\wedge c_1).
\]
Here $c_1=\frac{i}{2\pi}F_{\hat{A}^{t}}\in \Omega^{2}(W'\times I)$.
Since $\hat{A}$ is spin on $W'_2\times I$, we have $c'_1\equiv 0$ on $W'_2\times I$, so  
\[
\operatorname{hol}_{\nabla'}(\xi')=\int_{W'_{1}\times I}(p_{1}(T^{V}(W'_{1}\times I))\wedge c_1-c_1\wedge c_1\wedge c_1).
\]
On the other hand, both $p_{1}(T^{V}(W'_{1}\times I))$ and $c_1|_{W'_{1}\times I}$ are pulled back from the fiber $X_{i}$. So the integral equals $0$.
\end{proof}

\begin{proof}[Proof of Proposition \ref{prop: relative Chern class equals index}] By Lemma \ref{lem: holonomy equals 0}, we have 
\[
\langle c_{1}(\operatorname{det}(\widetilde{D}^{+}(E_0)), \xi^{S_1}_{D},\cdots, \xi^{S_n}_{D}),[\Sigma_0]\rangle=\frac{i}{2\pi}\int_{\Sigma_0}F_{\nabla}.
\]
By the local index theorem (\ref{eq: local index formula}), we have 
\begin{equation}\label{eq: curvature formula}
\frac{i}{2\pi}\int_{\Sigma_0}F_{\nabla}=\frac{1}{48}\int_{E_0}(p_{1}(T^{V}E_0)\wedge c_1(\mathfrak{s}_{E_0})-c_1(\mathfrak{s}_{E_0})\wedge c_1(\mathfrak{s}_{E_0})\wedge c_1(\mathfrak{s}_{E_0}))    
\end{equation}
Since $p_{1}(T\Sigma)=0\in \Omega^{4}(\Sigma)$, we have $p_{1}(T^{V}E_0)=p_{1}(T E_0)$. Note that $c_1(\mathfrak{s}_{E_0}))$ is the restriction of the closed form 
\[
c_{1}(\mathfrak{s}_{E}):=\frac{i}{2\pi}F_{A^{t}_{E}}\in \Omega^{2}(E).
\]
So we can rewrite (\ref{eq: curvature formula}) as 
\[
\frac{i}{2\pi}\int_{\Sigma_0}F_{\nabla}=\frac{1}{48}\int_{E_0}(p_{1}(TE)\wedge c_1(\mathfrak{s}_{E})-c_1(\mathfrak{s}_{E})\wedge c_1(\mathfrak{s}_{E})\wedge c_1(\mathfrak{s}_{E}))   
\]
By our choice of $A^{t}_{E}$, the form $c_{1}(\mathfrak{s}_{E})|_{D_i}$ equals $0$. And the form $c_{1}(\mathfrak{s}_{E})|_{E^{\circ}_{i}}$ is pulled back from the fiber $X_{i}$, just like the form $p_{1}(TE^{\circ}_{i})$. The upshot is that the differential form  
\[
p_{1}(TE)\wedge c_1(\mathfrak{s})-c_1(\mathfrak{s})\wedge c_1(\mathfrak{s})\wedge c_1(\mathfrak{s})
\]
is identically vanishing on $E\setminus E_0$. Thus, from the Index Theorem for the $6$-dimensional Dirac operator we obtain:
\[
\frac{i}{2\pi}\int_{\Sigma_0}F_{\nabla}=\frac{1}{48}\int_{E}(p_{1}(TE)\wedge c_1(\mathfrak{s}_{E})-c_1(\mathfrak{s}_{E})\wedge c_1(\mathfrak{s}_{E})\wedge c_1(\mathfrak{s}_{E}))=\operatorname{ind}(D^{+}(E,\mathfrak{s}_{E})).   
\]
\end{proof}

\begin{proof}[Proof of Theorem \ref{thm: main generalized}] We may assume that there exists at least one singular fiber because otherwise the result follows from \cite[Corollary 1.3]{baraglia-konno}. 

Consider the family $E_0/\Sigma_0$, whose restriction to $\partial \Sigma_0$ is isomorphic to the disjoint union of the mapping tori $T(\tau_{S_i})/S^1$ for $1\leq i\leq n$. Consider the family Dirac operators $\widetilde{D}^{+}(E_0)$. The bundle $\Det(\widetilde{D}^{+}(E_0))$ is trivial because it is a complex line bundle over the punctured surface $\Sigma_0$. We pick any trivialization of $\Det(\widetilde{D}^{+}(E_0))$ and restrict it to $\partial_{i}\Sigma_0$. This gives a framing $\xi^{\partial_{i}}_{D}$ on $\Det(\widetilde{D}^{+}(T(\tau_{s_i}))$. The two framings $\xi^{\partial_{i}}_{D}$ and $\xi^{S_{i}}_{D}$ differ by an integer. We have: 
\begin{equation}\label{eq: proof main eq 1}
\begin{split}
\sum^{n}_{i=1}(\xi^{S_{i}}_{D}-\xi^{\partial_{i}}_{D})&=c_{1}(\operatorname{det}(\widetilde{D}^{+}(E_0)),\xi^{S_{1}}_{D},\cdots, \xi^{S_{n}}_{D})-c_{1}(\operatorname{det}(\widetilde{D}^{+}(E_0)),\xi^{\partial_{1}}_{D},\cdots, \xi^{\partial_{n}}_{D})\\&=c_{1}(\operatorname{det}(\widetilde{D}^{+}(E_0)),\xi^{S_{1}}_{D},\cdots, \xi^{S_{n}}_{D})\\&=\operatorname{ind}(D^{+}(E,\mathfrak{s}_E))
\end{split}
\end{equation}
The last equality follows from Proposition \ref{prop: relative Chern class equals index}.

Now we consider the bundle $H^{+}(f|_{E_0})$ over $\Sigma_0$. Since the spin-c structure $\mathfrak{s}|_{X}$ is preserved by the monodromy of $E_0/\Sigma_0$ and since $\operatorname{SW}(X,\mathfrak{s}_0)\neq 0$, the monodromy of $E_0/\Sigma_0$ must preserve the homological orientation on $X$. Hence the bundle $H^{+}(f|_{E_0})$ is trivial. We pick a trivialization of $H^{+}(f|_{E_0})$ and restricts to a framing $\xi^{\partial_{i}}_{d}$ on $H^{+}(f|_{T(\tau_{S_i})})$. By our definition of $H^{+}(f)$, we have 
\begin{equation}\label{eq: proof main eq 2}
\sum^{n}_{i=1}(\xi^{S_{i}}_{d}-\xi^{\partial_{i}}_{d})=\langle w^{+}_{2}(H^{+}(f)),[\Sigma]\rangle.
\end{equation}
We have two vanishing results for the family Bauer--Furuta invariants: by Proposition \ref{proposition: FBF vanishing original} we have 
\begin{equation}\label{eq: proof main eq 3}
\FBafu(\tau_{S_{i}},\mathfrak{s}_{X},\xi^{S_i}_{D},\xi^{S_i}_{d})=0, \forall 1\leq i\leq n \quad , 
\end{equation}
and by Proposition \ref{prop: FBF of boundary vanishing} we have 
\begin{equation}\label{eq: proof main eq 4}
\sum^{n}_{i=1}\FBafu(\tau_{S_{i}},\mathfrak{s}_{X},\xi^{\partial_i}_{D},\xi^{\partial_i}_{d})=0.
\end{equation}
Since $\operatorname{SW}(X,\mathfrak{s}_{X})$ is odd, Proposition \ref{prop: change of framing formula} implies that 
\begin{equation}\label{eq: proof main eq 5}
\FBafu(\tau_{S_{i}},\mathfrak{s}_{X},\xi^{S_i}_{D},\xi^{S_i})-\FBafu(\tau_{S_{i}},\mathfrak{s}_{X},\xi^{\partial_i}_{D},\xi^{\partial_i})= (\xi^{S_i}_{D}-\xi^{\partial_i}_{D})+(\xi^{S_i}_{d}-\xi^{\partial_i}_{d}).
\end{equation}
Combining Equation (\ref{eq: proof main eq 3}), (\ref{eq: proof main eq 4}), (\ref{eq: proof main eq 5}), we obtain  that 
\[
\sum^{n}_{i=1}(\xi^{S_i}_{D}-\xi^{\partial_i}_{D})=\sum^{n}_{i=1}(\xi^{S_i}_{d}-\xi^{\partial_i}_{d}).
\]
By Equations (\ref{eq: proof main eq 1}) and (\ref{eq: proof main eq 2}), we have 
\begin{equation}
\langle w^{+}_{2}(H^{+}(f)),[\Sigma]\rangle=    \operatorname{ind}(D^{+}(E,\mathfrak{s}_E))
\end{equation}
\end{proof}

\begin{lemma}\label{lem: cohomology of E}
(1) We have an exact sequence \begin{equation}\label{eq: long exact sequence}
0\rightarrow H^2(S^2;\mathbb{Z})\xrightarrow{{f^*}} H^{2}(E;\mathbb{Z})\xrightarrow{j} \oplus^{n}_{i=1} H^{2}(X_{i},\partial X_{i};\mathbb{Z})\xrightarrow{\partial} H^{3}(E_0,\partial E_0;\mathbb{Z})\rightarrow \cdots.
\end{equation}

Here $j$ is induced by the inclusion of the singular fiber $X_{f(p_{i})}\hookrightarrow E$ and the homeomorphism $X_{i}/\partial X_{i}\cong X_{f(p_i)}$.\\
(2) Suppose $H_{1}(X)$ has no 2-torsion. Then $H^{3}(E_0,\partial E_0;\mathbb{Z})$ also has no 2-torsion.
\end{lemma}

\begin{proof}[Proof of Lemma \ref{lem: cohomology of E}] We consider the triple 
\[
(E,\sqcup^{n}_{i=1}E_i,\sqcup^{n}_{i=1}D_{i})
\]
and  the associated long exact sequence 
\[
H^{1}(E, \sqcup^{n}_{i=1}E_i)\to H^{2}(E, \sqcup^{n}_{i=1}D_{i}))\to 
H^{2}(\sqcup^{n}_{i=1}E_i,\sqcup^{n}_{i=1}D_{i})) \to H^{2}(E, \sqcup^{n}_{i=1}E_i) \rightarrow H^{3}(E, \sqcup^{n}_{i=1}D_{i})).
 \]
 By excision, we have \[
H^{2}(\sqcup^{n}_{i=1}E_i,\sqcup^{n}_{i=1}D_{i})\cong H^{*}(\sqcup^{n}_{i=1}E^{\circ}_{i},\sqcup^{n}_{i=1}(E^\circ_{i}\cap D_i))\cong \oplus^{n}_{i=1}H^*(X_{i},\partial X_{i})\]
 and 
 \[
 H^{*}(E,\sqcup^{n}_{i=1}E_i)\cong H^*(E_0,\partial E_0).
 \]
 This gives the exact sequence. 
\[0\rightarrow H^2(E_0,\partial E_0)\rightarrow H^{2}(E)\xrightarrow{j} \oplus^{n}_{i=1} H^{2}(X_{i},\partial X_{i})\xrightarrow{\partial} H^{3}(E_0,\partial E_0)\rightarrow \cdots.\]

Next, we claim that  $f:(E_0,\partial E_0)\to (\Sigma_0,\partial \Sigma_0)$ induces an isomorphism \[f^*:H^{2}(E_0,\partial E_0)\cong H^2(\Sigma_0,\partial \Sigma_0)\cong \mathbb{Z}.\]    
To see this, we consider the Serre spectral sequence that computes  $H^{2}(E_0,\partial E_0)$. The second page $E^{i,j}_{2}$ of this spectral sequence is $H^{i}(\Sigma_0,\partial \Sigma_0;H^{j}(X))$, the cohomology of the base with the cohomology of the fiber as local coefficient. Note that \[E_2^{0,2}=H^{0}(\Sigma_0,\partial \Sigma_0; H^{2}(X))=\ker(H^0(\Sigma_0;H^{2}(X))\to H^0(\partial \Sigma_0;H^{2}(X)))=0\] 
and note that $E^{1,1}_{2}=E^{0,1}_{2}=0$. And $E^{2,0}_{2}=H^{2}(\Sigma_0,\partial \Sigma_0)$. So the desired result follows.  

(2) By the Lefschetz duality, we have $H^{3}(E_0,\partial E_0;\mathbb{Z})\cong H_{1}(E_0;\mathbb{Z})$. A straightforward application via the Mayer--Vietoris sequence shows that this group is torsion free. 
\end{proof}

\begin{lemma}\label{lem: spin-c extension}
Let $\mathfrak{s}_{X}$ be a  spin-c structure on $X$ such that $\langle c_{1}(\mathfrak{s}_{X}),[S_i]\rangle=0$ for all $S_i$. Then there exists a spin-c structure $\mathfrak{s}_{E}$ on $E$ such that $\mathfrak{s}|_{X}=\mathfrak{s}_{X}$.
\end{lemma}
\begin{proof}
Since     $\langle c_{1}(\mathfrak{s}_{X}),[S_i]\rangle=0$, the isomorphic class of $\mathfrak{s}_{X}$ is preserved under the monodromy of the bundle $E_0\to \Sigma_0$. So there exists a spin-c structure $\widetilde{\mathfrak{s}}_{E_0}$ on $T^{V}E_0$ that restricts to $\mathfrak{s}_{X}$ on fibers. Note that for any $i\geq 1$, the map 
\[
H^{2}(E_{i},E_{i}\cap E_{0};\mathbb{Z})\to H^{2}(E_{i},E_{i}\cap E_{0};\mathbb{F}_2)
\]
is surjective. So there is no obstruction to extend $\widetilde{\mathfrak{s}}_{E_0}$ to a spin-c structure $\widetilde{\mathfrak{s}}_{E^{\circ}}$ on $T^{V}E^{\circ}$. Together with a spin structure on $T\Sigma$, $\widetilde{\mathfrak{s}}_{E^{\circ}}$ determines a spin-c structure $\mathfrak{s}_{E^\circ}$ on $E^{\circ}$. (See (\ref{diagram: spinc}).) And $\mathfrak{s}_{E^\circ}$ extends to a spin-c structure $\mathfrak{s}_{E}$ on $E$. 
\end{proof}

\begin{proof}[Proof of Corollary \ref{cor: K3}] By Lemma \ref{lem: spin-c extension}, there exists a spin-c structure $\mathfrak{s}_{E}$ on $E$ whose restriction to the fiber is the spin structure $\mathfrak{s}_0$ on $K3$. By Lemma \ref{lem: cohomology of E}, there exists $a\in \mathbb{Z}\cong H^{2}(S^2)$ such that $c_{1}(\mathfrak{s})=f^*(a)$. Since $w_{2}(TE)\equiv c_{1}(\mathfrak{s}_{E})\mod 2$, $E$ is spin if and only if $a$ is even. On the other hand, we have 
\[
\operatorname{ind}(D^{+}(E,\mathfrak{s}_{E}))=\frac{1}{48}\langle f^{*}(a)\cup p_{1}(TE),[E]\rangle = a \mod 2. 
\]
So by Theorem \ref{thm: main generalized}, the space $E$ is spin if and only if $w_{2}(H^{+}(f))=0$.    
\end{proof}

\begin{proof}[Proof of Theorem \ref{thm: main}] Suppose the composition $\tau_{S_1}\cdots \tau_{S_n}$ is smoothly isotopic to the identity. Then there exists a smooth Lefschetz fibration $f:E\to S^2$ with $X$ as the fiber and $S_{1},\cdots, S_{i}$ the vanishing cycles. By Lemma \ref{lem: spin-c extension}, there exists a spin-c structure $\mathfrak{s}_{E}$ on $E$ that restricts to $\mathfrak{s}$ on fibers. By Theorem \ref{thm: main generalized} and Proposition \ref{proposition:extension}, we have 
\[
\Delta(S_1,\cdots, S_n)=\langle w_{2}(H^{+}(f)),[S^2]\rangle \equiv \operatorname{ind}(D^{+}(E,\mathfrak{s}_{E})) \mod 2.
\]
So it suffices to show that the index 
\[\operatorname{ind}(D^{+}(E,\mathfrak{s}_{E}))=\frac{1}{48}\langle p_{1}(E)\cup c_{1}(\mathfrak{s}_{E})-c_{1}(\mathfrak{s}_{E})\cup c_{1}(\mathfrak{s}_{E})\cup c_{1}(\mathfrak{s}_{E}),[E]\rangle \]
is even. Consider the image of $c_{1}(\mathfrak{s}_{E})$ under the map $j$ in (\ref{eq: long exact sequence}). Since $c_{1}(\mathfrak{s}_{E})$ is divisible by $32$, there exits $b\in \oplus^{n}_{i=1}H^{2}(X_{i},\partial X_{i} )$ such that $j(c_{1}(\mathfrak{s}))=32 b$. In particular, $\partial (32b)=0\in H^{3}(E_0,\partial E_0;\mathbb{Z})$. By Lemma \ref{lem: cohomology of E}, $H^{3}(E_0,\partial E_0)$ has no 2-torsion, so $\partial b=0$. Hence there exists $h\in H^{2}(E;\mathbb{Z}) $  such that $j(h)=b$. Hence $c_{1}(\mathfrak{s})-32h\in \ker j=\image(f^{*})$. To this end, we see that $c_{1}(\mathfrak{s})=f^*(a)+32b$ for some $a\in H^{2}(S^2;\mathbb{Z})$. This implies that 
\[
\begin{split}
\langle p_{1}(E)\cup (f^*(a)+32b)-(f^*(a)+32b)^{3}\rangle &\equiv \langle p_{1}(E)\cup f^*(a), [E]\rangle \\ &\equiv a\cdot  \langle p_{1}(X),[X]\rangle \\ &\equiv a\cdot \sigma(X)\\ &\equiv 0 \mod 32.    
\end{split}
\]
So $\operatorname{ind}(D^{+}(E,\mathfrak{s}_{E}))$ is even. 
\end{proof}

\section{Examples}
\label{section Examples}

\subsection{Elliptic surfaces}\label{subsection:ellipticsurfaces}

We shall use the standard notation of elliptic surfaces, as in \cite{gompfstipsicz}.
For example, $E(n)$ denotes the simply-connected minimal elliptic surface with Euler characteristic $12n$ and no multiple fibers, and $E(n)_{p,q}$ denotes the elliptic surface obtained from $E(n)$ by performing logarithmic transformations of multiplicities $p$ and $q$ along two distinct regular fibers.
$N(n)_{p,q}$ denotes the the Gompf nucleus inside $E(n)_{p,q}$.

\begin{lemma}
\label{lem: fiber sum}
Let $n \geq 2$ and $p \geq q \geq 1$ with $p,q$ coprime. 
Then the elliptic surface $E(n)_{p,q}$ admits a symplectic structure $\omega$ for which $M(2,3,7)$ is embedded symplectically into $(E(n)_{p,q}, \omega)$ away from $N(n)_{p,q}$. 
\end{lemma}

\begin{proof} We begin by recalling certain compactifications of Milnor fibers in weighted projective spaces (\cite{dolgachev82}). Let $n \geq 2$, and consider the complex hypersurface $S = \{ x^2 + y^3 + z^{6n-1}+w^{36n-6}=0\}$ in the weighted projective $3$-space $\mathbb{P}:= \mathbb{P}(18n-3,12n-2,6,1)$. We regard both $\mathbb{P}$ and $S$ as singular varieties, whose singularities are isolated cyclic quotients. The variety $S$ can be regarded as compactification of the (open) Milnor fiber of the Brieskorn surface singularity $x^2 + y^3 + z^{6n-1}=0$, since in the affine locus $\{ w \neq 0\} \cong \mathbb{C}^3$ of $\mathbb{P}$, the hypersurface $S$ is described by $x^2 + y^3 + z^{6n-1} +1 =0$. As explained in \cite[\S 2.3.1]{KLMME}, $S$ is smooth away from $3$ isolated quotient singularities located on the divisor at infinity $C = S \cap \{w=0\}$. 

After minimally resolving the quotient singularities in $S$ we obtain a non-singular surface $\widetilde{S}$. The divisor at infinity in $\widetilde{S}$, i.e. the strict transform $\widetilde{C}$ of $C$, is the configuration of curves shown in \cite[Figure 2(B)]{KLMME}, which contains a $(-1)$--curve. As explained in \cite[\S 2.3.1]{KLMME}, successively blowing down $(-1)$--curves turns $\widetilde{S}$ into a minimal complex surface $X$ diffeomorphic to $E(n)$, which still contains an embedding of the compact Milnor fiber $M(2,3,6n-1)$, with complement $X \setminus M(2,3,6n-1)$ identified with the Gompf nucleus $N(n)$. Performing logarithmic transformations on $X=E(n)$ on two disjoint tori with self-intersection $0$ in the Gompf nucleus $N(n)$ leads to an embedding of $M(2,3,6n-1) \subset E(n)_{p,q}$ with complement $N(n)_{p,q}$.\\

With additional care, the embedding $M(2,3,6n-1) \subset E(n)_{p,q}$ just described can be made symplectic, for suitable symplectic form $\omega$ on $E(n)_{p,q}$, as we now explain.

Like ordinary projective space, the weighted projective space $\mathbb{P}$ carries a tautological sheaf $\mathcal{O}_{\mathbb{P}}(-1)$. By \cite[Theorem 4B.7]{beltrametti-robbiano} there exists an integer $k >0$ such that sheaf $\mathcal{O}_{\mathbb{P}}(k)$ is very ample. This induces an embedding of the weighted projective space inside some (ordinary) complex projective space. Using this, one can embed the smooth surface $\widetilde{S}$ in an ordinary projective space $\widetilde{S} \subset \mathbb{C}P^N$, in such a way that $\widetilde{S}\setminus \widetilde{C} $ is properly embedded in an affine piece $\mathbb{C}^N$.

We have a Kähler form $\omega$ on $\widetilde{S}$ by restriction of the Fubini--Study form in $\mathbb{C}P^N$. Consider the compact Milnor fiber of $x^2 +y^3 +z^{6n-1}=0$ given by 
\[
M(2,3,6n-1) := \{ (x,y,z) \in \mathbb{C}^3 \, | \, x^2 + y^3 + z^{6n-1} +1 =0 \, \text{and} \, |x|^2+|y|^2 +|z|^2 \leq 1 \} \quad , \quad r>0 
\]
and equipped with the symplectic form $\omega_0$ given by restriction of the standard form in $\mathbb{C}^3$ (this is the natural symplectic structure on the Milnor fiber). We want to symplectically embed $(M(2,3,6n-1) , \omega_0 )$ in $(\widetilde{S} , \omega)$. Of course, $M(2,3,6n-1)$ is naturally embedded in $\widetilde{S} \setminus \widetilde{C}$, but the symplectic forms $\omega_0$ and $\omega$ don't match. However, these two forms each arise from a strictly plurisubharmonic exhaustive function on the same complex manifold $\widetilde{S}\setminus \widetilde{C}$. Thus, by \cite[Theorem 1.4.A]{eliashberg-gromov}, there is a symplectomorphism $(\widetilde{S}\setminus \widetilde{C} , \widetilde{\omega}) \cong (\widetilde{S} \setminus \widetilde{C} , \omega_0 )$. This provides a symplectic embedding of $(M(2,3,6n-1) , \omega_0 ) $ in the Kähler surface $(\widetilde{S} , \omega)$.

The passage from $\widetilde{S}$ to $X = E(n)$ involves blowing down symplectic $(-1)$--spheres, which can be carried out symplectically, leading to a symplectic form on $X$, also denoted $\omega$, with a symplectic embedding of $(M(2,3,6n-1), \omega_0 ) \subset (X, \omega )$ disjoint from the Gompf nucleus $N(n)$. 
In addition, by \cite[\S 3]{rationalblowdown}, the logarithmic transformations of order $p$ in the Gompf nucleus $N(n) $ can be realised by a sequence of $p-1$ blowups and a rational blowdown of a $C_p$ configuration. The $p-1$ blowups can be performed symplectically, and by \cite{rationalblowdown-symplectic} the rational blowdown of $C_p$ can also be done symplectically since the $C_p$ configuration can be chosen to be symplectic: indeed, the configuration $C_p$ is obtained from a nodal sphere with self-intersection $0$ in the neighborhood of the cusp fiber in $N(n)$ by the procedure explained in \cite[\S 3]{rationalblowdown}, and since this nodal sphere can be chosen symplectic then so can $C_p$. Furthermore, since the neighborhood of the cusp in $N(n)$ contains two disjoint such nodal spheres, the logarithmic transformation can be done symplectically twice (with orders $p$ and $q$). This proves the existence of a symplectic form $\omega$ on $E(n)_{p,q}$ with a symplectic embedding of the Milnor fiber $(M(2,3,6n-1), \omega_0 )$ away from $N(n)_{p,q}$. \\

Finally, we note that for $n \geq 2$ the singularity $x^2 +y^3 + z^{6n-1}=0$ is adjacent to the singularity $x^2 + y^3 + z^7 =0$, and hence by \cite[Lemma 9.9]{keating-free} there is a symplectic embedding of their Milnor fibers $(M(2,3,7), \omega_0 ) \subset (M(2,3,6n-1), \omega_0 )$. Hence, by the above construction, $(M(2,3,7),\omega_0 )$ symplectically embeds in $(E(n)_{p,q}, \omega )$ away from $N(n)_{p,q}$. 
\end{proof}

If one does not insist on obtaining an explicit construction of the symplectic form, Lemma~\ref{lem: fiber sum} can also be proved in the following way using fiber sums.

\begin{proof}[Alternative proof of Lemma~\ref{lem: fiber sum}]
As seen above, on $E(2) = K3$, there exists a symplectic structure for which the Milnor fiber $M(2,3,7)$ is symplectically embedded. 
The complement of this embedding contains the Gompf nucleus $N(2)$. 
We fix a complex structure on $E(n)$ that makes $E(n)$ an elliptic fibration.
By construction, logarithmic transformations on $E(n)$ are operations that do not change the complex structure away from the regular fiber $F_0$ on which the operation is performed \cite{Kodaira-log-transform}.
Hence, another regular fiber $F$ away from $F_0$ is a complex submanifold of $E(n)_{p,q}$. 
Recall also that the Gompf nucleus $N(n)_{p,q}$ contains a regular fiber, so we can take $F$ inside $N(n)_{p,q}$. 
Recall also that every complex surface with even first Betti number admits a K{\"a}hler structure (see, for example, \cite{Buchdahl99,Lamari99}).
Since $b_1(E(n)_{p,q}) = 0$, as we have assumed $p$ and $q$ to be coprime, $E(n)_{p,q}$ admits a K{\"a}hler structure.
As we say, the fiber $F$ is a complex submanifold of $E(n)_{p,q}$, and hence a symplectic submanifold for any Kähler structure on $E(n)_{p,q}$.

Since Gompf’s fiber sum is a local operation that changes the symplectic structure only in neighborhoods of the symplectic submanifolds along which the sum is taken \cite{Gompf-fiber-sum}, by picking a K{\"a}hler structure on $E(n-2)_{p,q}$ and performing the symplectic sum along a regular fiber in $N(n-2)_{p,q}$ and a regular fiber in $N(2)$, we obtain a symplectic structure on $E(n)_{p,q}$ for which $M(2,3,7)$ is symplectically embedded.
\end{proof}

\begin{lemma}
\label{lem: BK mod 2 basic class}
Let $n \geq 1$ and $p \geq q \geq 1$. 
Suppose that $p$ and $q$ are odd, coprime integers, and that $(p, q)$ does not lie in the set
\begin{align}
\label{eq: exceptional p q}
\{(1,1), (1,3), (1,5), (1,7), (1,9), (3,5)\}.
\end{align}
Then $E(4n)_{p,q}$ admits a mod 2 basic class $\mathfrak{s}$ for which $c_1(\mathfrak{s})$ is divisible by 32.
\end{lemma}

\begin{proof}
This is proven in \cite[Proof of Theorem 5.2.]{baragliakonno2024irreducible}.    
\end{proof}

\begin{theorem}
\label{thm: elliptic surface examples}
Let $n \geq 1$ and $p \geq q \geq 1$. 
Suppose that $p$ and $q$ are odd, coprime integers, and that $(p, q)$ does not lie in the set \eqref{eq: exceptional p q}.  Then $E(4n)_{p,q}$ admits a symplectic structure $\omega$ and a smooth embedding of $M(2,3,7)$ such that:
\begin{itemize}
    \item the embedding of $M(2,3,7)$ into $E(4n)_{p,q}$ is symplectic with respect to $\omega$, and
    \item the Dehn twist on $E(4n)_{p,q}$ along the boundary of $M(2,3,7)$ is not smoothly isotopic to the identity.
\end{itemize} 
\end{theorem}

\begin{proof}
Lemma~\ref{lem: fiber sum} provides a symplectic structure $\omega$ on $E(4n)_{p,q}$ for which $M(2,3,7)$ is symplectically embedded into $(E(4n)_{p,q}, \omega)$. 
Lemma~\ref{lem: BK mod 2 basic class} yields a mod 2 basic class $\mathfrak{s}$ such that $c_1(\mathfrak{s})$ is divisible by 32. 
In addition, we have $c_1(\mathfrak{s})|_{M(2,3,7)} = 0$ by Lemma~\ref{lem: fiber sum} and the fact that every basic class of the elliptic surface $E(m)_{p,q}$ is supported in the nucleus $N(m)_{p,q}$ (see, for example, \cite[Theorem 3.3.6]{gompfstipsicz}). 
Moreover, the signature of $E(4n)_{p,q}$ is divisible by 32. 
Therefore, we can apply Corollary~\ref{cor: examples} to conclude that the Dehn twist on $E(4n)_{p,q}$ along the boundary of $M(2,3,7)$ is not smoothly isotopic to the identity. 
This completes the proof.
\end{proof}

\begin{corollary}
\label{cor: twist on elliptic surfaces}
Let $n,p,q$ be as in Theorem~\ref{thm: elliptic surface examples}.
Then there is a smooth embedding of $M(2,3,7)$ into $E(4n)_{p,q}$ such that the Dehn twist on $E(4n)_{p,q}$ along $\del M(2,3,7)=\Sigma(2,3,7)$ is an exotic diffeomorphism.
\end{corollary}

\begin{proof}
The non-triviality of the Dehn twist as a smooth mapping class has been proven in Theorem~\ref{thm: elliptic surface examples}. 
Thus, it suffices to show that the Dehn twist is trivial as a topological mapping class. 
This follows from the fact that the Dehn twist acts trivially on homology, together with a result of Quinn~\cite{Quinn86} (with a recent correction by \cite{gabai2023pseudoisotopies}), which states that a homeomorphism of a simply-connected closed $4$-manifold is topologically isotopic to the identity if it acts trivially on homology.
\end{proof}

\subsection{Non-symplectic irreducible 4-manifolds: knot surgery}\label{subsection:knowsurgery}

The first examples of exotic diffeomorphisms of simply-connected irreducible $4$-manifolds were recently constructed by Baraglia and the first author~\cite{baragliakonno2024irreducible}. 
However, there seems to be no reason to expect that the diffeomorphisms in~\cite{baragliakonno2024irreducible} can be written as Dehn twists along Seifert fibered $3$-manifolds.
Moreover, the construction in~\cite{baragliakonno2024irreducible} essentially uses realization results from complex geometry (\cite{Lonnediffeomorphism98,EbelingOkonekdiffeomorphism91}), so the 4-manifolds there are required to be K{\"a}hler (note that a complex surface admits a K{\"a}hler structure under the assumption of simple-connectivity). 
In contrast, our results can be used to detect exotic diffeomorphisms of irreducible 4-manifolds that do not even admit symplectic structures, highlighting a major difference between the method in~\cite{baragliakonno2024irreducible} and that of the present paper:

\begin{theorem}
\label{thm: nonsymplectic}
There exist simply-connected irreducible closed smooth $4$-manifolds $X$ that do not admit any symplectic structure but admit exotic diffeomorphisms.  
\end{theorem}

The proof of this theorem is elaborated in the following example:

\begin{example}
\label{ex: knot surgery}
We consider Fintushel–Stern's knot surgery~\cite{FS-knot-surgery}.
Let 
$k \geq 1$, and $T(k)$ be the $k$-twist knot (see Figure~1 in \cite{FS-knot-surgery}). 
As noted in \cite{FS-knot-surgery}, the Alexander polynomial of $T(k)$ is given by
\[
\Delta_{T(k)}(t) = kt - (2k+1) + kt^{-1}.
\]
For a positive integer $N\geq1$, put $K(k,N) = \#_N T(k)$.
Since the Alexander polynomial is multiplicative under connected sum, we have
\begin{align}
\label{eq: Alexander poly}
\Delta_{K(k,N)}(t) = (kt - (2k+1) + kt^{-1})^N.    
\end{align}

For $n\geq1$, pick a regular elliptic fiber $F$ of $E(n)$ for a given elliptic fibration structure on $E(n)$.
Let $X$ be the Fintushel--Stern knot surgery of $E(n)$ along $F$ using the knot $K(k,N)$: in the notation of \cite{FS-knot-surgery}, $X = E(n)_{K(k,N)}$.

The Seiberg--Witten invariant of $E(n)$ expressed as a Laurent polynomial is given by
\[
\mathcal{SW}(E(n)) = (t-t^{-1})^{n-2},
\]
where $t$ is the (Poincar{\'e} dual of the) homology class of the fiber $F$ (see, for example, \cite[Lecture 2]{FS6lec} or \cite[Theorem 3.3.20]{Nicolaescu-book}).
Hence it follows from the knot surgery formula \cite[Theorem 1.5]{FS-knot-surgery} and \eqref{eq: Alexander poly} that 
\begin{align}
\label{eq: SW polynomial}
\mathcal{SW}(X) = (t-t^{-1})^{n-2}(kt - (2k+1) + kt^{-1})^N.
\end{align}
Expanding \eqref{eq: SW polynomial}, we see that the coefficient of every term is neither $1$ nor $-1$, provided that $k \ge 2$. 
This means that the Seiberg–Witten invariant is neither $1$ nor $-1$ for any spin$^c$ structure. 
Hence, $X$ does not admit a symplectic structure by Taubes’s theorem~\cite{Taubes94}.

We shall use the leading term of \eqref{eq: SW polynomial}, which is
$k^N t^{n-2+N}$, so the Seiberg--Witten invariant of the spin-c structure $\mathfrak{s}$ with 
\begin{align}
\label{eq: dibisibility}
c_1(\mathfrak{s}) = (n-2+N)t
\end{align}
is given by 
\begin{align}
\label{eq: SW inv compute}
\operatorname{SW}(X,\mathfrak{s}) = k^N.
\end{align}

Now we see that the 4-manifold $X$ is irreducible, following \cite[Proof of Theorem 1.6]{Szabo-non-symplectic-irreducible}. 
Assume that $X$ splits into a connected sum, $X= Y\# Z$.
Since the Seiberg--Witten invariant of $X$ is non-trivial as seen above, one of $Y$ and $Z$, say $Z$, is negative-definite.
It follows from Donaldson's diagonalization theorem that $Z$ is homotopy equivalent to $\#_m\overline{\CP}^2$ for some $m\geq0$. Now the blow-up formula of the Seiberg--Witten invariant \cite{FS-blow-up} shows that every basic class of $X$ is of the form $L\pm E_1 \pm \cdots \pm E_n$, where the signs need not be the same.
Here $ (E_1,\cdots, E_m)$ is a basis of $H^2(Z;\Z)$ with $E_i^2= -1$ and $L$ is a basic class of $Y$. 
If $m>0$, let $K, K'$ be basic classes defined by $K=L+E_1+\dots+E_m$ and $K'=L-E_1-\dots-E_m$. 
Then $K-K'=2(E_1+\cdots+E_m)$, thus $(K-K')^2=-4m$.
However, by construction, every basic class of $X$ is a multiple of (the Poincar{\'e} dual of) the fiber $F$ and the self-intersection number of the fiber is zero.
Thus we should have $(K-K')^2=0$. Thus $m=0$, which implies that $X$ is irreducible.

Now we make the following assumptions:
\begin{align*}
\text{$n$ is divisible by $4$}, \  \text{$n - 2 + N$ is divisible by $32$}, \ 
k \ge 3, \ \text{$k$ is odd}.
\end{align*}
It is clear that there are infinitely many tuples $(n, N, k)$ satisfying these assumptions.  
Under these conditions, $c_1(\mathfrak{s})$ is divisible by $32$ by~\eqref{eq: dibisibility}, and $\operatorname{SW}(X, \mathfrak{s})$ is odd by~\eqref{eq: SW inv compute}.  
Furthermore, $\sigma(X)$ is divisible by $32$, since $n$ is divisible by $4$.  
As observed in Lemma~\ref{lem: fiber sum}, the Milnor fiber $M(2,3,7)$ is smoothly embedded in $E(n)$ away from $N(n)$, and we may assume that the knot surgery is performed on $N(n)$.  
Therefore, $M(2,3,7)$ is smoothly embedded in $X$.  
Thus, we can apply Corollary~\ref{cor: examples} and conclude that the Dehn twist on $X$ along $\partial M(2,3,7)$ is not smoothly isotopic to the identity.  
Together with the topological triviality result of~\cite{Quinn86}, we conclude that this Dehn twist defines an exotic diffeomorphism of $X$.
\end{example}

\newpage
\appendix

\section{Mathematica code}\label{appendix:code}

\begin{lstlisting}[language=Mathematica, basicstyle=\scriptsize\ttfamily]
(*Gabrielov numbers and monodromy orders*)
G = {{2, 3, 7}, {2, 3, 8}, {2, 3, 9}, {2, 4, 5}, {2, 4, 6}, {2, 4, 7}, {2, 5, 5}, {2, 5, 6}, {3, 3, 4}, {3, 3, 5}, {3, 3, 6}, {3, 4, 4}, {3, 4, 5}, {4, 4, 4}};
orders = {42, 30, 24, 30, 22, 18, 20, 16, 24, 18, 15, 16, 13, 12};

For[k = 1, k <= Length[G], k++,
 {p, q, r} = G[[k]];
 h = orders[[k]];
 n = p + q + r;
 (*Create intersection matrix*)
 M = DiagonalMatrix[ConstantArray[-2, n]];
 For[i = 1, i <= n - 3, i++,
  If[i < 
      p - 1 || (i > p - 1 && i < p + q - 2) || (i > p + q - 2 && 
       i < p + q + r - 3), M[[i, i + 1]] = 1;
    M[[i + 1, i]] = 1;];];
 M[[n - 2, p - 1]] = 1; M[[p - 1, n - 2]] = 1;
 M[[n - 2, p + q - 2]] = 1; M[[p + q - 2, n - 2]] = 1;
 M[[n - 2, p + q + r - 3]] = 1; M[[p + q + r - 3, n - 2]] = 1;
 M[[n - 1, p - 1]] = 1; M[[p - 1, n - 1]] = 1;
 M[[n - 1, p + q - 2]] = 1; M[[p + q - 2, n - 1]] = 1;
 M[[n - 1, n - 3]] = 1; M[[n - 3, n - 1]] = 1;
 M[[n - 2, n - 1]] = -2; M[[n - 1, n - 2]] = -2;
 M[[n, n - 1]] = 1; M[[n - 1, n]] = 1;
 
 (*Reflection*)
 e[i_] := UnitVector[n, i];
 R[i_, v_List] := v + (v . M[[i]])*e[i];
 (*Vectors a and b*)
 a = 2*e[n - 2] - 2*e[n - 1] - e[n];
 b = ConstantArray[0, n];
 For[i = 1, i <= p - 1, i++, b = b + (i/p)*e[i];];
 For[i = 1, i <= q - 1, i++, b = b + (i/q)*e[p - 1 + i];];
 For[i = 1, i <= r - 1, i++, b = b + (i/r)*e[p + q - 2 + i];];
 b = b + e[n - 2];
 innerProduct[u_List, v_List, N_List] := (u . N) . v;
 (*Compute endpoints of segments in loop eta*)
 vecList = {{1, 0}};
 v = a;
 For[j = 1, j <= h, j++,
  For[i = 1, i <= n, i++,
    v = R[i, v];
    vproj = {innerProduct[v, a, M]/innerProduct[a, a, M], 
      innerProduct[v, b, M]/innerProduct[b, b, M]};
    AppendTo[vecList, vproj];
    ];
  ];
 (*Plot loop eta*)
 l = Length[vecList];
 colors = Table[ColorData["GrayTones"][m/l], {m, 1, l}];
 plot1 = 
  Graphics[{Table[{colors[[m]], PointSize[0.015], 
      Point[vecList[[m]]]}, {m, 1, l}], 
    Table[{colors[[m]], Arrowheads[0.02], 
      Arrow[{vecList[[m]], vecList[[m + 1]]}]}, {m, 1, l - 1}]}, 
   Axes -> True, AxesOrigin -> {0, 0}, GridLines -> Automatic, 
   PlotLabel -> Row[{"p = ", p, ", q = ", q, ", r = ", r}], 
   ImageSize -> 600, AspectRatio -> 1];
 Print[plot1];
 ]
\end{lstlisting}

 \newpage
\section{$\Delta$ for the remaining Exceptional Unimodal Singularities}\label{appendix:figures}

\begin{figure}[htbp]
    \centering

    \begin{subfigure}[b]{0.45\textwidth}
        \includegraphics[width=\textwidth]{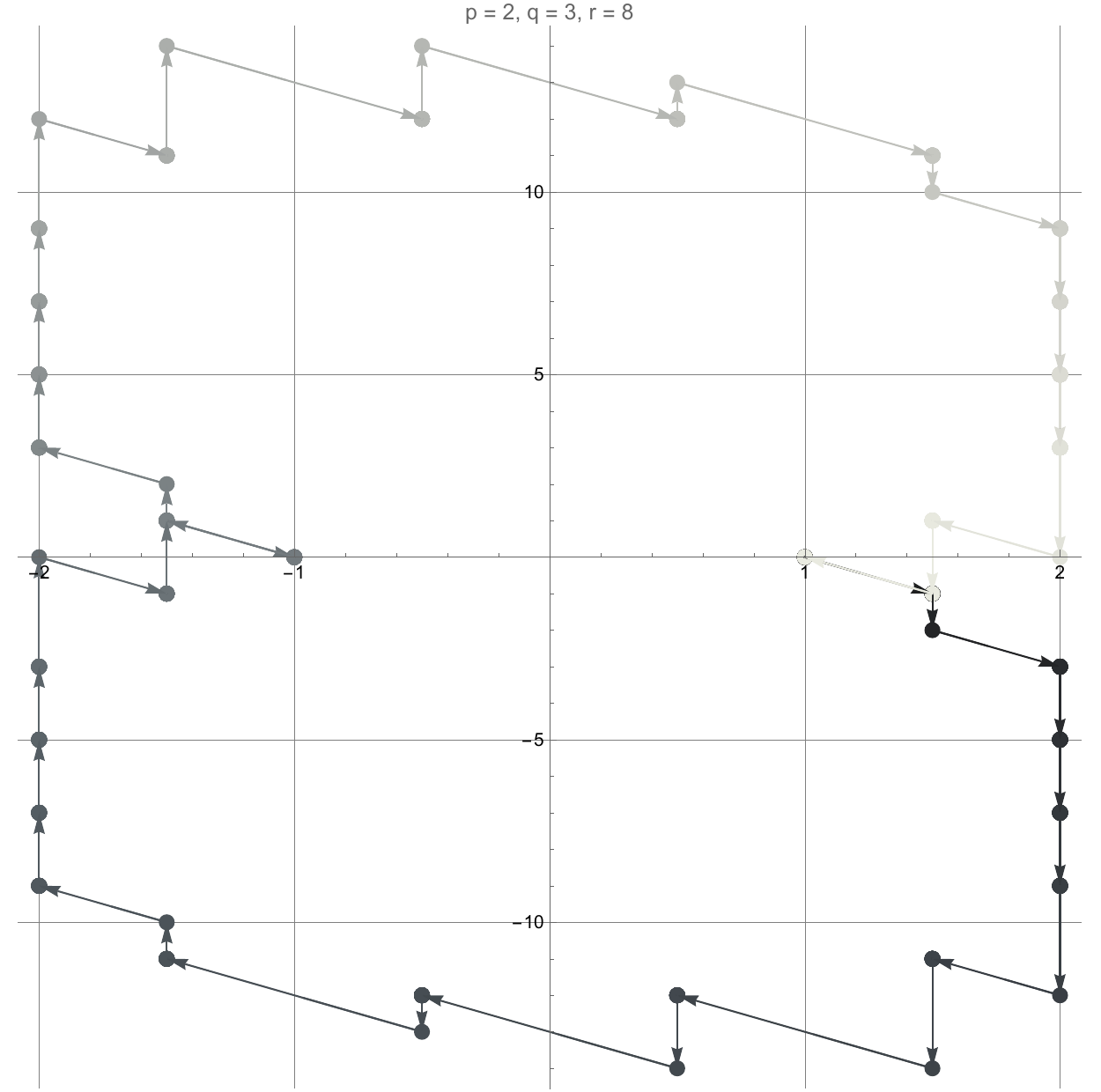}
    \end{subfigure}
    \hfill
    \begin{subfigure}[b]{0.45\textwidth}
        \includegraphics[width=\textwidth]{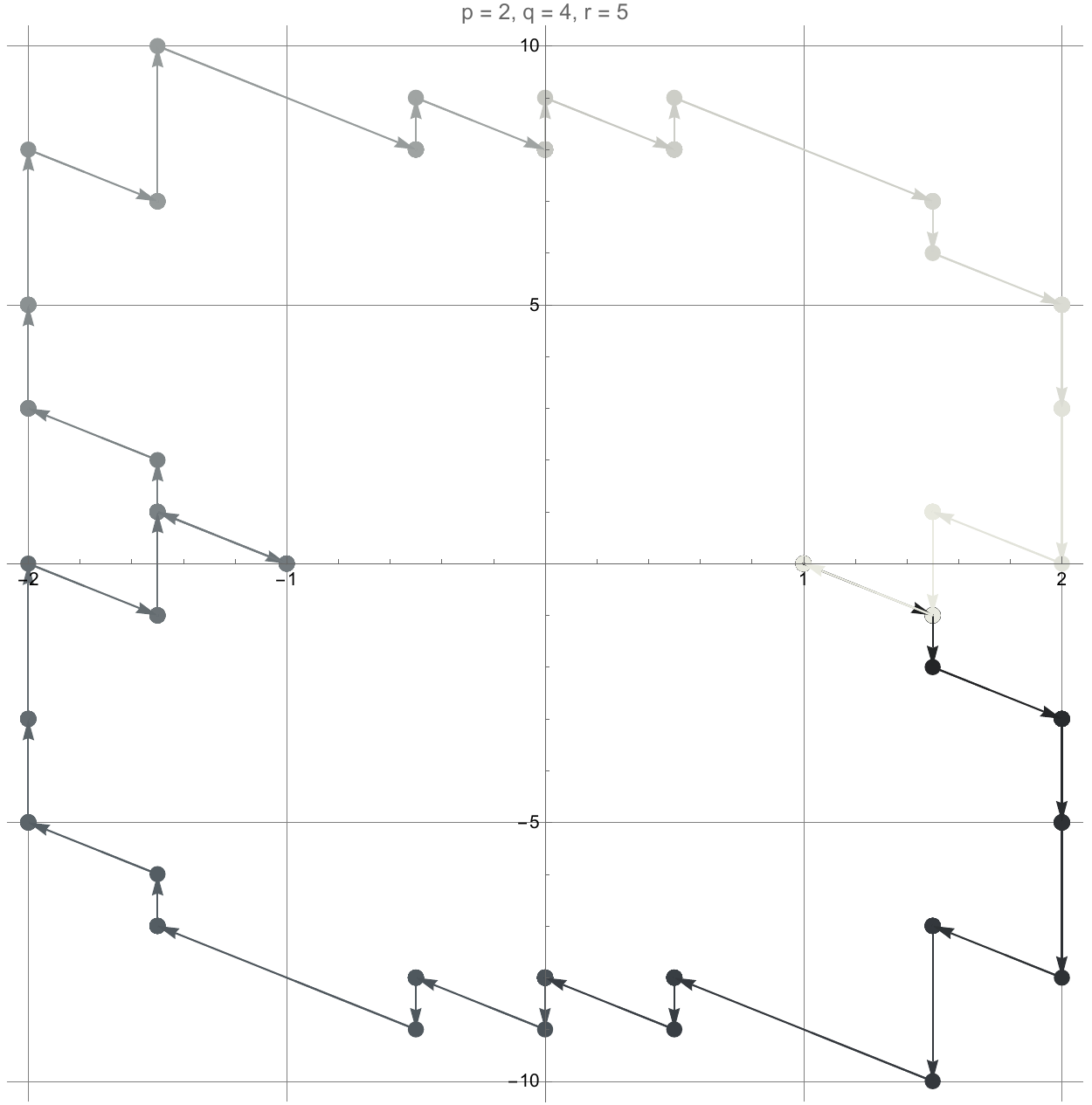}
    \end{subfigure}

    \vspace{0.3cm}

    \begin{subfigure}[b]{0.45\textwidth}
        \includegraphics[width=\textwidth]{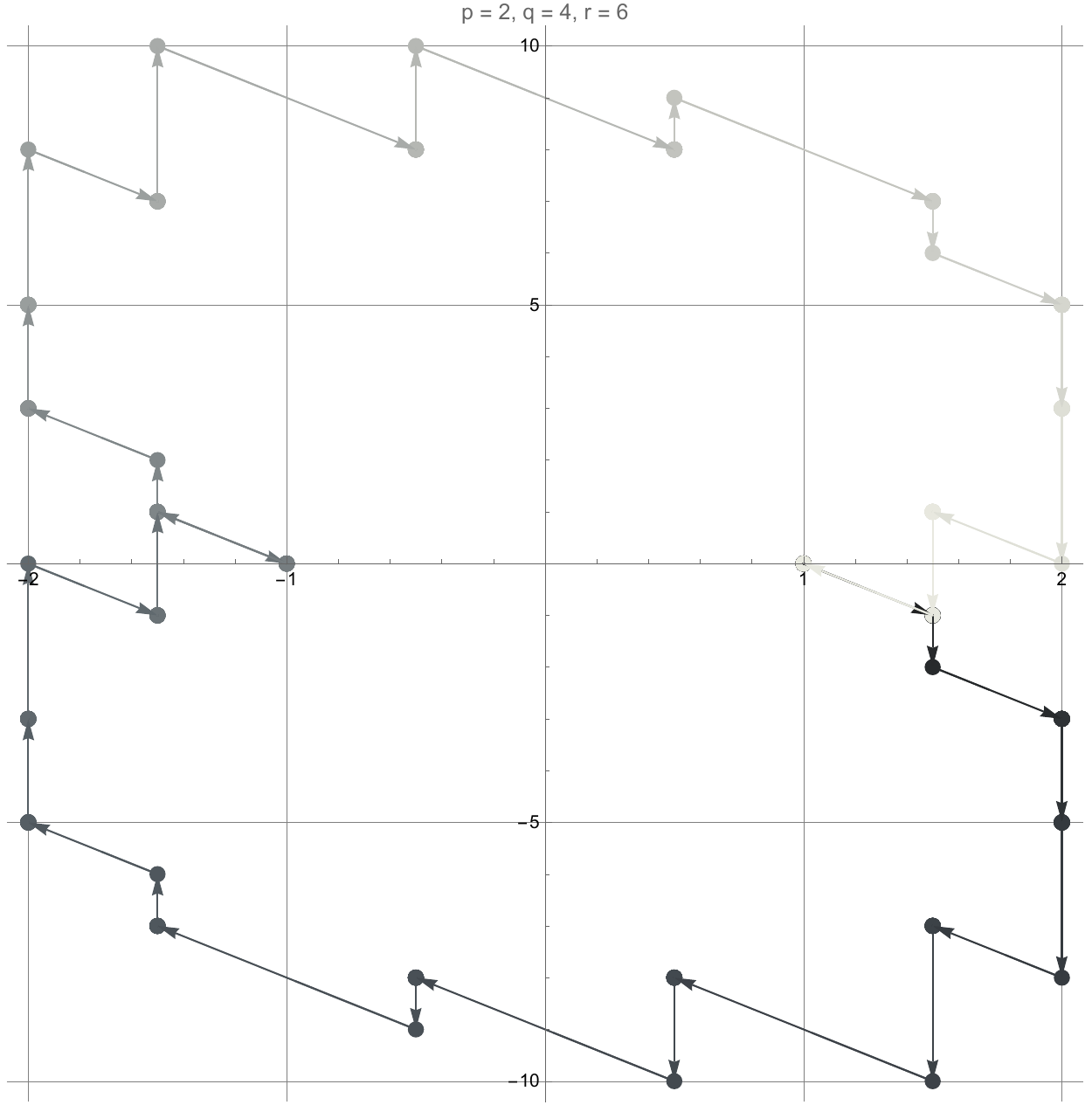}
    \end{subfigure}
    \hfill
    \begin{subfigure}[b]{0.45\textwidth}
        \includegraphics[width=\textwidth]{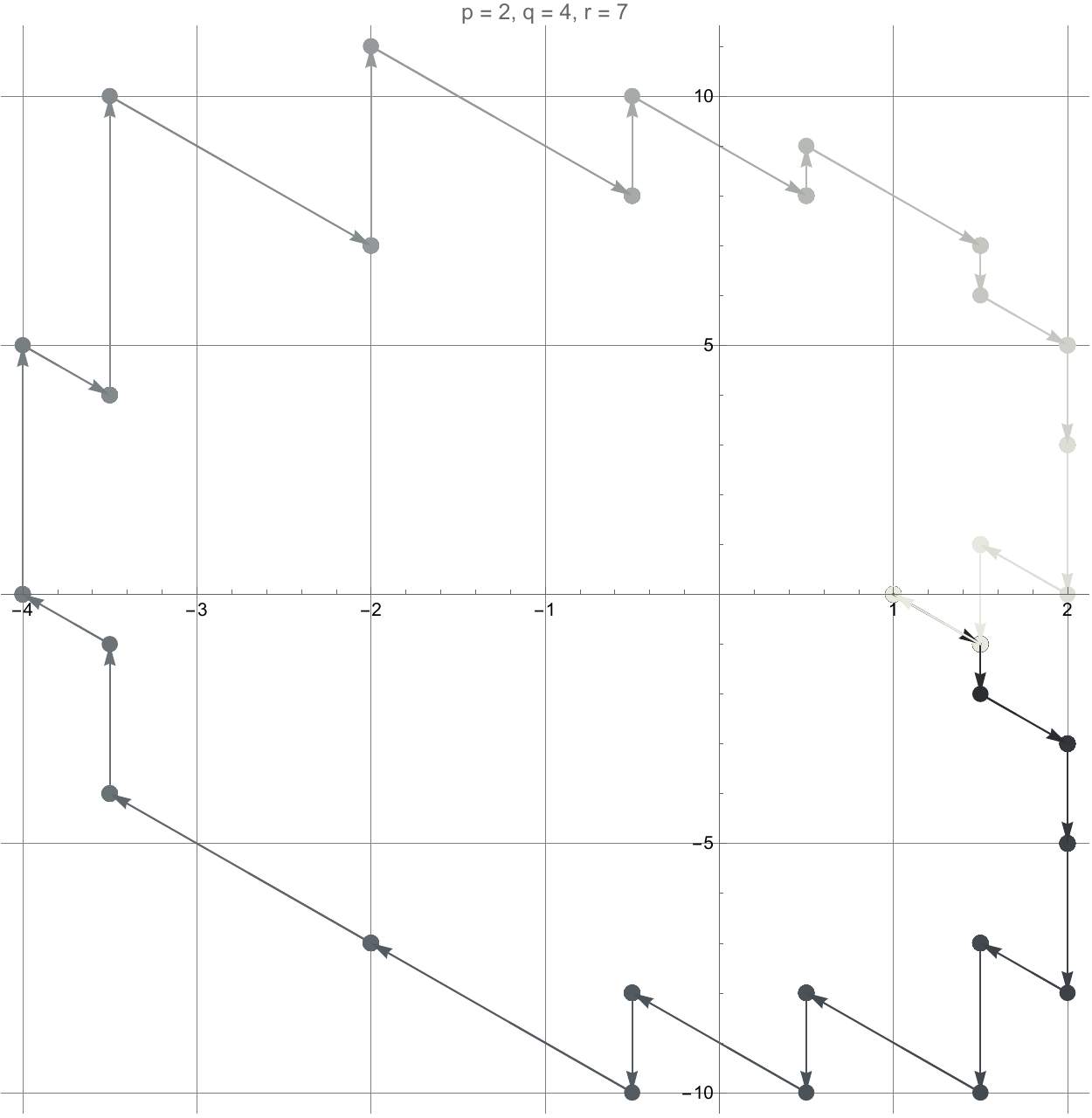}
    \end{subfigure}

    \vspace{0.3cm}

    \begin{subfigure}[b]{0.45\textwidth}
        \includegraphics[width=\textwidth]{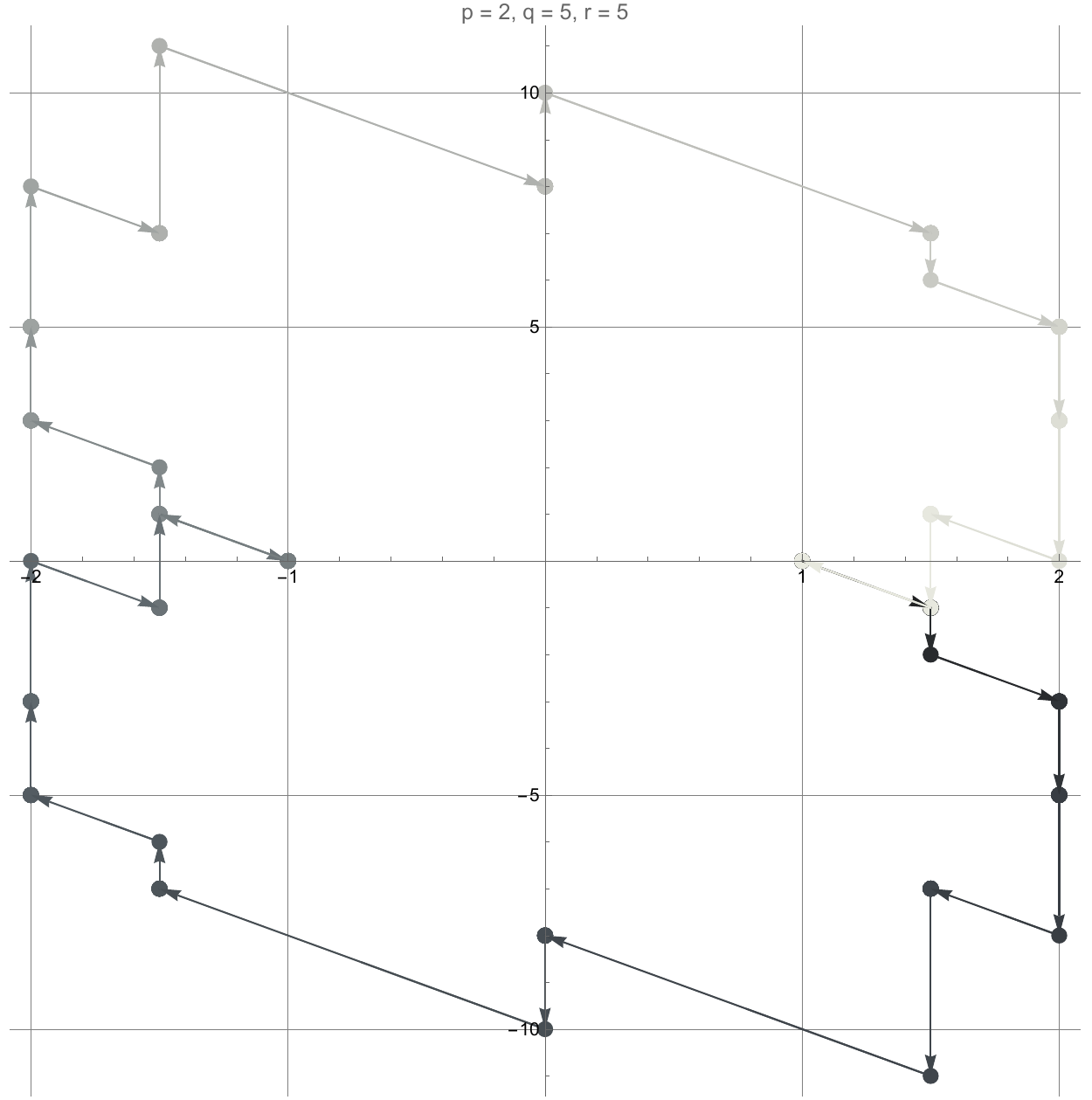}
    \end{subfigure}
    \hfill
    \begin{subfigure}[b]{0.45\textwidth}
        \includegraphics[width=\textwidth]{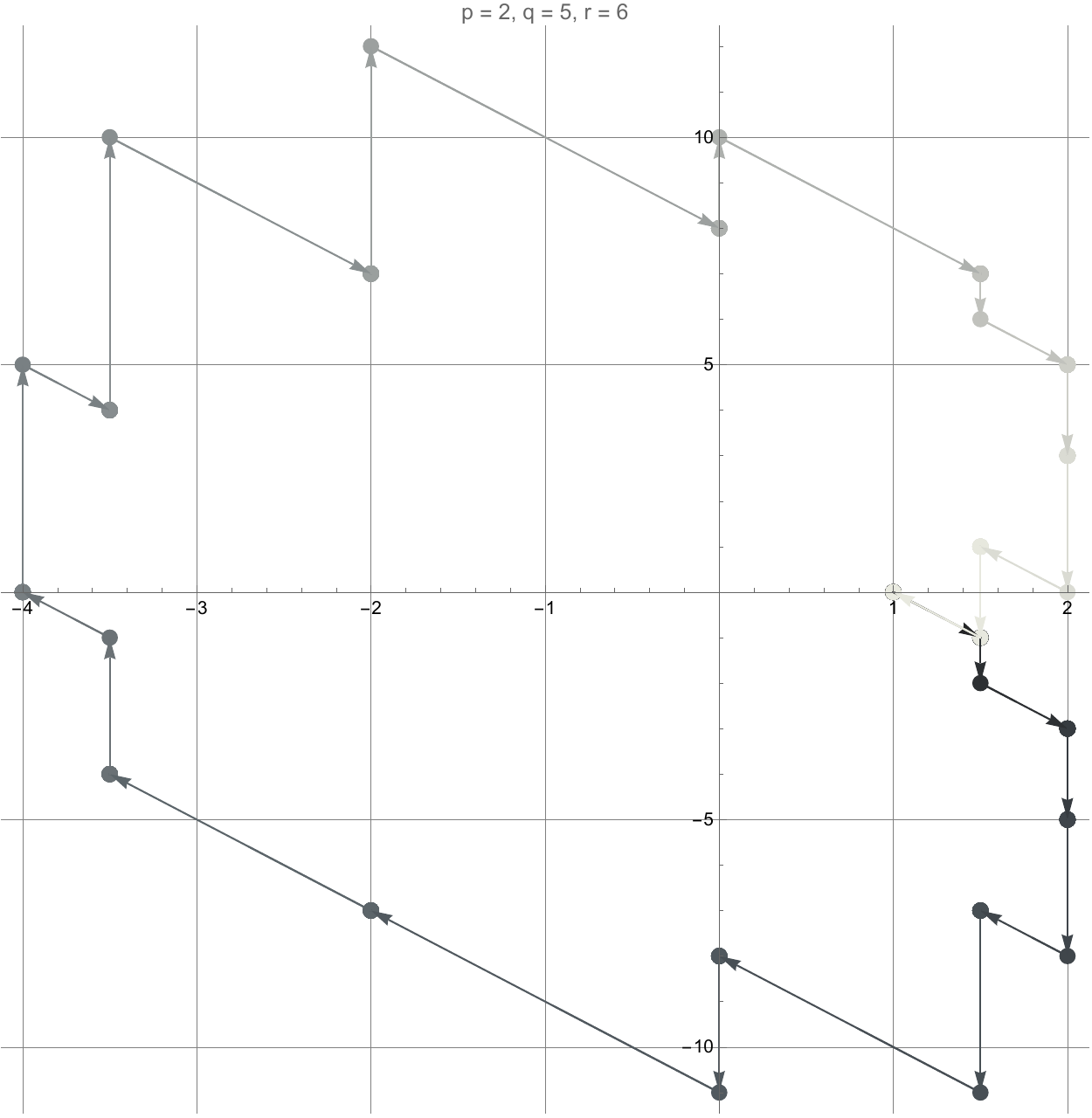}
    \end{subfigure}

\end{figure}

\clearpage

\begin{figure}[htbp]
    \centering

    \begin{subfigure}[b]{0.45\textwidth}
        \includegraphics[width=\textwidth]{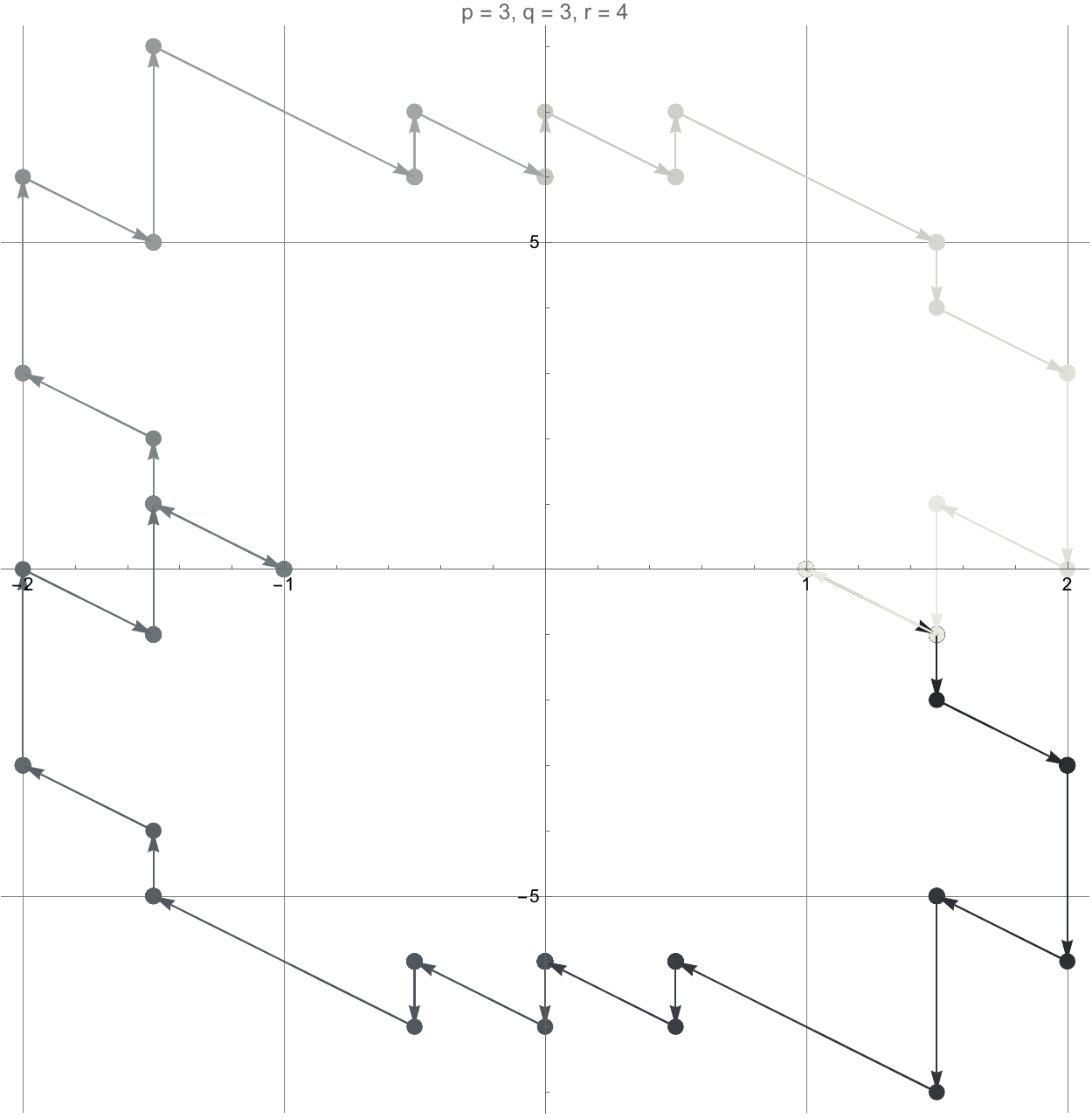}
    \end{subfigure}
    \hfill
    \begin{subfigure}[b]{0.45\textwidth}
        \includegraphics[width=\textwidth]{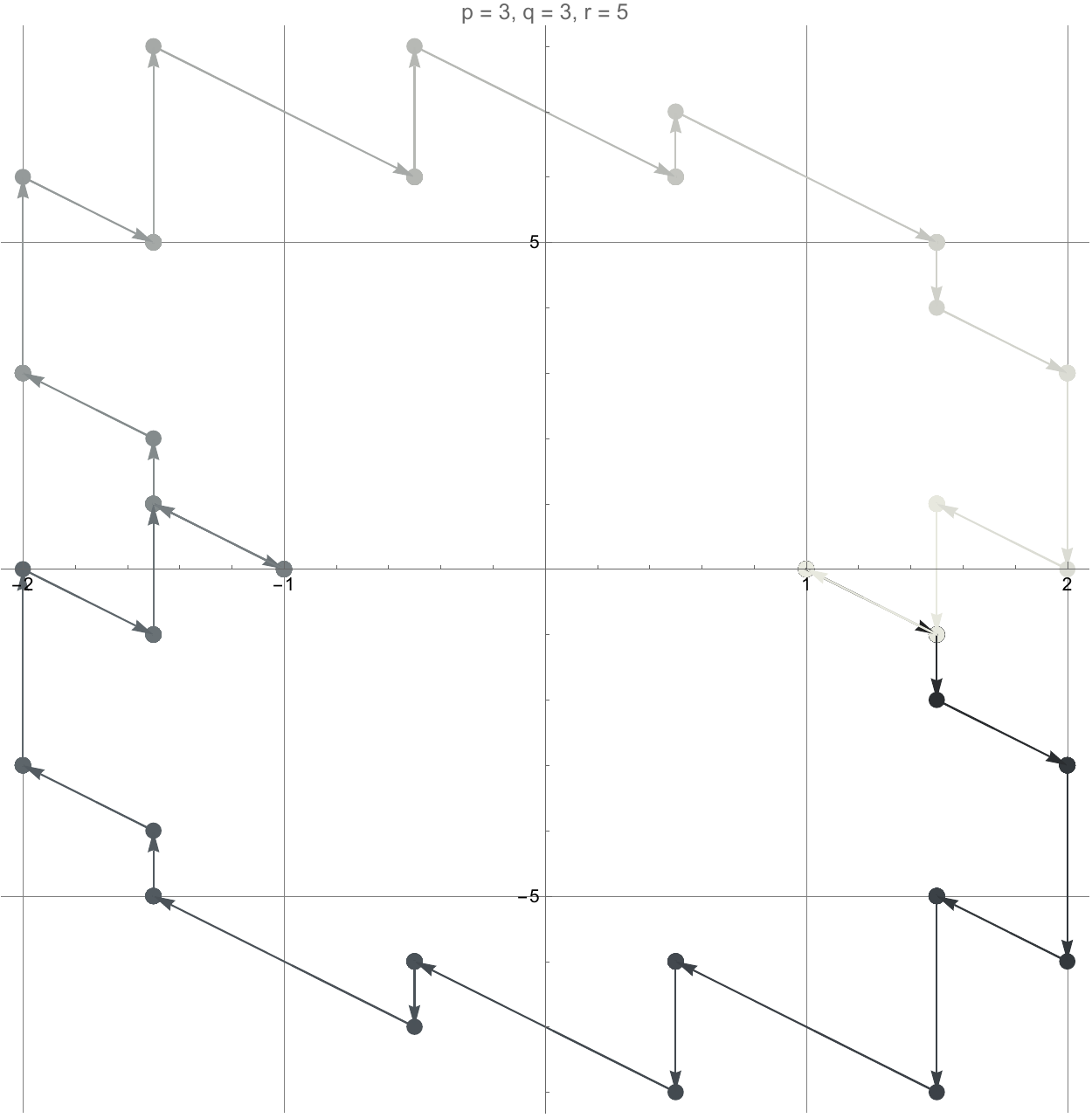}
    \end{subfigure}

    \vspace{0.3cm}

    \begin{subfigure}[b]{0.45\textwidth}
        \includegraphics[width=\textwidth]{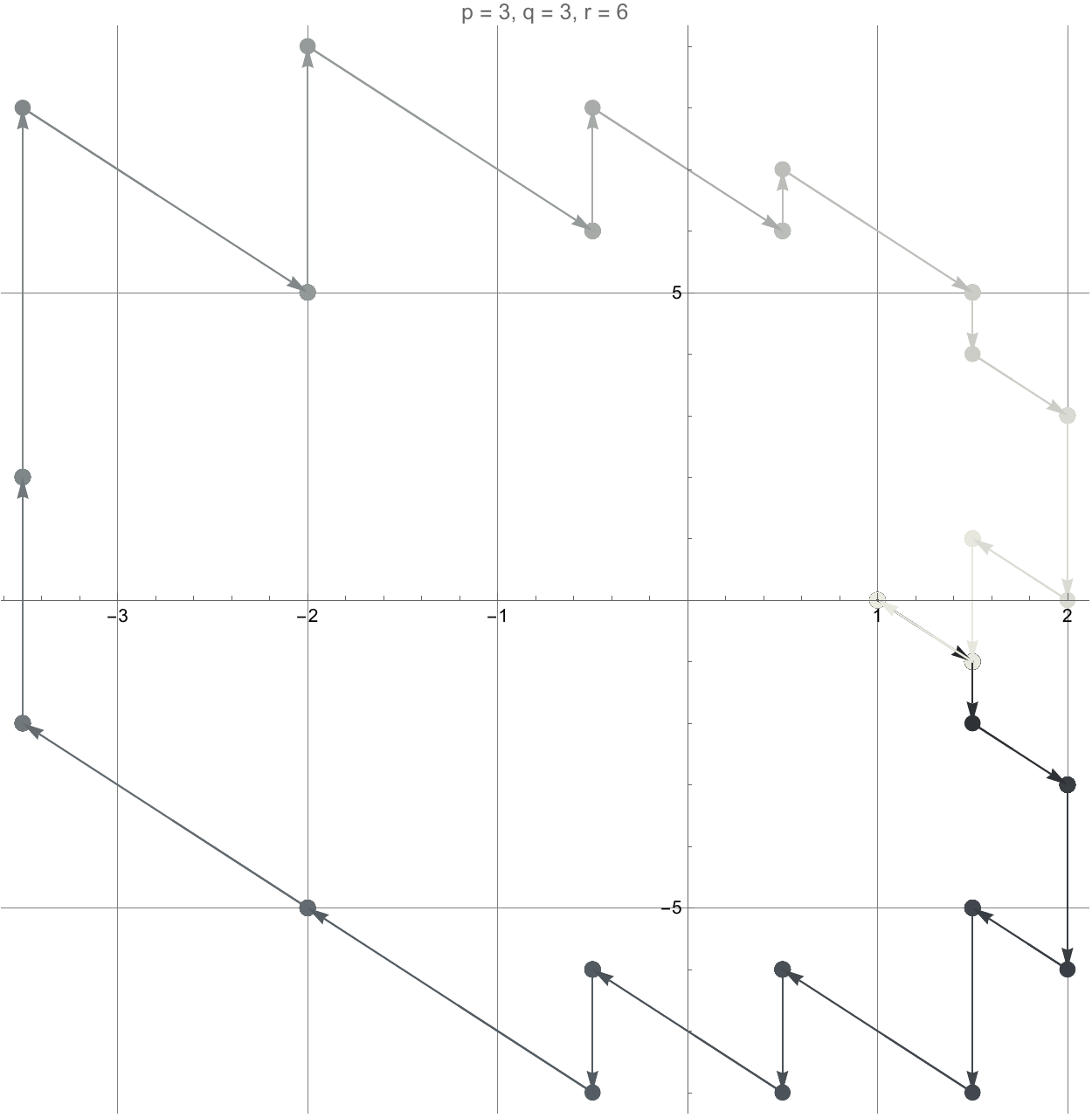}
    \end{subfigure}
    \hfill
    \begin{subfigure}[b]{0.45\textwidth}
        \includegraphics[width=\textwidth]{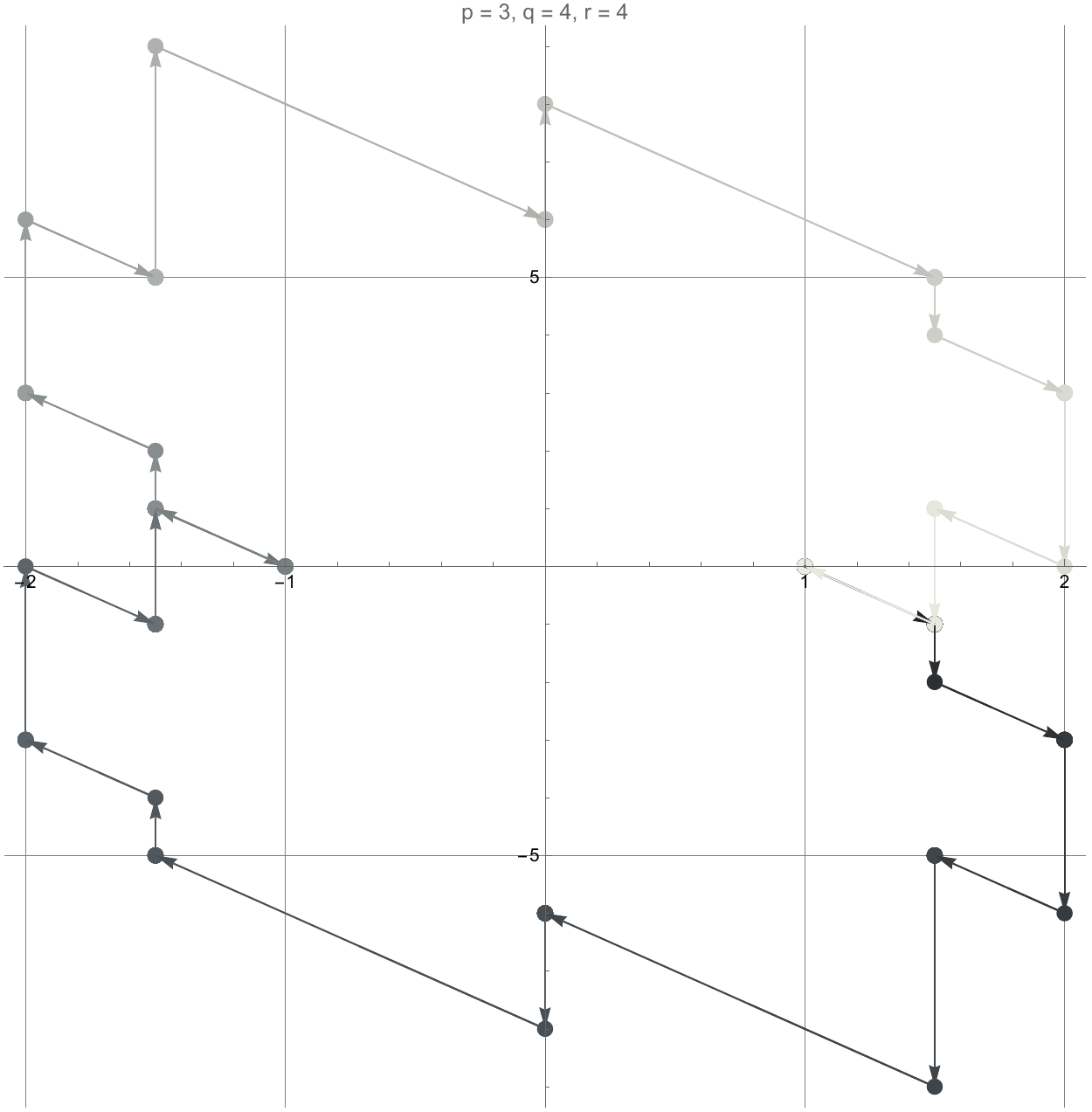}
    \end{subfigure}

    \vspace{0.3cm}

    \begin{subfigure}[b]{0.45\textwidth}
        \includegraphics[width=\textwidth]{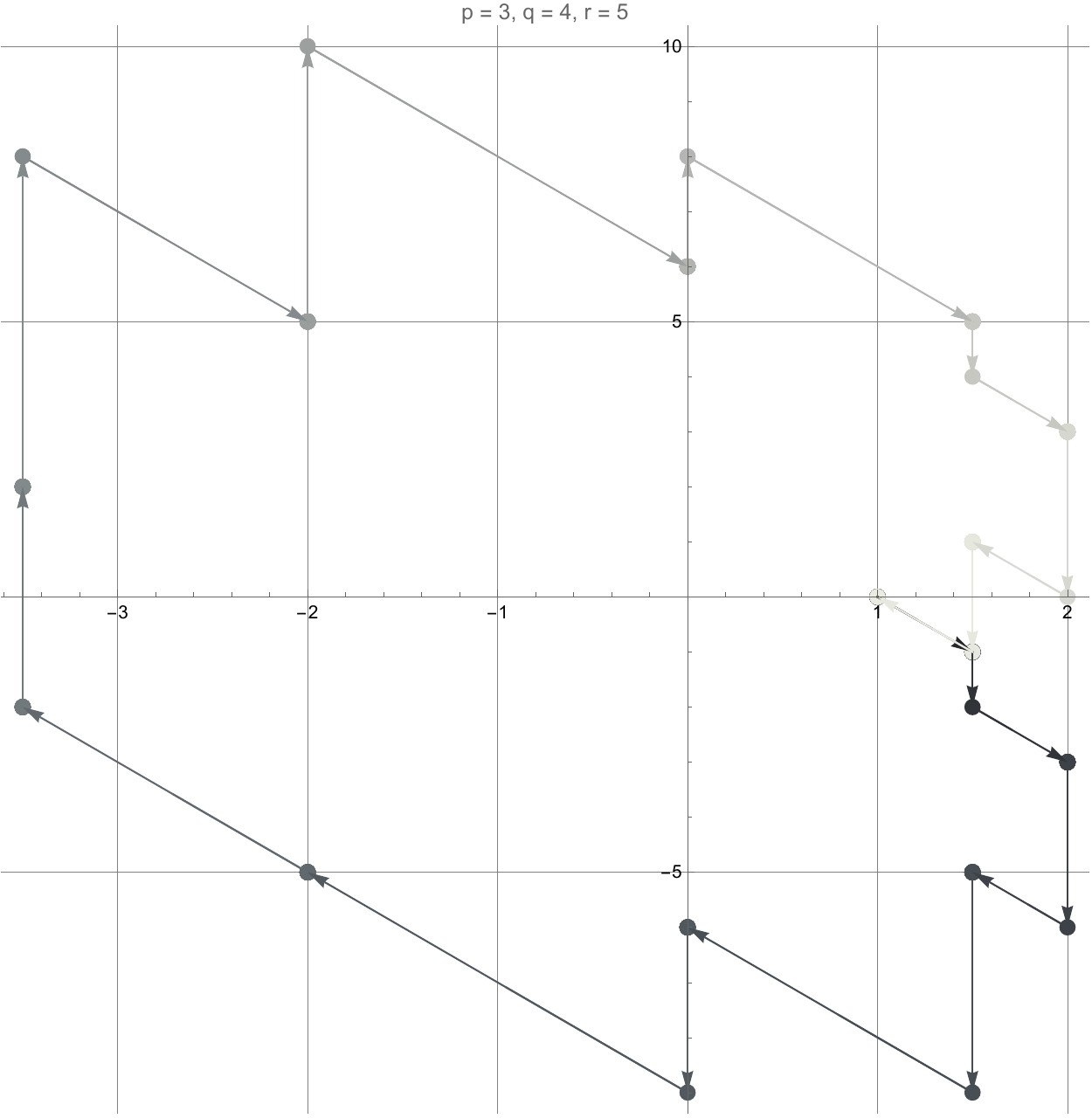}
    \end{subfigure}
    \hfill
    \begin{subfigure}[b]{0.45\textwidth}
        \includegraphics[width=\textwidth]{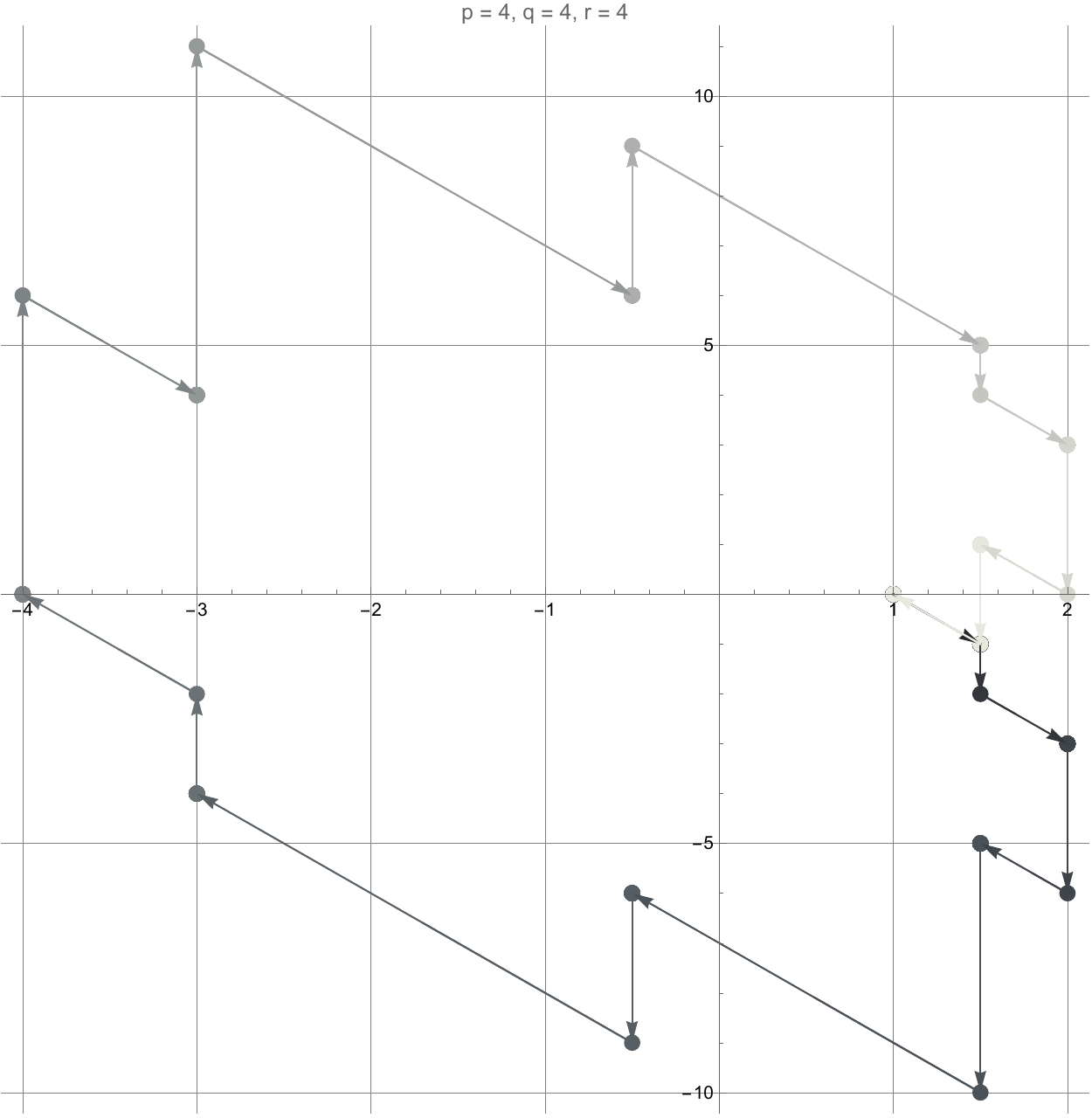}
    \end{subfigure}

\end{figure}


\newpage

\bibliographystyle{alpha}
\bibliography{main.bib}

\end{document}